\documentclass[final,1p,times]{elsarticlePLF} 
\usepackage{a4wide}
\usepackage{geometry}  
\usepackage{amssymb}
\usepackage{amsmath} 
\usepackage{amsthm} 
\usepackage{verbatim}
\newtheorem{theorem}{Theorem}[section]
\newtheorem{lemma}[theorem]{Lemma}
\newtheorem{e-proposition}[theorem]{Proposition}

\newdefinition{e-definition}[theorem]{Definition\ } 
\newdefinition{remark}[theorem]{Remark\ } 
\numberwithin{equation}{section}
\newcommand \al {\alpha}
\newcommand \be {\beta}
\newcommand \de {\delta}
\newcommand \ep {\varepsilon}
\newcommand \ga {\gamma}
\newcommand \Ga {\Gamma}
\newcommand \La {\boldsymbol\Lambda}
\newcommand \om {\boldsymbol\omega}
\newcommand \ph {\varphi}
\newcommand \si {\sigma}
\newcommand \Si {\Sigma}
\newcommand \bs {\boldsymbol}
\newcommand \Abs {\bs{A}}
\newcommand \fbs {\bs{f}}
\newcommand \Fbs {\bs{F}}
\newcommand \hbs {\bs{h}}
\newcommand \Hbs {\bs{H}}
\newcommand \Tbs {\bs{T}}
\newcommand \Wbs {\bs{W}\!}
\newcommand \Sibs {\bs{\Si}}
\newcommand \Nbb {\mathbb{N}}
\newcommand \Rbb {\mathbb{R}}
\newcommand \Acal {\mathcal{A}}
\newcommand \Bcal {\mathcal{B}}
\newcommand \Ccal {\mathcal{C}}
 
\newcommand \Fcal {\mathcal{F}} 
\newcommand \Lcal {\mathcal{L}}

\newcommand \Scal {\mathcal{S}}
\newcommand \ov {\overline}
\newcommand \nh {\hat}
\newcommand \wh {\widehat}
\newcommand \nt {\tilde}   
\newcommand \wt {\widetilde}
\newcommand \z     [1] {#1_0} 
\newcommand \zz   [1] {#1_{\! 0}} 
\newcommand \bpb [1] {#1^*_{\textup{b}}} 
\newcommand \Flaw {\dot}      
\newcommand \Claw {\ddot}       
\newcommand \Elaw {\dddot}     
\newcommand \Dlaw {\ddddot}  
\newcommand \pa {\partial}
\newcommand \na {\nabla}
\DeclareMathOperator{\divop}{div}

\DeclareMathOperator{\id}{id}
\DeclareMathOperator{\sym}{sym}
\DeclareMathOperator{\lin}{lin}
\DeclareMathOperator{\aff}{aff}
\DeclareMathOperator{\qua}{qua}
\DeclareMathOperator{\sse}{ss} 
\DeclareMathOperator{\svk}{svk} 
\DeclareMathOperator{\hke}{Hooke}

\journal{Journal de Math\'ematiques Pures et Appliqu\'ees}

\begin{document}

%==================================================================

\begin{frontmatter}

\title{The equations of elastostatics in a Riemannian manifold} 

\author[ljll]{Nastasia Grubic}
\ead{grubic@ann.jussieu.fr}

\author[ljll]{Philippe G. LeFloch}
\ead{contact@philippelefloch.org}

\author[ljll]{Cristinel Mardare\corref{cor}}
\ead{mardare@ann.jussieu.fr}

\cortext[cor]{Corresponding author}

\address[ljll]{Universit\'e Pierre et Marie Curie \& Centre National de la Recherche Scientifique,  \\
Laboratoire Jacques-Louis Lions, 4 Place Jussieu, 75005 Paris, France.
\\
Completed in December 2013. Revised in March 2014.}

\begin{abstract} 
To begin with, we identify the equations of elastostatics in a Riemannian manifold, which generalize those of classical elasticity in the three-dimensional Euclidean space. Our approach relies on the principle of least energy, which asserts that the deformation of the elastic body arising in response to given loads minimizes over a specific set of admissible deformations the total energy of the elastic body, defined as the difference between the strain energy and the potential of the loads. Assuming that the strain energy is a function of the metric tensor field induced by the deformation, we first derive the principle of virtual work and the associated nonlinear boundary value problem of nonlinear elasticity from the expression of the total energy of the elastic body. We then show that this boundary value problem possesses a solution if the loads are sufficiently small (in a sense we specify).  
\vskip 0.5\baselineskip \noindent
{\bf R\'esum\'e} \vskip 0.5\baselineskip \noindent
Dans un premier temps, nous identifions les \'equations de l'\'elastostatique dans une vari\'et\'e riemannienne, qui g\'e\-n\'e\-ralisent celles de la th\'eorie classique de l'\'elasticit\'e dans l'espace euclidien tridimensionnel. Notre approche repose sur le principe de moindre action, qui affirme que la d\'eformation du corps \'elastique sous l'action des forces externes minimise sur l'en\-semble des d\'eformations admissibles l'\'energie totale du corps \'elastique, d\'efinie comme la diff\'erence entre l'\'energie de d\'eformation et le potentiel des forces externes. Sous l'hypoth\`ese que l'\'en\'ergie de d\'eformation est une fonction du champ de tenseurs m\'etriques induit par la d\'eformation, nous d\'erivons dans un premier temps le principe des travaux virtuels et le probl\`eme aux limites associ\'e  \`a partir de l'expression de l'\'energie totale du corps \'elastique. Nous d\'emontrons ensuite que ce probl\`eme aux limites admet une solution si les forces externes sont suffisamment petites (en un sens que nous pr\'ecisons). 
\end{abstract}


\end{frontmatter}

%==================================================================

\section{Introduction} 
\label{sect1}

In this paper we study the deformation of an elastic body immersed in a Riemannian manifold in response to applied body and surface forces. We first show how the equations of elastostatics can be derived from the principle of least energy, and we then establish existence theorems for these equations. These equations generalize the classical equations of elastostatics in the three-dimensional Euclidean space, and have applications in both classical and relativistic elasticity theory; cf.~Section \ref{sect10}. The definitions and notations used, but not defined in this introduction, can be found in Section \ref{sect2}.

Alternative approaches to the modeling of elastic bodies in a Riemannian manifold can be found elsewhere in the literature; see, for instance, \cite{efrati, epstein, marsden, segev1, segev2, segev3, segev2suite} and the references therein. Our approach is akin to the one in Ciarlet \cite{ciarlet}, but is formulated in a Riemannian manifold instead of the three-dimensional Euclidean space. As such, our results can be easily compared with their counterparts in classical elasticity and in this respect can be used to model the deformations of thin elastic shells whose middle surface must stay inside a given surface in the three-dimensional Euclidean space. More specifically, letting $(N,{\nh g})$ be the three-dimensional Euclidean space and ${\z\ph}:M\to {\nh M}\subset N$ be a global local chart (under the assumption that it exists) of the reference configuration ${\nh M}:={\z\ph}(M)$ of an elastic body immersed in $N$ reduces our approach to the three-dimensional classical elasticity in curvilinear coordinates (see, for instance, \cite{ciarletG}), while letting $M={\nh M}\subset N$ and ${\z\ph}=\id_{\nh M}$ reduces our approach to the classical three-dimensional elasticity in Cartesian coordinates (see, for instance, \cite{ciarlet}). 

An outline of the paper is as follows. Section \ref{sect2} describes the mathematical framework and notation used throughout the paper. Basic notions from differential and Riemannian geometry are  briefly discussed. It is important to keep in mind that in all that follows, the physical space containing the elastic body under consideration is a differential manifold $N$ endowed with a single metric tensor ${\nh g}$, while the abstract configuration of the elastic body (by definition, a manifold whose points label the material points of the elastic body) is a differential manifold $M$ endowed with two metric tensors, one ${g}={g}[\ph]:=\ph^*{\nh g}$ induced by an unknown deformation $\ph:M\to N$, and one ${\z g}={g}[\z\ph]:=\z\ph^*{\nh g}$ induced by a reference deformation $\z\ph:M\to N$. The connection and volume on $N$ are denoted ${\nh\na}$ and ${\nh\om}$ (induced by ${\nh g}$), respectively. The connections and volume forms on $M$ are denoted ${\na}={\na}[\ph]$ and ${\om}={\om}[\ph]$ (induced by ${g}={g}[\ph]$) and  ${\zz\na}$ and ${\z\om}$ (induced by ${\z g}$). 

Tensor fields on $M$ will be denoted by plain letters, such as ${\xi}$, and their components in a local chart will be denoted with Latin indices, such as ${\xi}^i$. 
Tensor fields on $N$ will be denoted by letters with a hat, such as ${\nh\xi}$, and their components in a local chart will be denoted with Greek indices, such as $\nh\xi^\al$. 
Tensor fields on $M\times N$ will be denoted by letters with a tilde, such as ${\nt\xi}$ or ${\nt T}$, and their components in local charts will be denoted with Greek and Latin indices, such as $\nt\xi^\al$ or $\nt T^i_\al$. 

Functionals defined over an infinite-dimensional manifold, such as $\Ccal^1(M,N)$ or 
$$
\Ccal^1(TM):=\{{\xi}:M\to TM; \ {\xi}(x)\in T_xM \text{ for all } x\in M\},
$$ 
where $T_xM$ denotes the tangent space to $M$ at $x\in M$, will be denoted with letters with a bracket, such as $T[\ ]$. Functions defined over a finite-dimensional manifold, such as $M\times N$ or $T^p_qM$, will be denoted with letters with a parenthesis, such as ${\Flaw T}(\ )$. Using the same letter in $T[\ ]$ and ${\Flaw T}(\ )$ means that the two functions are related, typically (but not always) by 
$$
(T[\ph])(x)={\Flaw T}(x, \ph(x), D\ph(x)) \text{ \ for all } x\in M, 
$$
where $D\ph(x)$ denotes the differential of $\ph$ at $x$. 
In this case, the function ${\Flaw T}(\ )$ is called the constitutive law of the function ${T}[\ ]$
and the above relation is called the constitutive equation of ${T}$. 
Letters with several dots denote constitutive laws of different kind, for instance, 
$$
(T[\ph])(x)={\Flaw T}(x, \ph(x), D\ph(x))={\Claw T}(x, g[\ph](x)) ={\Elaw T}(x, E[\z\ph,\ph](x))={\Dlaw T}(x, {\xi}(x), {\zz\na}{\xi}(x)) 
$$
for all $x\in M$, where 
$$
E[\z\ph,\ph]:=\frac12({g}[\ph]-{\z g}) \text{  \  and \ } {\xi}:=\exp_{\z\ph}^{-1}\ph
$$ 
(the mapping $\exp_{\z\ph}$ is defined below). The derivative of a function $f[\ ]$ at a point $\ph$ in the direction of a tangent vector ${\eta}$ at $\ph$ will be denoted $f'[\ph]\eta$.

In Section \ref{sect3}, we define the kinematic notions used to describe the deformation of an elastic body. The main novelty is the relation 
$$
\ph=\exp_{\z\ph}{\xi}:=({\wh\exp}({\ph_0}_*\xi))\circ\z\ph
$$
between a displacement field ${\xi}\in \Ccal^1(TM)$ of a reference configuration $\z\ph(M)$ of the body and the corresponding deformation $\ph:M\to N$ of the same body. Of course, this relation only holds if the vector field ${\xi}$ is small enough, so that the exponential maps of $N$ be well defined at each point $\z\ph(x)\in N$, $x\in M$. We will see in the next sections that the exponential maps on the Riemannian manifold $N$ replace, to some extent, the vector space structure of the three-dimensional  Euclidean space appearing in classical elasticity. The most important notions defined in this section are the metric tensor field, also called the right Cauchy-Green tensor field, 
$$
{g}[\ph]:=\ph^*{\nh g},
$$
induced by a deformation $\ph:M\to N$, the strain tensor field, also called Green-St. Venant tensor field, 
$$
E[\ph,\psi]:=\frac12({g}[\psi]-g[\ph])
$$ 
associated with a reference deformation $\ph$ and a generic deformation $\psi$, and the linearized strain tensor field ($\Lcal$ denotes the Lie derivative operator on $M$; see Section \ref{sect2})
$$
e[\ph,{\xi}]:=\frac12 \Lcal_{{\xi}}(g[\ph]),
$$
associated with a reference deformation $\ph$ and  a displacement field ${\nt\xi}=(\ph_*\xi)\circ\ph$ of the configuration $\ph(M)$. 
 
In Section \ref{sect4}, we express the assumption that the body is made of an elastic material in mathematical terms. The assumption underlying our model is that the strain energy density associated with a deformation $\ph$ of the body is of the form   
$$
(\Wbs[\ph])(x):=\bs{\Elaw W}(x, (E[\z\ph,\ph])(x))\in \La^n_xM, \ x\in M, 
$$
or equivalently, 
$$
\Wbs[\ph]={\z W}[\ph]{\z\om}, \text{ \ where } ({\z W}[\ph])(x):={\z{\Elaw W}}(x, (E[\z\ph,\ph])(x))\in \Rbb, \ x\in M, 
$$
where $\La^n_xM$ denotes the space of $n$-forms on $M$ at $x\in M$, $\z\ph:M\to N$ denotes a reference deformation of the body, and ${\z\om}:=\z\ph^*{\nh\om}$ denotes the volume form on $M$ induced by $\z\ph$. 

The stress tensor field associated with a deformation $\ph$ is then defined in terms of this density by 
$$
\Sibs[\ph]:=\frac{\pa \bs{\Elaw W}}{\pa E}(\cdot, E[\z\ph,\ph]). 
$$
Other equivalent stress tensor fields, denoted $\Tbs[\ph]$, ${\nt\Tbs}[\ph]$, ${\nh\Sibs}[\ph]$, and ${\nh\Tbs}[\ph]$, are defined in terms of $\Sibs[\ph]$ by lowering and pushing forward some of its indices; cf. Remark \ref{Tij}. The novelty are the tensor fields ${\nh\Sibs}[\ph]:=\ph_*\Sibs[\ph]$ and $\Tbs[\ph]:={g}[\ph]\cdot \Sibs[\ph]$, where $\cdot$ denotes the contraction of one single index, which are not needed in classical elasticity because of the particularity of the three-dimensional Euclidean space (which possesses in particular a constant orthonormal frame field). The three other tensor fields ${\nh\Tbs}[\ph]$, ${\nt\Tbs}[\ph]$, and $\Sibs[\ph]$, correspond in classical elasticity to the Cauchy, the first Piola-Kirchhoff, and the second Piola-Kirchhoff, stress tensor fields, respectively. 

Hereafter, boldface letters denote volume forms with scalar or tensor coefficients; the corresponding plain letters denote the components of such volume forms over a fixed volume form with scalar coefficients. For instance, if $\om:=\ph^*{\nh\om}$ denotes the volume on $M$ induced by a deformation $\ph$, then 
$$
\Wbs=W\om, \quad
\Sibs={\Si}\otimes\om, \quad
\Tbs={T}\otimes\om, \quad
{\nt\Tbs}={\nt T}\otimes\om, \quad
{\nh\Tbs}={\nh T}\otimes{\nh\om}, \quad
{\nh\Sibs}={\nh\Si}\otimes{\nh\om}.
$$
In the particular case where the volume form is ${\z\om}:=\z\ph^*{\nh\om}$, where $\z\ph$ defines the reference configuration of the body, we use the notation
$$
\Wbs={\z W}\,{\z\om}, \quad
\Sibs={\z\Si}\otimes{\z\om}, \quad
\Tbs={\z T}\otimes{\z\om}, \quad
{\nt\Tbs}={\z{\nt T}}\otimes{\z\om}.
$$
Incorporating the volume form in the definition of the stress tensor field might seem redundant (only ${\z W}, {\z\Si}, {\z{\nt T}}, {\nh T}$ are defined in classical elasticity), but it has three important advantages: First, it allows to do away with the Piola transform and use instead the more geometric pullback operator. Second, it allows to write the boundary value problem of both nonlinear and linearized elasticity (equations \eqref{intro-e2} and \eqref{intro-e3}, resp. \eqref{intro-e4} and \eqref{intro-e5}, below) in divergence form, by using appropriate volume forms, viz., $\om$ in nonlinear elasticity and ${\z\om}$ in linearized elasticity, so that ${\na}\om=0$ and ${\zz\na}{\z\om}=0$. Third, the normal trace of $\Tbs=\Tbs[\ph]$ on the boundary of $M$ appearing in the boundary value problem \eqref{intro-e2} is independent of the choice of the metric used to define the unit outer normal vector field to $\pa M$, by contrast with the normal trace of ${T}={T}[\ph]$ on the same boundary appearing in the boundary value problem \eqref{intro-e3}; see relation \eqref{Tn} and the subsequent comments.
 
Section \ref{sect5} is concerned with the modeling of external forces. The main assumption is that the densities of the applied body and surface forces are of the form 
$$
\begin{aligned}
({\fbs}[\ph])(x)
&:=\bs{\Flaw f}(x,\ph(x), D\ph(x)) 
\in T^*_xM\otimes\La^n_xM, 
&& x\in M,\\
({\hbs}[\ph])(x)
&:=\bs{\Flaw{h}}(x,\ph(x), D\ph(x))
\in T^*_xM\otimes\La^{n-1}_x\Ga_2,
&&  x\in \Ga_2\subset {\pa M},
\end{aligned}
$$
where $\bs{\Flaw f}$ and $\bs{\Flaw{h}}$ are sufficiently regular functions, and $\Ga_1\cup\Ga_2=\Ga:=\pa M$ denotes a measurable partition of the boundary of $M$. 

In Section \ref{sect6}, we combine the results of the previous sections to derive the model of nonlinear elasticity in a Riemannian manifold, first as a minimization problem (Proposition \ref{PAL}), then as variational equations (Proposition \ref{VE}), and finally as a boundary value problem (Proposition \ref{BVP}).  
 The latter asserts that the deformation $\ph$ of the body must satisfy the system 
\begin{equation}
\label{intro-e2}
\aligned
 - \divop\, {\Tbs}[\ph] &={\fbs}[\ph] && \text{in }  \textup{int}M,\\
{\Tbs}[\ph]_{\nu}  &={\hbs}[\ph] && \text{on } {\Ga_2},\\
        \ph &=\z\ph && \text{on } {\Ga_1},
\endaligned
\end{equation} 
or equivalently, the system
\begin{equation}
\label{intro-e3}
\aligned
 - \divop\, {T }[\ph] &={f}[\ph] && \text{in }  \textup{int}M,\\
{T}[\ph]\cdot (\nu[\ph]\cdot g[\ph])  &={h}[\ph] && \text{on } {\Ga_2},\\
        \ph &=\z\ph && \text{on } {\Ga_1},
\endaligned
\end{equation}
where $\divop=\divop[\ph]$ and $\nu[\ph]$ respectively denote the divergence operator and the unit outer normal vector field to the boundary of $M$ induced by the metric $g=g[\ph]$. Note that the divergence operators appearing in these boundary value problems depend themselves on the unknown $\ph$. 
 
In Section \ref{sect7}, we deduce the equations of linearized elasticity from those of nonlinear elasticity, by linearizing the stress tensor field ${\Sibs}[\ph]$ with respect to the displacement field ${\xi}:=\exp_{{\z\ph}}^{-1}\ph$ of the reference configuration ${\z\ph}(M)$ of the body, assumed to be a natural state (that is, an unconstrained configuration of the body). Thus the unknown in linearized elasticity is the displacement field ${\xi}\in \Ccal^1(TM)$, instead of the deformation $\ph\in\Ccal^1(M,N)$ in nonlinear elasticity. 

The elasticity tensor field of an elastic material, whose (nonlinear) constitutive law is 
$\bs{\Elaw W}$, is defined at each $x\in M$ by 
$$
\Abs(x):= \frac{\pa^2 \bs{\Elaw W}}{\pa E^2}(x,0).  
$$
The linearized stress tensor field associated with a displacement field ${\xi}$ is then defined by 
$$
\Tbs^{\lin}[\xi]:=(\Abs:e[\z\ph,{\xi}])\cdot{\z g},
$$
where ${\z g}={\z\ph^*}{\nh g}$ and $:$ denotes the contraction of two indices (the last two contravariant indices of $\Abs$ with the two covariant indices of $e[\z\ph,{\xi}]$). The affine part with respect to ${\xi}$ of the densities of the applied forces are defined by 
$$
{\fbs}^{\aff}[{\xi}]:={\fbs}[\z\ph]+{\fbs}'[\z\ph]{\xi}
\text{ \ and \ }
{\hbs}^{\aff}[{\xi}]:={\hbs}[\z\ph]+{\hbs}'[\z\ph]{\xi},
$$
where 
${\fbs}'[\z\ph]{\xi}:=\left[\frac{d}{dt}{\fbs}[\exp_{\z\ph}(t{\xi})]\right]_{t=0}=\bs{f_1}\cdot {\xi}+\bs{f_2} : {\zz\na}{\xi}$ for some appropriate tensor fields $\bs{f_1}\in \Ccal^0(T^0_2M\otimes\La^nM)$ and $\bs{f_2}\in \Ccal^0(T^1_2M\otimes\La^nM)$ (a similar relation holds for ${\hbs}'[\z\ph]{\xi}$).

It is then shown that, in linearized elasticity, the unknown displacement field of the reference configuration $\z\ph(M)$ is the vector field ${\nt\xi}=({\z\ph}_*{\xi})\circ{\z\ph}$, where ${\xi}\in\Ccal^1(TM)$ satisfies the boundary value problem 
\begin{equation}
\label{intro-e4}
\aligned
 - {\z\divop}\, {\Tbs^{\lin}}[{\xi}]&={\fbs}^{\aff}[{\xi}] && \text{in }  \textup{int}M,\\
{\Tbs^{\lin}}[{\xi}]_{\z\nu}  &={\hbs}^{\aff}[{\xi}] && \text{on } {\Ga_2},\\
        {\xi} &=0 && \text{on } {\Ga_1},
\endaligned
\end{equation}
or equivalently, the boundary value problem
\begin{equation}
\label{intro-e5}
\aligned
 - {\z\divop}\, {\z{T}^{\lin}}[{\xi}]&={\z{f^{\aff}}}[{\xi}] && \text{in }  \textup{int}M,\\
{\z{T}^{\lin}}[{\xi}]\cdot ({\z\nu}\cdot {\z g})  &={\z{h^{\aff}}}[{\xi}]  && \text{on } {\Ga_2},\\
        {\xi} &=0 && \text{on } {\Ga_1},
\endaligned
\end{equation}
where ${\z\divop}$ and ${\z\nu}$ respectively denote the divergence operator and 
the unit outer normal vector field to the boundary of $M$ induced by the metric ${\z g}$; cf. Proposition \ref{EE-lin}.  It is also shown that these boundary value problems are equivalent to the variational equations 
\begin{equation}
\label{intro-e6}
\int_M (\Abs:e[\z\ph,{\xi}]):e[\z\ph,{\eta}] 
=\int_M{\fbs}^{\aff}[{\xi}]\cdot{\eta}
+\int_{{\Ga_2}}{\hbs}^{\aff}[{\xi}]\cdot{\eta},
\end{equation}
for all sufficiently regular vector fields ${\eta}$ that vanish on $\Ga_1$. 

In Section \ref{sect8}, we establish an existence and regularity theorem for the equations of linearized elasticity in a Riemannian manifold (eqns \eqref{intro-e4}-\eqref{intro-e6}). 
We show that the variational equations \eqref{intro-e6} have a unique solution in the Sobolev space $\{{\xi}\in H^1(TM); \ {\xi}=0 \text{ on } \Ga_1\}$ provided the elasticity tensor field $\Abs$ is uniformly positive-definite and ${\fbs}'[\z\ph]$ and ${\hbs}'[\z\ph]$ are sufficiently small in an appropriate norm. 
The key to this existence result is a Riemannian version of Korn's inequality, due to \cite{chen}, asserting that, if $\Ga_1\neq \emptyset$, there exists a constant $C_K<\infty$ such that ($\Lcal$ denotes the Lie derivative operator on $M$; see Section \ref{sect2})
$$
\|{\xi}\|_{H^1(TM)}\leq C_K \|e[\z\ph,{\xi}]\|_{L^2(S_2M)}, \ e[\z\ph,{\xi}] :=\frac12\Lcal_{{\xi}}{\z g},
$$
for all ${\xi}\in H^1(TM)$ that vanish on $\Ga_1$. The ``smallness assumption'' mentioned above depends on this constant: the smaller $C_K$ is, the larger ${\fbs}'[\z\ph]$ and ${\hbs}'[\z\ph]$ are in the existence result for linearized elasticity. 

Furthermore, when $\Ga_1=\pa M$, we show that the solution to the equations of linearized elasticity belongs to the Sobolev space $W^{m+2,p}(TM)$, $m\geq 0$, $1<p<\infty$, and satisfies the boundary value problems \eqref{intro-e3} and \eqref{intro-e4} if the data ($\pa M$, ${\z\ph}$, ${\fbs}[\z\ph]$, and ${\fbs}'[\z\ph]$) satisfies specific regularity assumptions.

In Section \ref{sect9}, we study the existence of solutions to the equations of nonlinear elasticity \eqref{intro-e2} in the particular case where $\Ga_1=\pa M$ and the applied forces and the constitutive law of the elastic material are sufficiently regular. Under these assumptions, the equations of linearized elasticity define a surjective continuous linear operator  
${\bs{\Acal}}^{\lin}[\xi]:={\z\divop}\, {{\Tbs}^{\lin}}[{\xi}]+{\fbs}'[\z\ph]{\xi}: X\to {\bs{Y}}$, where 
$$
X:=W^{m+2,p}(TM)\cap W^{1,p}_0(TM) 
\text{ and } 
{\bs{Y}}:=W^{m,p}(T^*M\otimes \La^nM), 
$$
for some exponents $m\in\Nbb$ and $1<p<\infty$ that satisfy the constraint $(m+1)p>n$, where $n$ denotes the dimension of the manifold $M$. 

Using the substitution $\ph=\exp_{\z\ph}{\xi}$ (when ${\xi}$ is small enough in the $\Ccal^0(TM)$-norm, so that the mapping $\exp_{\z\ph}:\Ccal^1(TM)\to \Ccal^1(M,N)$ is well-defined), we recast the equations of nonlinear elasticity \eqref{intro-e3}  into an equivalent boundary value problem, viz., 
$$
\aligned
 - \divop\, {\Tbs}[\exp_{\z\ph}{\xi}] 
 &={\fbs}[\exp_{\z\ph}{\xi}] 
 && \text{in }  \textup{int}M,\\
  {\xi} 
  &=0 
  && \text{on } {\pa M},
\endaligned
$$
whose unknown is the displacement field ${\xi}$. We then show that the mapping ${\bs{\Acal}}:X\to {\bs{Y}}$ defined by 
$$
{\bs{\Acal}}[{\xi}]
:=\divop\, {\Tbs}[\exp_{\z\ph}{\xi}] 
+ {\fbs}[\exp_{\z\ph}{\xi}] 
\text{ \ for all } {\xi}\in X,
$$
satisfies ${\bs{\Acal}}'[0]={{\bs{\Acal}}^{\lin}}$. Consequently, proving an existence theorem for the equations of nonlinear elasticity amounts to proving the existence of a zero of the mapping ${\bs{\Acal}}$. This is done by using a variant of Newton's method, where a zero of ${\bs{\Acal}}$ is found as the limit of the sequence 
$$
{\xi}_1:=0 \text{ and } {\xi}_{k+1}:={\xi}_k-{\bs{\Acal}}'[0]^{-1}{\bs{\Acal}}[{\xi}_k], \ k\geq 1. 
$$

Note that the constraint $(m+1)p>n$ ensures that the Sobolev space $W^{m+1,p}(T^1_1M)$, to which ${\zz\na}{\xi}$ belongs, is an algebra. This assumption is crucial in proving that the mapping ${\bs{\Acal}}:X\to {\bs{Y}}$ is differentiable, since
$$
({{\bs{\Acal}}}[{\xi}])(x)={\Dlaw{{\bs{\Acal}}}}(x,{\xi}(x), {\zz\na}{\xi}(x)), \ x\in M,
$$
for some regular enough mapping ${\Dlaw{{\bs{\Acal}}}}$, defined in terms of the constitutive laws of the elastic material and of the applied forces under consideration; cf. relations \eqref{Acal2} and \eqref{cl+}. Thus ${\bs{\Acal}}$ is a nonlinear Nemytskii (or substitution) operator, which is known to be non-differentiable if ${\xi}$ belongs to a space with little regularity. 

In addition to making regularity assumptions, we must assume that ${\fbs}'[\z\ph]$ is sufficiently small in an appropriate norm, so that the operator ${\bs{\Acal}}'[0]\in\Lcal(X,{\bs{Y}})$ is invertible; cf. Theorem \ref{exist-lin}, which establishes the existence and regularity for linearized elasticity. 

Finally, we point out that the assumptions of the existence theorem of Section \ref{sect9} are slightly weaker than those usually made in classical elasticity, where either $p>n$ is imposed instead of $(m+1)p>n$ (cf. \cite{ciarlet}), or ${\Dlaw{\fbs}}$ is assumed to belong to the smaller space $\Ccal^{m+1}(M\times TM\times T^1_1M)$ (cf. \cite{valent}).

%==================================================================

\section{Preliminaries} 
\label{sect2}

More details about the definitions below can be found in, for instance, \cite{AMR} and \cite{aubin}. 

Throughout this paper, $N$ denotes an oriented, smooth differentiable manifold of dimension $n$, endowed with a smooth Riemannian metric ${\nh g}$, while $M$ denotes either a compact, oriented, smooth differentiable manifold of dimension $n$, or $M:=\ov\Omega\subset \nt M$, where $\nt M$ is a smooth oriented differentiable manifold of dimension $n$ and $\Omega$ is a bounded, connected, open subset of $\nt M$, whose boundary $\Ga:=\pa M$ is Lipschitz-continuous.  Generic points in $M$ and $N$ are denoted $x$ and $y$, respectively, or $(x^i)_{i=1}^n$ and $(y^\al)_{\al=1}^n$ in local coordinates. To ease notation, the $n$-tuples $(x^i)$ and $(y^\al)$ are also denoted $x$ and $y$, respectively. 

The tangent and cotangent bundles of $M$ are denoted $TM:=\bigsqcup_{x\in M}T_xM$ and $T^*\! M:=\bigsqcup_{x\in M}T^*_xM$, respectively. The bundle of all $(p,q)$-tensors ($p$-contravariant and $q$-covariant) is denoted $T^p_qM:=(\otimes^p TM)\otimes (\otimes^q T^*\! M)$.  Partial contractions of one or two indices between two tensors will be denoted $\cdot$ or $:$ , respectively. 

The bundle of all symmetric $(0,2)$-tensors is denoted 
$$
S_2M :=\bigsqcup_{x\in M}S_{2, x}M\subset T^0_2M,
$$
and the bundle of all positive-definite symmetric $(0,2)$-tensors is denoted by 
$$
S^{+}_2M :=\bigsqcup_{x\in M}S^{+}_{2,x}M \subset S_2M  
$$
Analogously, the bundle of all symmetric $(2,0)$-tensors is denoted by $S^2M :=\bigsqcup_{x\in M}S^2_{ x}M$.

The bundle of all $k$-forms (that is, totally antisymmetric $(0,k)$-tensors) is denoted $\Lambda^kM:=\bigsqcup_{x\in M}\Lambda^k_{x}M$; volume forms on $M$ and on $\Ga$ (that is, nowhere-vanishing sections of $\Lambda^nM$ and of $\Lambda^{n-1}\Gamma$) will be denoted by boldface letters, such as $\om$ and $\bs{i}_{{\nu}}\om$.

Fiber bundles on $M\times N$ will also be used with self-explanatory notation. For instance, 
$$
T^*\! M\otimes TN:=\bigsqcup_{(x,y)\in M\times N}T_x^*\! M\otimes T_yN, 
$$
where $T_x^*\! M\otimes T_yN$ is canonically identified with the space $\Lcal(T_xM,T_yN)$ of all linear mappings from $T_xM$ to $T_yN$. 

The set of all mappings $\ph:M\to N$ of class $\Ccal^k$ is denoted $\Ccal^k(M,N)$. Given any mapping $\ph\in\Ccal^0(M,N)$, the pullback bundle of $T^p_qN$ by $\ph$ is denoted and defined by 
$$
\ph^*T^p_qN:=\bigsqcup_{x\in M} T^p_{q,\ph(x)}N.
$$

The pushforward and pullback mappings induced by a mapping $\ph\in\Ccal^1(M,N)$ are denoted $\ph_*:T^p_0M\to T^p_0N$ and $\ph^*:T^0_qM\to T^0_qN$, respectively. For instance, if $p=1$ and $q=2$, then 
$$
(\ph_*{\xi})^\al(\ph(x)):=\frac{\pa \ph^\al}{\pa x^i} (x){\xi}^i(x) 
 \text{ \ and \ } 
(\ph^*{\nh g})_{ij}(x):=\frac{\pa \ph^\al}{\pa x^i}(x)\frac{\pa \ph^\be}{\pa x^j}(x){\nh g}_{\al\be}(\ph(x)), 
\quad x\in M,
$$ 
where the functions $y^\al=\ph^\al(x^i)$ describe the mapping $\ph$ in local coordinates, denoted $(x^i)$ on $M$ and  $(y^\al)$ on $N$. 

The Lie derivative operators on $M$ and $N$ are denoted $\Lcal$ and ${\nh\Lcal}$, respectively. For instance, the Lie derivative of ${\nh g}$ along a vector field  ${\nh\xi}\in \Ccal^1(TN)$ is defined by 
$$
{\nh\Lcal}_{{\nh\xi}}{\nh g}:=\lim_{t\to 0} \frac1t(\ga_{{\nh\xi}}(\cdot,t)^*{\nh g}-{\nh g}),
$$
where $\ga_{{\nh\xi}}$ denotes the flow of ${\nh\xi}$. This flow is defined as the mapping $(y,t)\in N\times (-\ep,\ep)\to \ga_{{\nh\xi}}(y,t)\in N$, where $\ep>0$ is a sufficiently small parameter (whose existence follows from the compactness of $M$), and $\ga_{{\nh\xi}}(y,\cdot)$ is the unique solution to the Cauchy problem 
$$
\frac{d}{d t} \ga_{{\nh\xi}}(y,t)={\nh\xi}(\ga_{{\nh\xi}}(y,t)) \text{ \ for all }  t\in (-\ep, \ep), \text{ \ and } 
\ga_{{\nh\xi}}(y,0)=y. 
$$

The notation $\xi|_\Gamma$ designates the restriction to the set $\Gamma$ of a function or tensor field $\xi$ defined over a set that contains $\Gamma$. Given any smooth fiber bundle $X$ over $M$ and any submanifold $\Gamma\subset M$, we denote by $\Ccal^k(X)$ the space of all sections of class $\Ccal^k$ of the fiber bundle $X$, and we let
$$
\Ccal^k(X|_\Gamma):=\{S|_\Gamma; \ S\in \Ccal^k(X)\}.
$$
If $S\in \Ccal^k(X)$ is a section of a fiber bundle $X$ over $M$, then $S(x)$ denotes the value of $S$ at $x\in M$.

The tangent at $x\in M$ of a mapping $\ph\in\Ccal^{k}(M,N)$ is a linear mapping $T_x\ph\in\Lcal(T_xM,T_{\ph(x)}N)$. The section $D\ph\in \Ccal^{k-1}(T^*\! M\otimes \ph^*TN)$, defined at each $x\in M$ by 
$$
D\ph(x)\cdot \xi(x):=(T_x\ph)(\xi(x)) \text{ for all } \xi\in TM,
$$
is the differential of $\ph$ at $x$. In local charts, 
$$
D\ph(x)=\frac{\pa\ph^\al}{\pa x^i}(x) \ dx^i(x)\otimes \frac{\pa}{\pa y^\al}(\ph(x)), 
\ x\in M.
$$

Let ${\nh\na} : \Ccal^k(TN)\to \Ccal^{k-1}(T^*\! N\otimes TN)$ denote the Levi-Civita connection on the Riemannian manifold $N$ induced by the metric ${\nh g}$. Any immersion $\ph\in\Ccal^{k+1}(M,N)$ induces the metrics 
$$
{g}={g}[\ph]:=\ph^*{{\nh g}} \in \Ccal^k(S^{+}_2M)
\text{ \ and \ }
{\nt g}={\nt g}[\ph]:={\bpb\ph}{\nh g} \in \Ccal^k(S^{+}_2(\ph^*TN)),
$$
where 
$$
(\ph^*{{\nh g}})({\xi},{\eta}):={\nh g}(\ph_*{\xi},\ph_*{\eta})\circ\ph 
\text{ \ and \ } 
({\bpb\ph}{\nh g}) ({\nh\xi}\circ\ph,{\nh\eta}\circ\ph):={\nh g}({{\nh\xi}},{\nh\eta})\circ\ph,
$$
and the corresponding connections
$$
\aligned
{\na}={\na}[\ph] &: \Ccal^k(TM)\to \Ccal^{k-1}(T^*\! M\otimes TM),\\
{\nt\na}={\nt\na}[\ph]&: \Ccal^k(\ph^*TN)\to \Ccal^{k-1}(T^*\! M\otimes \ph^*TN).
\endaligned
$$
In local coordinates, we have 
$$
{g}_{ij}:=\frac{\pa\ph^\al}{\pa x^i}\frac{\pa\ph^\be}{\pa x^j}\,{\nt g}_{\al\be}, \quad 
{\nt g}_{\al\be}:={\nh g}_{\al\be}\circ\ph, 
$$
and
$$
\aligned
{\nh\na}_\al\nh\xi^\be& =\frac{\pa\nh\xi^\be}{\pa y^\al}+{\nh\Ga}_{\al\ga}^\be \nh\xi^\ga, \\
{\na}_i{\xi}^j &= \frac{\pa{\xi}^j}{\pa x^i}+{\Ga}_{ik}^j {\xi}^k,\\
{\nt\na}_i \nt\xi^\al & =\frac{\pa\nt\xi^\al}{\pa x^i}+\frac{\pa\ph^\be}{\pa x^i}{\nt\Ga}_{\be\ga}^\al \nt\xi^\ga,
\endaligned
$$
where ${\nh\Ga}_{\al\ga}^\be$, ${\Ga}_{ik}^j$, and  ${\nt\Ga}_{\be\ga}^\al:={\nh\Ga}_{\be\ga}^\al \circ\ph$, denote the Christoffel symbols associated with the metric tensors ${\nh g}$, ${g}$, and ${\nt g}$, respectively. Note that the metric tensors ${g}$ and ${\nt g}$ and the connections ${\na}$ and ${\nt\na}$ all depend on the immersion $\ph$. To indicate this dependence, the notation ${g}[\ph]$, ${\nt g}[\ph]$, ${\na}[\ph]$, and ${\nt\na}[\ph]$, will sometimes be used instead of shorter notation ${g}$, ${\nt g}$, ${\na}$, and ${\nt\na}$.

The above connections are related to one another by the relations
\begin{equation}
\label{nabla}
{\nt\na}{\nt\xi}={D\ph}\cdot{\na}{\xi}=D\ph\cdot(({\nh\na}{\nh\xi})\circ\ph)
\end{equation}
for all ${\xi}\in \Ccal^k(TM)$, ${\nh\xi}:=\ph_*{\xi}$, ${\nt\xi}:={\nh\xi}\circ\ph$, which in local coordinates read: 
\begin{equation}
\label{nabla-ij}
{\nt\na}_i\nt\xi^\al
=\frac{\pa\ph^\al}{\pa x^j} {\na}_i{\xi}^j
=\frac{\pa\ph^\be}{\pa x^i} (({\nh\na}_\be\nh\xi^\al)\circ\ph),
\end{equation}
where $\displaystyle \nt\xi^\al=\nh\xi^\al\circ\ph:=\frac{\pa\ph^\al}{\pa x^i}{\xi}^i$. Note that 
$$
{\na}{\theta} =\ph^*({\nh\na}{\nh\theta})
\text{ \ and \ }
{\nt\na}_{\eta}{\nt\xi} = ({\nh\na}_{\nh\eta}{\nh\xi})\circ\ph, 
\text{ \ where }
{\nh\eta}=\ph_*{\eta}, 
\ {\theta}=\ph^*{\nh\theta}, 
\text{ and }
{\nt\xi}={\nh\xi}\circ\ph,
$$
for all  
${\eta}\in \Ccal^{k-1}(TM)$, 
${\nh\theta}\in \Ccal^k(T^*\! N)$ and 
${\nh\xi}\in \Ccal^k(TN)$, $k\geq 1$.

The connection ${\na}$, resp. ${\nh\na}$, is extended to arbitrary tensor fields on $M$, resp. on $N$, in the usual manner, by using the Leibnitz rule. The connection ${\nt\na}$ is extended to arbitrary sections ${\nt S}\in \Ccal^k(T^p_qM\otimes \ph^*(T^r_sN))$ by using the Leibnitz rule and the connection ${\na}={\na}[\ph]$.  For instance, if $p=q=r=s=1$, then the section ${\nt\na}_{{\eta}}{\nt S}\in \Ccal^{k-1}(T^p_{q}M\otimes \ph^*(T^r_sN))$ is defined by 
$$
\aligned
({\nt\na}_{{\eta}} {\nt S})({\xi},{\si}, {\nt\zeta},{\nt\tau}) 
:= {\eta}({\nt S}({\xi},{\si}, {\nt\zeta},{\nt\tau}))
- {\nt S}({\na}_{{\eta}} {\xi},{\si}, {\nt\zeta},{\nt\tau}) 
- {\nt S}({\xi},{\na}_{{\eta}}{\si}, {\nt\zeta},{\nt\tau}) 
\\  
\quad 
- {\nt S}({\xi},{\si}, {\nt\na}_{{\eta}}{\nt\zeta},{\nt\tau}) 
- {\nt S}({\xi}, {\si}, {\nt\zeta},{\nt\na}_{{\eta}}{\nt\tau}), 
\endaligned
$$
for all sections ${\eta}\in \Ccal^{k-1}(T M)$, ${\xi}\in \Ccal^k(TM)$, ${\si}\in \Ccal^k(T^*\! M)$, ${\nt\zeta} \in \Ccal^k(\ph^*TN) $, and ${\nt\tau} \in \Ccal^k(\ph^*T^*\!N)$.

The divergence operators induced by the connections ${\na}={\na}[\ph]$, ${\nt\na}={\nt\na}[\ph]$, and ${\nh\na}$, are respectively denoted ${\divop}={\divop}[\ph]$, ${\wt\divop}={\wt\divop}[\ph]$, and ${\wh\divop}$. In particular, if ${\nt\Tbs}={\nt T}\otimes{\om}$ with ${\nt T}\in\Ccal^1(TM\otimes\ph^*T^*\!N)$ and ${\om}\in\Ccal^1(\La^nM)$, then, at each  $x\in M$,
$$
\aligned
({\wt\divop}\,{\nt T})(x) & :=({\nt\na}_i {\nt T}^{i}_\al)(x)dy^\al(\ph(x)),\\
({\wt\divop}\,{\nt\Tbs})(x) & :=({\nt\na}_i {\nt\Tbs}^{i}_{j_1...j_n\,\al})(x) 
dx^{j_1}(x)\otimes ...\otimes dx^{j_n}(x)\otimes dy^\al(\ph(x)).
\endaligned
$$
If in addition ${\na}{\om}=0$, then 
$$
{\nt\na}_{{\eta}} {\nt\Tbs}=({\nt\na}_{{\eta}} {\nt T})\otimes {\om} 
\text{ \ and \ }
{\wt\divop}\,{\nt\Tbs}=({\wt\divop}\,{\nt T})\otimes{\om}.
$$

The interior product $\bs{i}_{{\eta}}:{\nt\Tbs}\in \Ccal^0(TM\otimes \ph^*T^*\! N\otimes \La^nM)\to \bs{i}_{{\eta}}{\nt\Tbs}\in \Ccal^0(TM\otimes \ph^*T^*\! N\otimes\La^{n-1}M)$ is defined by 
$$
(\bs{i}_{{\eta}}{\nt\Tbs})({\theta}, {\nt\xi},{\zeta}_1,...,{\zeta}_{n-1}) :={\nt\Tbs}({\theta}, {\nt\xi},{\eta}, {\zeta}_1,...,{\zeta}_{n-1})
$$
for all ${\eta},{\zeta}_1,...,{\zeta}_{n-1}\in \Ccal^0(TM)$, ${\theta}\in \Ccal^0(T^*\! M)$, and ${\nt\xi}\in \Ccal^0(\ph^*TN)$, or equivalently, by
$$
\bs{i}_{{\eta}}{\nt\Tbs}={\nt T}\otimes \bs{i}_{{\eta}}{\om} 
\text{ \ if \ } {\nt\Tbs}={\nt T}\otimes{\om}.
$$  

The \emph{normal trace} of a tensor field ${\nt\Tbs}={\nt T}\otimes{\om} \in \Ccal^0(TM\otimes \ph^*T^*\! N\otimes\La^nM)$ on the boundary ${\pa M}$ is defined by 
$$
{\nt\Tbs}_{\nu}:=(\bs{i}_{{\nu}}{\nt\Tbs})\cdot({\nu}\cdot{g})\in \Ccal^0((\ph^*T^*\! N)|_{\pa M}\otimes \La^{n-1}({\pa M})),
$$
or equivalently, by
\begin{equation}
\label{Tn}
{\nt\Tbs}_{\nu}=({\nt T}\cdot ({\nu}\cdot{g}))\otimes \bs{i}_{{\nu}}{\om}  \text{ \ on } {\pa M}, 
\end{equation}
where ${\nu}$ denotes the unit outer normal vector field to ${\pa M}$ defined by the metric ${g}$. Note that the definition of ${\nt\Tbs}_{\nu}$ is independent of the choice of the Riemannian metric ${g}$, since 
$$
(\bs{i}_{\nu_1}{\nt\Tbs})\cdot(\nu_1\cdot g_1)=(\bs{i}_{\nu_2}{\nt\Tbs})\cdot( \nu_2\cdot g_2)  \text{ \ on } {\pa M} 
$$
for all Riemannian metrics $g_1$ and $g_2$ on $M$ ($\nu_i$ denotes the unit outer normal vector field to ${\pa M}$ defined by the metric $g_i$, $i=1,2$). Indeed, 
$$
(\bs{i}_{\nu_1}{\nt\Tbs})\cdot(\nu_1\cdot g_1)=g_2(\nu_1,\nu_2)[(\bs{i}_{\nu_2}{\nt\Tbs})\cdot( \nu_1\cdot g_1)]  \text{ \ on } {\pa M}
$$
and 
$$
g_2(\nu_1,\nu_2)( \nu_1\cdot g_1)=\nu_2\cdot g_2  \text{ \ on } {\pa M}.
$$

Integration by parts formulae involving either connection ${\na}$, ${\nt\na}$ and ${\nh\na}$ will be needed in Section \ref{sect6}. We establish here the formula for the connection ${\nt\na}$, since it does not seem to appear elsewhere in the literature.  Letting $M=N$ and $\ph=\id_M$ in the lemma below yields the integration by parts formulae for the other two connections ${\na}$ and ${\nh\na}$, which otherwise are classical. Recall that $\cdot$ , resp. $:$ , denotes the contraction of one, resp. two, indices (no confusion about the indices should arise).

\begin{lemma}
\label{IP}
For each ${\nt\xi}\in\Ccal^1(\ph^*TN)$ and for each ${\nt\Tbs}\in \Ccal^1(TM\otimes \ph^*T^*\! N\otimes\La^nM)$, 
$$
\int_M {\nt\Tbs}:{\nt\na}{\nt\xi} = - 
\int_M ({\wt\divop}\,{\nt\Tbs})\cdot {\nt\xi}+
\int_{{\pa M}} {\nt\Tbs}_{\nu}\cdot{\nt\xi}.
$$
\end{lemma}

\begin{proof}
Let ${\om}$ denote the volume form induced by the metric ${g}$ and let ${\nt T} \in \Ccal^1(TM\otimes \ph^*T^*\! N)$ be defined by ${\nt\Tbs}={\nt T}\otimes{\om}$. Then 
$$
{\nt\Tbs}:{\nt\na}{\nt\xi}=({\nt T}: {\nt\na}{{\nt\xi}})\,{\om}, \quad 
{\wt\divop}\,{\nt\Tbs}\cdot{\nt\xi}=({\wt\divop}\,{\nt T}\cdot{\nt\xi})\,{\om},
$$
and 
$$
{\nt\Tbs}_{\nu}\cdot{\nt\xi}
=(\bs{i}_{{\nu}}{\nt\Tbs}\cdot({\nu}\cdot{g}))\cdot{\nt\xi}
=[({\nt T}\cdot({\nu}\cdot{g}))\cdot{\nt\xi}] \,\bs{i}_{{\nu}}{\om};
$$
hence proving the integration by parts formula of Lemma \ref{IP} is equivalent to proving that 
$$
\int_M ({\nt T}: {\nt\na}{{\nt\xi}})\,{\om}=-\int_M ({\wt\divop}\,{\nt T}\cdot{\nt\xi})\,{\om}+\int_{{\pa M}}[({\nt T}\cdot({\nu}\cdot{g}))\cdot{\nt\xi}] \,\bs{i}_{{\nu}}{\om}.
$$

Since ${\nt T}:{\nt\na}{\nt\xi} = {\nt T}^{i}_\al{\nt\na}_i \nt\xi^\al ={\na}_i({\nt T}^{i}_\al\nt\xi^\al)-({\nt\na}_i {\nt T}^{i}_\al)\nt\xi^\al$, we have
$$
\aligned
\int_M ({\nt T}: {\nt\na}{{\nt\xi}})\, {\om}
&=\int_M{\divop}({\nt T}\cdot {\nt\xi})\,{\om} -\int_M ({\wt\divop}{\nt T}\cdot {\nt\xi})\, {\om} 
=\int_M \Lcal_{{\nt T}\cdot{\nt\xi}}\, {\om}-\int_M ({\wt\divop}{\nt T}\cdot{\nt\xi})\, {\om}, 
\endaligned
$$
where $\Lcal$ denotes the Lie derivative on $M$. The first integral of the right-hand side can be written as
$$
\int_M \Lcal_{{\nt T}\cdot{\nt\xi}}\, {\om}
=\int_M d (\bs{i}_{{\nt T}\cdot{\nt\xi}}{\om})
=\int_{{\pa M}} \bs{i}_{{\nt T}\cdot{\nt\xi}}{\om}.
$$
Let ${\nu} $ be the unit outer normal vector field to ${\pa M}$ defined by the metric ${g}$. 
Since 
$$
{\nt T}\cdot{\nt\xi}={g}({\nt T}\cdot{\nt\xi},{\nu}){\nu}+\{{\nt T}\cdot{\nt\xi} - {g}({\nt T}\cdot{\nt\xi},{\nu}){\nu}\} \text{ \ on } {\pa M}, 
$$
and since the vector field $\{{\nt T}\cdot{\nt\xi}-{g}({\nt T}\cdot{\nt\xi},{\nu}){\nu}\}$ is tangent to ${\pa M}$, the integrand of the last integral becomes
$$
\aligned
 \bs{i}_{{\nt T}\cdot{\nt\xi}}{\om} 
& = \bs{i}_{{g}({\nt T}\cdot{\nt\xi},{\nu}){\nu}} {\om} 
 ={g}({\nt T}\cdot{\nt\xi},{\nu})\bs{i}_{{\nu}} {\om}
 =({\nt T}\cdot{\nt\xi})\cdot({\nu}\cdot{g})\,\bs{i}_{{\nu}} {\om}
  =[({\nt T}\cdot({\nu}\cdot{g}))\cdot{\nt\xi}] \,\bs{i}_{{\nu}}{\om} 
  \text{ \ on } {\pa M}.
\endaligned
$$
Therefore,
$$
\int_M ({\nt T}: {\nt\na}{{\nt\xi}})\,{\om}=-\int_M ({\wt\divop}\,{\nt T}\cdot{\nt\xi})\,{\om}+\int_{{\pa M}}[({\nt T}\cdot({\nu}\cdot{g}))\cdot{\nt\xi}] \,\bs{i}_{{\nu}}{\om}.
$$
\end{proof}

All functions and tensor fields appearing in Sections \ref{sect3}-\ref{sect7} are of class $\Ccal^k$ over their domain of definition,  with $k$ sufficiently large so that all differential operators be defined in the classical sense (as opposed to the distributional sense). Functions and tensor fields belonging to Sobolev spaces on the Riemannian manifold $(M,{\z g})$, where ${\z g}:={\z\ph}^*{\nh g}$ denotes the pullback of the Riemannian metric ${\nh g}$ by a reference deformation ${\z\ph}\in\Ccal^1(M,N)$, will be used in Sections \ref{sect8}-\ref{sect9} in order to prove existence theorems for the equations of elastostatics introduced in Sections \ref{sect6} and \ref{sect7}. The Sobolev space $W^{k,p}(TM)$ is defined for each $k\in\Nbb$ and $1\leq p<\infty$ as the completion in the Lebesgue space $L^p(TM)$ of the space $\Ccal^k(TM)$ with respect to the norm 
$$
\| {\xi} \|_{k,p}=\| {\xi} \|_{W^{k,p}(TM)}:=\Big\{ \int_M 
\Big(|{\xi}|^p+\sum_{\ell=1}^k|{\zz\na}^\ell {\xi} |^p\Big) 
\,{\z\om} \Big\}^{1/p},
$$
where 
$$
\aligned
|{\zz\na}^\ell {\xi}|
&:=\{{\z g}({\zz\na}^\ell {\xi},{\zz\na}^\ell {\xi})\}^{1/2}\\
&=
\left\{
{({g_0})}_{i j} {({g_0})}^{i_1 j_1}...{({g_0})}^{i_\ell j_\ell}  ({\zz\na})_{i_1...i_\ell}{\xi}^i({\zz\na})_{j_1...j_\ell}{\xi}^j
\right\}^{1/2}.
\endaligned
$$

The Sobolev space $W^{k,p}_0(TM)$ is defined as the closure in $W^{k,p}(TM)$ of the space 
$$
\Ccal^k_c(TM):=\{{\xi}\in \Ccal^k(TM); \ \mathrm{support}({\xi})\subset \mathrm{int}M\}.
$$
We will also use the notation $H^k(TM) := W^{k, 2}(TM)$ and $H^k_0(TM) := W^{k, 2}_0(TM)$.

%==================================================================
 
\section{Kinematics}
\label{sect3}

Consider an elastic body undergoing a deformation in a Riemannian manifold $(N,{\nh g})$ in response to external forces. Let the material points of the body be identified with the points of a manifold $M$, hereafter called the \emph{abstract configuration} of the body. Examples of abstract configurations of a body are a subset ${\z N}\subset N$ that the body occupies in the absence of external forces, or the range of a global chart of ${\z N}$ (under the assumption that it exists). Unless otherwise specified, the manifolds $M$ and $(N,{\nh g})$ satisfy the same regularity assumptions as in the previous section. 

All the kinematic notions introduced below are natural extensions of their counterparts in classical elasticity. Specifically, if $(N,{\nh g})$ is the three-dimensional Euclidean space and if the reference configuration of the elastic body is described by a global chart with $M$ as its range, then our definitions coincide with the classical ones in curvilinear coordinates; see, for instance, \cite{ciarletG}.

A \emph{deformation} of the body is an immersion $\Ccal^1(M,N)$ that preserves orientation and  satisfies the axiom of impenetrability of matter. This means that 
$$
\aligned
& \det D\ph(x)>0 \text{ \ for all }x \in M, \ \\
& \ph|_{\textup{int}M}:\textup{int}M\to N \text{ is injective},
\endaligned
$$
where $\textup{int}M$ denotes the interior of $M$. Note that $\ph$ needs not be injective on the whole $M$ since self-contact of the deformed boundary may occur. 

A \emph{displacement field} of the configuration $\ph(M)$ of the body is a section ${\nt\xi}\in \Ccal^1(\ph^*TN)$. It is often convenient to identify displacement fields of $\ph(M)$ with vector fields on ${\xi}\in\Ccal^1(TM)$ by means of the bijective mapping 
$$
{\xi} \to {\nt\xi}:=(\ph_*{\xi})\circ\ph.
$$
When no confusion should arise, a vector field ${\xi}\in\Ccal^1(TM)$ will also be called displacement field.

\begin{remark}
An example of displacement field is the velocity field of a body: If $\psi(t):M\to N$ is a time-dependent family of deformations and $\ph:=\psi(0)$, then ${\nt\xi}:=\frac{d\psi}{dt}(0)$ is a displacement field of the configuration $\ph(M)$.
\end{remark}

Deformations $\psi:M\to N$ that are close in the $\Ccal^0(M,N)$-norm (the smallness assumption is specified below) to a given deformation $\ph\in\Ccal^1(M,N)$ are canonically related to the displacement fields ${\nt\xi}\in \Ccal^0(\ph^*TN)$ of the configuration $\ph(M)$ of the body 
by the relation 
$$
\psi=(\wh\exp\, {\nh\xi})\circ\ph, \ \ {\nh\xi}\circ\ph={\nt\xi}, 
$$
where $\wh\exp$ denotes the exponential maps on $N$. When ${\nt\xi}=(\ph_*{\xi})\circ\ph$ is defined by means of a vector field ${\xi}\in\Ccal^0(TM)$ on the abstract configuration $M$, we let
\begin{equation}
\label{def-exp}
\psi=\exp_\ph{\xi}:=(\wh\exp\, \ph_*{\xi})\circ\ph.
\end{equation}
Of course, these relations only make sense if 
$
|{\nh\xi}(\ph(x))|=|(\ph_*{\xi})(\ph(x))|< {\nh\de}({\ph(x)})
$ 
for all $x\in M$, where ${\nh\de}(y)$ denotes the injectivity radius of $N$ at $y\in N$ (i.e., ${\nh\de}(y)$ is the largest radius for which the exponential map at $y$ is a diffeomorphism). 

Let ${\nh\de}(\ph(M)):=\min_{y\in \ph(M)}{\nh\de}(y)$ be the injectivity radius of the compact subset $\ph(M)$ of $N$, and define the set 
\begin{equation}
\label{C0-phi}
\Ccal^0_{\ph}(TM):=\{{\xi}\in \Ccal^0(TM); \ \|\ph_*{\xi}\|_{\Ccal^0(TN|_{\ph(M)})}<{\nh\de}(\ph(M))\}.
\end{equation}
It is then clear from the properties of the exponential maps on $N$ that the mapping 
$$
\exp_\ph=\wh\exp\circ D\ph :\Ccal^0_{\ph}(TM) \to \Ccal^0(M,N)
$$
is a $\Ccal^1$-diffeomorphism onto its image. Together with its inverse, denoted $\exp_{\ph}^{-1}$, this diffeomorphism will be used in Sections \ref{sect7}-\ref{sect9} to transform the equations of elasticity in which the unknown is the deformation $\ph$ of a body into equivalent equations in which the unknown is the displacement field ${\xi}$ of a given configuration ${\z\ph}(M)$ of the body; of course, the two formulations are equivalent only for vector fields $\xi:=\exp_{{\z\ph}}^{-1}\ph$ that are sufficiently small in the $\Ccal^0(TM)$-norm.

\begin{remark}
(a) 
The relation $\psi=\exp_{\ph}{\xi}$ means that, for each $x\in M$, $\psi(x)$ is the end-point of the geodesic arc in $N$ with length $|{\xi}(x)|$ starting at the point $\ph(x)$ in the direction of $(\ph_*{\xi})(\ph(x))$.

(b) 
The relation ${\xi}=\exp_{\ph}^{-1}\psi$ means that, for each $x\in M$, ${\xi}(x)$ is the pullback by the immersion $\ph$ of the vector that is tangent at $\ph(x)$ to the geodesic arc joining $\ph(x)$ to $\psi(x)$ in $N$ and whose norm equals the length of this geodesic arc. 
\end{remark}

The \emph{metric tensor field}, also called the \emph{right Cauchy-Green tensor field}, associated with a deformation $\ph\in\Ccal^1(M,N)$ is the pullback by $\ph$ of the metric ${\nh g}$ of $N$, i.e., 
$$
g[\ph]:=\ph^*{\nh g}.
$$
Note that the notation $C:={g}[\ph]$ is often used in classical elasticity.

The \emph{strain tensor field}, also called the \emph{Green-St Venant tensor field}, associated with two deformations $\ph,\psi\in\Ccal^1(M,N)$ is defined by 
$$
E[\ph,\psi]
:=\frac12(g[\psi]-g[\ph]). 
$$
The first argument $\ph$ is considered as a reference deformation, while the second argument $\psi$ is an arbitrary deformation.

The \emph{linearized strain tensor field}, also called the \emph{infinitesimal strain tensor field}, associated with a deformation $\ph\in \Ccal^1(M,N)$ and a vector field ${\xi}\in \Ccal^1(TM)$ (recall that $(\ph_*{\xi})\circ\ph$ is then a displacement field of the configuration $\ph(M)$ of the body) is the linear part with respect to ${\xi}$ of the mapping ${\xi}\mapsto E[\ph,\exp_\ph{\xi}]$, i.e., 
$$
e[\ph,{\xi}]:=\left[\frac{d}{d t}E[\ph,\exp_\ph(t{\xi})]\right]_{t=0}.
$$
Explicit expressions of $e[\ph,{\xi}]$ are given in Proposition \ref{e} below.

\begin{remark}
\label{E-law}
Let ${\z\ph}\in \Ccal^1(M,N)$ be a reference deformation and let ${\z g}={g}[{\z\ph}]:={\z\ph^*}{\nh g}$. 
Given any $(x,y)\in M\times N$ and any $F,G\in T_x^*\! M\otimes T_yN$, let $(F,G)^*:T^0_{2,y}N\to T^0_{2,x}M$ denote the pullback mapping, defined in terms of the bilinear mapping $(F,G):T^0_{2,x}M\to T^0_{2,y}N$ by letting  
\begin{equation*}
\label{FG}
((F,G)^*{\nh\tau})({\xi},{\eta})
:={\nh\tau}(F{\xi},G{\eta}) 
\text{ \ for all } ({\xi},{\eta})\in T_xM\times T_xM 
\text{ and all } {\nh\tau}\in T^0_{2,y}N, 
\end{equation*} 
and let 
$$
{\Flaw g}(x,y,F):=(F,F)^*{\nh g}(y) 
\text{ \ and \ } 
{\Flaw E}(x,y,F):=\frac12({\Flaw g}(x,y,F)-{\z g}(x)).
$$
Then
$$
g[\ph](x)={\Flaw g}(x,\ph(x),D\ph(x)) \text{ \ and \ } E[\ph,{\z\ph}](x)= {\Flaw E}(x,\ph(x),D\ph(x)).
$$
The mappings ${\Flaw g}$ and ${\Flaw E}$ defined in this fashion are called the \emph{constitutive laws} of the right Cauchy-Green tensor field ${g}[\ph]$ and of the Green-St Venant tensor field $E[{\z\ph},\ph]$ associated with a deformation $\ph$. 
\end{remark}

The linearized strain tensor field $e[\ph,{\xi}]$ can be expressed either in terms of the Lie derivative on $M$ or on $N$, or in terms of either of the connections defined in the previous section, as we now show. 
Recall that $\cdot$ denotes the (partial) contraction of one single index of two tensors.

\begin{e-proposition}
\label{e}
Given any immersion $\ph\in\Ccal^1(M,N)$ and any vector field ${\xi}\in\Ccal^1(TM)$, define the vector fields 
$$
{\nh\xi}={\nh\xi}[\ph]:=\ph_*{\xi}\in \Ccal^1(T\ph(M)) \text{ \ and \ } 
{\nt\xi}={\nt\xi}[\ph]:=({\ph_*{\xi}})\circ\ph\in\Ccal^1(\ph^*TN),
$$
and the corresponding one-form fields
$$
{\xi}^\flat={\xi}^\flat[\ph]:={g}[\ph] \cdot {\xi}, \quad 
\nh\xi^\flat=\nh\xi^\flat[\ph]:={\nh g}\cdot {\nh\xi}[\ph], \text{ and } 
\nt\xi^\flat=\nt\xi^\flat[\ph]:={\nt g}[\ph]\cdot {\nt\xi}[\ph],
$$ 
where ${g}={g}[\ph]:=\ph^*{{\nh g}}$ and ${\nt g}={\nt g}[\ph]:={\bpb\ph}{\nh g}$ (see Section \ref{sect2}). Let ${\na}={\na}[\ph]$, ${\nt\na}={\nt\na}[\ph]$, and ${\nh\na}$, respectively denote the connections induced by the metric tensors ${g}$, ${\nt g}$, and ${\nh g}$, and let $\Lcal$ and ${\nh\Lcal}$ respectively denote the Lie derivative operators on $M$ and on $N$. Then  
$$
e[\ph,{\xi}]
 =\frac12\Lcal_{{\xi}}{g}
 =\frac12\ph^*({\nh \Lcal}_{{\nh\xi}}{\nh g})
$$
and 
\begin{equation}
\label{e's}
\aligned
e[\ph,{\xi}]
&=\frac12 ({\na}{\xi}^\flat+({\na}{\xi}^\flat)^T)
  =\frac12 ({g}\cdot{\na}{\xi}+({g}\cdot{\na}{\xi})^T)\\
& =\frac12\ph^*({\nh\na}\nh\xi^\flat+({\nh\na}\nh\xi^\flat)^T) 
   = \frac12\ph^*({\nh g}\cdot{\nh\na}{\nh\xi}+({\nh g}\cdot {\nh\na}{\nh\xi})^T)\\
&=\frac12( D\ph\cdot{\nt\na}\nt\xi^\flat+(D\ph\cdot{\nt\na}\nt\xi^\flat)^T)
  =\frac12 ({\nt g}\cdot D\ph\cdot{\nt\na}{\nt\xi}+({\nt g}\cdot D\ph\cdot{\nt\na}{\nt\xi})^T). 
\endaligned
\end{equation}
In local charts, equations \eqref{e's} are equivalent to the relations
\begin{equation}
\label{eij's}
\aligned
e_{ij}[\ph,{\xi}]
& = \frac12({\na}_i{\xi}_j + {\na}_j{\xi}_i)
   = \frac12({g}_{jk}{\na}_i{\xi}^k+ {g}_{ik}{\na}_j{\xi}^k)\\
& = \frac12\frac{\pa\ph^\al}{\pa x^i}\frac{\pa\ph^\be}{\pa x^j}({\nh\na}_\be{\nh\xi}_\al 
    + {\nh\na}_\al{\nh\xi}_\be)\circ\ph \\
& =\frac12\Big(\frac{\pa\ph^\be}{\pa x^i}{\nt\na}_j{\nt\xi}_\be 
    + \frac{\pa\ph^\be}{\pa x^j}{\nt\na}_i{\nt\xi}_\be\Big),
\endaligned
\end{equation}
where, at each $x\in M$, 
${\xi}^\flat(x)={\xi}_i(x)dx^i(x)$,
${\nh\xi}^\flat(y)={\nh\xi}_\al(y) dy^\al(y)$, 
and  
${\nt\xi}^\flat(x)={\nt\xi}_\al(x)dy^\al(\ph(x))$.
\end{e-proposition}

\begin{proof}  
For each $t$ in a neighborhood of zero, define the deformations
$$
\ph(\cdot,t):=\exp_\ph(t{\xi}) \text{ \ and \ } \psi(\cdot,t):=\ga_{{\nh\xi}}(\cdot,t)\circ\ph,
$$
where ${\nh\xi}\in \Ccal^1(TN)$ denotes any extension of the section $\ph_*{\xi} \in \Ccal^1(T\ph(M))$ and $\ga_{{\nh\xi}}$ denotes the flow of ${\nh\xi}$ (see Section \ref{sect2}). By definition, 
$$
e[\ph, {\xi}]
=\left[\frac{d}{dt}E[\ph,\ph(\cdot,t)]\right]_{t=0}
=\lim_{t\to 0}\frac{\ph(\cdot,t)^*{\nh g}-\ph^*{\nh g}}{2t}.
$$

Since 
$$
\frac{\pa\ph}{\pa t}(x,0)
= \frac{\pa \psi}{\pa t}(x,0)
={\xi}(x) \text{ for all } x\in M,
$$
it follows from the above expression of $e[\ph,\xi]$ that 
$$
e[\ph,{\xi}]=\lim_{t\to 0}\frac{\psi(\cdot,t)^*{\nh g}-\ph^*{\nh g}}{2t}.
$$
Then the definition of the Lie derivative yields
$$
\aligned
e[\ph,{\xi}]
&=\ph^*\Big(\lim_{t\to 0} \frac{\ga_{{\nh\xi}}(\cdot,t)^*{\nh g}-{\nh g}}{2t}\Big)
=\frac12\ph^*({\nh \Lcal}_{{\nh\xi}}{\nh g})\\
&=\frac12\ph^*({\nh \Lcal}_{\ph_*{\xi}}{\nh g})
=\frac12\Lcal_{{\xi}}(\ph^*{\nh g})
=\frac12\Lcal_{{\xi}}{g}.
\endaligned
$$

Expressing the Lie derivative ${\nh \Lcal}_{{\nh\xi}}{\nh g}$ in terms of the connection ${\nh\na}$ gives
$$
\aligned
e_{ij}[\ph, {\xi}]
& = \frac12\frac{\pa\ph^\al}{\pa x^i}\frac{\pa\ph^\be}{\pa x^j}\big( {\nh g}_{\al\ga}{\nh\na}_\be\nh\xi^\ga  + {\nh g}_{\be\ga}{\nh\na}_\al\nh\xi^\ga \big)\circ\ph\\
& = \frac12\frac{\pa\ph^\al}{\pa x^i}\frac{\pa\ph^\be}{\pa x^j}\big({\nh\na}_\be{\nh\xi}_\al  + {\nh\na}_\al{\nh\xi}_\be\big)\circ\ph.
\endaligned
$$
This implies in turn that 
$$
\aligned
e_{ij}[\ph,{\xi}]
& = \frac12({g}_{ik}{\na}_j{\xi}^k+{g}_{jk}{\na}_i{\xi}^k)\\
& = \frac12({\na}_j{\xi}_i+{\na}_j{\xi}_j),
\endaligned
$$
and
$$
\aligned
e_{ij}[\ph, {\xi}]
& =\frac12\Big\{{\nt g}_{\al\ga}\frac{\pa\ph^\al}{\pa x^i}{\nt\na}_j\nt\xi^\ga  + {\nt g}_{\be\ga}\frac{\pa\ph^\be}{\pa x^j}{\nt \na}_i{\nt \xi}^\ga\Big\}\\
& =\frac12\Big\{\frac{\pa\ph^\al}{\pa x^i}{\nt\na}_j{\nt\xi}_\al  + \frac{\pa\ph^\be}{\pa x^j}{\nt\na}_i{\nt\xi}_\be\Big\}.
\endaligned
$$
\end{proof}

\begin{remark}
(a) 
Given any vector field ${\nh\xi}\in \Ccal^1(TN)$, define the linearized strain tensor field 
\begin{equation}
\label{e-hat}
{\nh e}[{\nh\xi}]
:=\frac12{\nh\Lcal}_{{\nh\xi}}{\nh g}
   =\frac12({\nh\na}\nh\xi^\flat+({\nh\na}\nh\xi^\flat)^T) 
     = \frac12({\nh g}\cdot{\nh\na}{\nh\xi}+({\nh g}\cdot {\nh\na}{\nh\xi})^T)
       \in \Ccal^0(S_2N). 
\end{equation}
Proposition \ref{e} shows that 
\begin{equation}
\label{e=e-hat}
{e}[\ph,{\xi}]=\ph^*({\nh e}[\ph_*{\xi}]).
\end{equation}

(b) A vector field ${\xi}\in\Ccal^1(TM)$ defines two families of deformations, $\ph(\cdot,t)$ and $\psi(\cdot,t)$, both starting at $\ph$ with velocity ${\nt\xi}=(\ph_*{\xi})\circ\ph$; see the proof of Proposition \ref{e}. Note that $\ph(\cdot,t)$ depends on the metric tensor field ${\nh g}$ of the manifold $N$, while $\psi(\cdot,t)$ depends only on the differential structure of the manifold $M$.
\end{remark}

%==================================================================

\section{Elastic materials}   
\label{sect4}

The behavior of elastic bodies in response to applied forces clearly depends on the elastic material of which they are made. Thus, before studying this behavior, one needs to specify this material by means of a constitutive law, i.e., a relation between deformations and stresses inside the body. Note that a constitutive law is usually given only for deformations $\ph$ that are close to a reference deformation $\z\ph$, so that plasticity and heating do not occur. 

We assume in this paper that the body is made of a hyperelastic material satisfying the axiom of frame-indifference, that is, an elastic material whose behavior is governed by a \emph{stored energy function} ${\Flaw{\Wbs}}:={\zz{\Flaw{W}}}{\z\om}$ satisfying the relation \eqref{FI} below. The \emph{stress tensor field} associated with a deformation $\ph:M\to N$ of the body will then be defined by any one of the sections $\Sibs[\ph]$, ${\nh\Sibs}[\ph]$, $\Tbs[\ph]$, ${\nt\Tbs}[\ph]$, and ${\nh\Tbs}[\ph]$ (see Definition \ref{defT}), which are related to each other by the formulae  \eqref{Tbs's} of Proposition \ref{propT} below.

Let a reference configuration ${\z\ph}(M)\subset N$ of the body be given by means of an immersion $\z\ph\in \Ccal^2(M,N)$. The metric tensor fields and the connections induced by $\z\ph$ on $TM$ and on $\z\ph^*TN$ are denoted by (see Section \ref{sect2})
$$
{\z g}:={g}[{\z\ph}], \ 
{\z{\nt g}}:={\nt g}[{\z\ph}], \ 
{\zz\na}:={\na}[{\z\ph}], \
{\zz{\nt\na}}:={\nt\na}[{\z\ph}]. 
$$
The volume form induced by $\z\ph$, or equivalently by the metric tensor field ${\z g}$,  on the manifold $M$ is denoted ${\z\om}:=\z\ph^*{\nh\om}$. 

The \emph{strain energy} corresponding to a deformation $\ph$ of an hyperelastic body is defined by 
$$
I[\ph]:=\int_M \Wbs[\ph]=\int_M {\z W}[\ph]{\z\om},
$$
where the $n$-form field $\Wbs[\ph]={\z W}[\ph]{\z\om}\in L^1(\La^nM)$ is of the form 
$$
(\Wbs[\ph])(x):={\Flaw{\Wbs}}(x,\ph(x), D\ph(x))={\zz{\Flaw{W}}}(x,\ph(x), D\ph(x)){\z\om}(x), \ x\in M,
$$
for some given mapping ${\Flaw{\Wbs}}(x,y,\cdot)={\zz{\Flaw{W}}}(x,y,\cdot)\, {\z\om}(x): T_x^*\! M\otimes T_yN\to \La^n_{x}M$, $(x,y)\in M\times N$, called the \emph{stored energy function} of the elastic material constituting the body.

We say that the stored energy function satisfies the axiom of material frame-indifference if 
$$
{\zz{\Flaw{W}}}(x,y,F)={\zz{\Flaw{W}}}(x,y',RF)
$$
for all $x\in M$, $y\in N$, $y'\in N$, $F\in T_x^*\! M\otimes T_yN$, and all isometries $R\in T_y^*\! N\otimes T_{y'}N=\Lcal(T_yN,T_{y'}N)$. 
In this case, the polar decomposition theorem applied to the linear mapping $F$ implies that, for each $x\in M$, there exist mappings ${\z{\Claw{W}}}(x,\cdot):S^{+}_{2,x}M\to \Rbb$ and ${\zz{\Elaw{W}}}(x,\cdot):S_{2,x}M\to \Rbb$ such that 
\begin{equation}
\label{FI}
{\zz{\Flaw{W}}}(x,y,F)={\z{\Claw{W}}}(x,C)={\zz{\Elaw{W}}}(x,E) \text{ \ for all } F\in T_x^*\! M\otimes T_yN,
\end{equation}
where the tensors $C$ and $E$ are defined in terms of $F$ by (see Remark \ref{E-law})
$$
C={\Flaw g}(x,y,F):=(F,F)^*({\nh g}(y))
\text{ \ and \ } 
E={\Flaw E}(x,y,F):=\frac12(C-{\z g}(x)).
$$ 
Hence the axiom of material frame-indifference implies that, at each $x\in M$, 
\begin{equation}
\label{w}
\aligned
(\Wbs[\ph])(x)
:\!&={\Flaw{\Wbs}}(x,\ph(x), D\ph(x))={\zz{\Flaw{W}}}(x,\ph(x), D\ph(x)){\z\om}(x)\\
&={\Claw{\Wbs}}(x, ({g}[\ph])(x))  = {\z{\Claw{W}}}(x, ({g}[\ph])(x)){\z\om}(x)\\
&=\bs{\Elaw W}(x, (E[{\z\ph}, \ph])(x)) ={\zz{\Elaw{W}}}(x, (E[{\z\ph},\ph])(x)){\z\om}(x), 
\endaligned
\end{equation}
where
$$
\aligned
(g[\ph])(x)&={\Flaw g}(x,\ph(x),D\ph(x))=(\ph^*{\nh g})(x) ,\\
(E[{\z\ph}, \ph])(x) & ={\Flaw E}(x,\ph(x),D\ph(x))=\frac12((g[\ph])(x)-{\z g}(x)).
\endaligned
$$

Let $(x,y)\in M\times N$. The Gateaux derivative of the mapping ${\Flaw{\Wbs}}(x,y,\cdot):T^*_xM\otimes T_yN\to \La^n_{x}M$ at $F\in T^*_xM\otimes T_yN$ in the direction $G\in T_x^*\! M\otimes T_yN$ is defined by
$$
\frac{\pa{\Flaw{\Wbs}}}{\pa F}(x,y,F):G=\lim_{t\to 0}\frac1t\Big\{{\Flaw{\Wbs}}(x,y,F+tG)-{\Flaw{\Wbs}}(x,y,F) \Big\}. 
$$

The \emph{constitutive law} of an elastic material whose stored energy function is ${\Flaw{\Wbs}}$ is the mapping that associates to each $(x,y)\in M\times N$ and each $F\in \Lcal(T_xM,T_yN)=T^*_xM\otimes T_yN$ the tensor 
\begin{equation}
\label{cl1}
{\Flaw{\nt\Tbs}}(x,y,F)={\zz{\Flaw{\nt{T}}}}(x,y,F)\otimes{\z\om}(x):=\frac{\pa{\Flaw{\Wbs}}}{\pa F}(x,y, F)=\frac{\pa{\zz{\Flaw{W}}}}{\pa F}(x,y, F)\otimes{\z\om}(x)
\end{equation}
in $(T_xM\otimes T_y^*\! N)\otimes \La^{n}_{x} M$.

The \emph{constitutive law} of an elastic material whose stored energy function is $\bs{\Elaw W}$ is the mapping associating to each $x\in M$ and each $E\in S_{2,x}M$ the tensor
\begin{equation}
\label{cl2}
{\Elaw{\Sibs}}(x,E)={\zz{\Elaw{\Si}}}(x,E)\otimes{\z\om}(x):=\frac{\pa\bs{\Elaw W}}{\pa E}(x,E)=\frac{\pa{\zz{\Elaw{W}}}}{\pa E}(x,E)\otimes{\z\om}(x)
\end{equation}
in $S^2_xM\otimes \La^n_xM$. 
 
The next lemma establishes a relation between the constitutive laws ${\Flaw{\nt\Tbs}}$ and ${\Elaw\Sibs}$ when the corresponding stored energy functions ${\Flaw\Wbs}={\z{\Flaw W}}\,{\z\om}$ and $\bs{\Elaw W}={\z{\Elaw W}}\,{\z\om}$ are related by \eqref{FI}.

\begin{lemma}
\label{T-Si}
Let the stored energy functions ${\Flaw\Wbs}={\z{\Flaw W}}\,{\z\om}$ and $\bs{\Elaw W}={\z{\Elaw W}}\,{\z\om}$ satisfy  \eqref{FI}.
Then 
$$
{\Flaw{\nt\Tbs}}(x,y,F)={\nh g}(y)\cdot F \cdot {\Elaw{\Sibs}}(x,E), \text{ \ where } 
E={\Flaw E}(x,y,F)=\frac12\{(F,F)^*{\nh g}(y)-{\z g}(x)\},
$$
for all linear operators $F\in T^*_xM\otimes T_yN=\Lcal(T_xM,T_yN)$.
\end{lemma}

\begin{proof}
It suffices to prove that ${\zz{\Flaw{\nt{T}}}}(x,y,F)={\nh g}(y)\cdot F \cdot {\zz{\Elaw{\Si}}}(x,E)$. Since ${\zz{\Flaw{W}}}(x,y,F)={\zz{\Elaw{W}}}(x, E)$, the chain rule implies that, for each $G\in T_x^*\! M\otimes T_yN$, 
$$
{\zz{\Flaw{\nt{T}}}}(x,y,F): G
=\frac{\pa {\zz{\Flaw{W}}}}{\pa F}(x,y,F):G
=\frac{\pa {\zz{\Elaw{W}}}}{\pa E}(x,E):\Big(\frac{\pa {\Flaw E}}{\pa F}(y,F):G\Big).
$$
Besides, 
$$
\aligned
\frac{\pa {\Flaw E}}{\pa F}(y,F):G
&=\lim_{t\to 0}\frac1{2t}\Big\{(F+tG,F+tG)^*{\nh g}(y)-(F,F)^*{\nh g}(y)\Big\}\\
&=\frac12\Big\{(F,G)^*{\nh g}(y)+(G,F)^*{\nh g}(y)\Big\}.
\endaligned
$$
Since the tensors ${\nh g}(y)$ and ${\z\Si}(x,E)=\frac{\pa {\zz{\Elaw{W}}}}{\pa E}(x,E)$ are both symmetric, the last two relations imply that 
$$
{\zz{\Flaw{\nt{T}}}}(x,y,F): G={\zz{\Elaw{\Si}}}(x,E) : (F,G)^*{\nh g}(y),
$$
which is the same as 
$$
{\zz{\Flaw{\nt{T}}}}(x,y,F): G=\left\{{\nh g}(y)\cdot F\cdot {\zz{\Elaw{\Si}}}(x,E)\right\}: G.
$$
\end{proof}

We are now in a position to define the stress tensor field associated with a deformation $\ph\in\Ccal^1(M,N)$ of an elastic body, a notion that plays a key role in all that follows.

\begin{e-definition}
\label{defT}
Let ${\z g}:=\z\ph^*{\nh g}$ and ${\z\om}:=\z\ph^*{\nh\om}$ respectively denote the metric tensor field and the volume form induced by a reference deformation ${\z\ph} \in \Ccal^1(M,N)$, let ${g}[\ph]:=\ph^*{\nh g}$ and ${\om}[\ph]:=\ph^*{\nh\om}$ respectively denote the metric tensor field and the volume form induced by a generic deformation $\ph\in\Ccal^1(M,N)$, and let 
$$
E[\z\ph,\ph]:=\frac12({g}[\ph]-{\z g})
$$
denote the strain tensor field associated with the deformations ${\z\ph}$ and $\ph$. Let ${\Elaw{\Sibs}}$ denote   the constitutive law defined by \eqref{cl2}. 

(a)
The \emph{stress tensor field} associated with a deformation $\ph$ is either of the following sections
$$
\aligned
& \Sibs[\ph] := {\Elaw{\Sibs}}(\cdot,E[\z\ph,\ph]) 
, && \qquad {\nh\Sibs}[\ph] := \ph_*(\Sibs[\ph]) 
, \\
& {\Tbs}[\ph]:={g}[\ph]\cdot {\Sibs}[\ph] 
, && \qquad {\nh\Tbs}[\ph] := {\nh g}\cdot {\nh\Sibs}[\ph] 
,\\ 
&   {\nt\Tbs}[\ph] := {\nt g}[\ph]\cdot D\ph \cdot \Sibs[\ph] 
, && \qquad
\endaligned 
$$
where $\cdot$ denotes the contraction of one index (no ambiguity should arise) and $\ph_*(\Sibs[\ph]):=\ph_*(\Si[\ph])\otimes{\nh\om}$ for each $\Sibs[\ph]=\Si[\ph]\otimes{\om}[\ph]$.

(b) The tensor fields 
${\Si}[\ph], {\z\Si}[\ph] \in \Ccal^0(S^2M)$
and 
${T}[\ph], {\z T}[\ph]  \in \Ccal^0(T^1_1M)$
and
${\nt T}[\ph], {\z{\nt T}}[\ph]  \in \Ccal^0(TM\otimes \ph^*T^*\! N)$
and 
${\nh\Si}[\ph] \in  \Ccal^0(S^2N|_{\ph(M)})$ and 
${\nh T}[\ph]  \in \Ccal^0(T^1_1N|_{\ph(M)})$, 
defined by 
$$
\aligned
& \Sibs[\ph] = {\Si}[\ph]\otimes{\om}[\ph]={\z\Si}[\ph]\otimes{\z\om}, 
&&\qquad {\nh\Sibs}[\ph]={\nh\Si}[\ph]\otimes{\nh\om},
\\
& \Tbs[\ph] = {T}[\ph]\otimes{\om}[\ph]={\z T}[\ph]\otimes{\z\om}, 
&& \qquad {\nh\Tbs}[\ph] ={\nh T}[\ph]\otimes{\nh\om},
\\
& {\nt\Tbs}[\ph] = {\nt T}[\ph]\otimes{\om}[\ph]={\z{\nt T}}[\ph]\otimes{\z\om},
&& \qquad 
\endaligned
$$
are also called stress tensor fields.

(c) The \emph{first Piola-Kirchhoff}, the \emph{second Piola-Kirchhoff}, and the \emph{Cauchy},  stress tensor fields associated with the deformation $\ph$ are the sections ${\z{\nt T}}[\ph]$, ${\z\Si}[\ph]$, and ${\nh T}[\ph]$, respectively.
\end{e-definition}

\begin{remark}
\label{Tij}
(a) 
The stress tensor fields ${\Si}[\ph]$, ${\z{\Si}}[\ph]$, and ${\nh\Si}[\ph]$, are symmetric.  

(b)
The stress tensor fields ${\Si}[\ph]$, ${T}[\ph]$, ${\nt T}[\ph]$, ${\nh\Si}[\ph]$, and ${\nh T}[\ph]$ are obtained from each other by lowering and raising indices in local charts. Specifically, if at each $x\in M$, 
$$
\aligned
& {\Si}[\ph](x) = {\Si}^{ij}(x) \ \frac{\pa}{\pa x^i}(x)\otimes \frac{\pa}{\pa x^j}(x),
&& \quad {\nh\Si}[\ph](\ph(x)) = {\nh\Si}^{\al\be}(\ph(x))\ \frac{\pa}{\pa y^\al}(\ph(x))\otimes \frac{\pa}{\pa y^\be}(\ph(x)),
\\
& {T}[\ph](x) = {T}^i_j(x)\ \frac{\pa}{\pa x^i}(x)\otimes dx^j(x), 
&& \quad {\nh T}[\ph](\ph(x)) = {\nh T}^\al_\be(\ph(x))\ \frac{\pa}{\pa y^\al}(\ph(x))\otimes dy^\be(\ph(x)),
\endaligned
$$
$$
{\nt T}[\ph](x) = {\nt T}^i_\be(x)\ \frac{\pa}{\pa x^i}(x)\otimes dy^\be(\ph(x)), 
$$
then  
\begin{equation}
\label{T-coeff}
\aligned
{T}^i_j = {g}_{jk}{\Si}^{ik}, \quad
{\nt T}^i_\al = {\nt g}_{\al\be}\frac{\pa\ph^\be}{\pa x^j}{\Si}^{ij}, \quad
{\nh\Si}^{\al\be}\circ\ph = \frac{\pa\ph^\al}{\pa x^i} \frac{\pa\ph^\be}{\pa x^j} {\Si}^{ij}, \text{ \ and \ } 
{\nh T}^\al_\be = {\nh g}_{\be\tau}{\nh\Si}^{\al\tau},
\endaligned
\end{equation}
where ${\nh g}_{\al\be}$, ${g}_{ij}$, and ${\nt g}_{\al\be}=({\nh g}_{\al\be}\circ\ph)$ respectively denote the components of the metric tensor fields ${\nh g}$, $g[\ph]=\ph^*{\nh g}$, and ${\nt g}[\ph]:={\nh g}\circ\ph$.

(c) The components of the stress tensor fields $\Sibs[\ph]$, $\Tbs[\ph]$, and ${\nt \Tbs}[\ph]$, over the volume forms ${\om}[\ph]$ and ${\z\om}$ are related to one another by 
\begin{equation}
\label{om-omz}
{\Si}[\ph]=\rho[\ph]\, {\z\Si}[\ph], \ 
{T}[\ph]=\rho[\ph]\, {\z T}[\ph], \ 
{\nt T}[\ph]=\rho[\ph]\, {\z{\nt T}}[\ph],
\end{equation}
where the function $\rho[\ph]:M\to\Rbb$ is defined by $\rho[\ph]{\om}[\ph]={\z\om}$. In local charts,
\begin{equation*}
\label{rho}
\rho(x)
=\frac{\det(\frac{\pa\z\ph^\al}{\pa x^i}(x))}{\det(\frac{\pa\ph^\al}{\pa x^i}(x))} \text{ \ for all \ } x\in M.
\end{equation*}
\end{remark}

The next proposition gathers for later use several properties of the stress tensor fields $\Tbs[\ph]$, ${\nt\Tbs}[\ph]$ and ${\nh\Tbs}[\ph]$.

\begin{e-proposition}
\label{propT}
(a) 
Let ${\Flaw{\nt\Tbs}}$ be the constitutive law defined by \eqref{cl1}. Then 
$$
({\nt\Tbs}[\ph])(x)={\Flaw{\nt\Tbs}}(x,\ph(x),D\ph(x)) \in (T_xM\otimes T_{\ph(x)}^*\! N)\otimes \La^n_{x}M, \ x\in M.
$$  

(b) 
The stress tensor fields appearing in Definition \ref{defT}  
are related to one another by 
\begin{equation}
\label{Tbs's}
\aligned
&{\Sibs}[\ph]: e[\ph,{\xi}] 
={\Tbs}[\ph]: {\na}{\xi}
={\nt\Tbs}[\ph]: {\nt\na}{\nt\xi}
=\ph^*({\nh\Tbs}[\ph]: {\nh\na}{\nh\xi})
=\ph^*({\nh\Sibs}[\ph]: {\nh e}[{\nh\xi}]),\\
&{\Si}[\ph]: e[\ph,{\xi}] 
={T}[\ph]: {\na}{\xi}
={\nt T}[\ph]: {\nt\na}{\nt\xi}
=({\nh T}[\ph]: {\nh\na}{\nh\xi})\circ\ph
=({\nh\Si}[\ph]: {\nh e}[{\nh\xi}])\circ\ph
\endaligned
\end{equation}
for all vector fields ${\xi}\in \Ccal^1(TM)$, where 
$$
{\Tbs}[\ph]: {\na}{\xi}:=({T}[\ph]: {\na}{\xi})\otimes\om[\ph]
\text{ \ and \ } 
{\nh\Tbs}[\ph]: {\nh\na}{\nh\xi}:=({\nh T}[\ph]: {\nh\na}{\nh\xi})\otimes{\nh\om}.
$$
As above, the vector fields ${\xi}$, ${\nt\xi}$ and ${\nh\xi}$ appearing in these relations are related to to one another by means of the formulae 
$$
{\nt\xi}=(\ph_*{\xi})\circ\ph 
\text{ \ and \ }
{\nh\xi}=\ph_*{\xi}.
$$
\end{e-proposition}

\begin{proof}
The relation of part (a) of Proposition \ref{propT} is an immediate consequence of Lemma \ref{T-Si}. The relations of part (b) are equivalent in a local chart to the relations (with self-explanatory notations):
$$
{\Si}^{ij} e_{ij}[\ph,{\xi}]
={T}^i_k \, {\na}_i\xi^k
={\nt T^i_\al}\, {\nt\na}_i {\nt\xi^\al}
=({\nh T^\be_\al}\ {\nh\na}_\be{\nh\xi^\al})\circ\ph
=({\nh \Si^{\be\tau}}\ {\nh e}_{\be\tau}[{\nh\xi}])\circ\ph.
$$

Using the relations 
${\Si}^{ij}={\Si}^{ji}$ and ${T}^i_k={g}_{jk}{\Si}^{ij}$ (cf. Remark \ref{Tij}), 
and noting that 
$e_{ij}[\ph,{\xi}]= \frac12({g}_{jk}{\na}_i{\xi}^k+ {g}_{ik}{\na}_j{\xi}^k)$ (cf. Theorem \ref{e}), 
we first obtain 
$$
{\Si}^{ij} e_{ij}[\ph,{\xi}]={\Si}^{ij}{g}_{jk}{\na}_i{\xi}^k={T}^i_k \, {\na}_i\xi^k. 
$$

We next infer from the relations 
${T}^i_k=\frac{\pa\ph^\al}{\pa x^k}{\nt T^i_\al}$ (cf. Remark \ref{Tij}) 
and 
${\nt\na}_i\nt\xi^\al=\frac{\pa\ph^\al}{\pa x^k} {\na}_i{\xi}^k$ (cf. relations \eqref{nabla-ij}) 
that  
$$
{T}^i_k \, {\na}_i\xi^k={\nt T^i_\al}\,(\frac{\pa\ph^\al}{\pa x^k}{\na}_i\xi^k)={\nt T^i_\al}\, {\nt\na}_i\nt\xi^\al. 
$$

Furthermore, since 
${\nt T^i_\al}={\nt g}_{\al\tau}\frac{\pa\ph^\tau}{\pa x^k}{\Si}^{ik}$ and  
${\nh\Si}^{\tau\be}\circ\ph=\frac{\pa\ph^\tau}{\pa x^k}\frac{\pa\ph^\be}{\pa x^i}{\Si}^{ik}$ and 
${\nh T_\al^\be}={\nh g}_{\al\tau} {\nh\Si}^{\tau\be}$ (cf. Remark \ref{Tij}), 
and since 
${\nt\na}_i\nt\xi^\al=\frac{\pa\ph^\be}{\pa x^i} (({\nh\na}_\be\nh\xi^\al)\circ\ph)$ (cf. relations \eqref{nabla-ij}), 
we have
$$
{\nt T^i_\al}\, {\nt\na}_i {\nt\xi^\al}
={\nt g}_{\al\tau}(\frac{\pa\ph^\tau}{\pa x^k}\frac{\pa\ph^\be}{\pa x^i}{\Si}^{ik}) (({\nh\na}_\be\nh\xi^\al)\circ\ph)
=({\nh g}_{\al\tau}\circ\ph) ({\nh\Si}^{\tau\be}\circ\ph) (({\nh\na}_\be\nh\xi^\al)\circ\ph) 
=({\nh T^\be_\al}\ {\nh\na}_\be{\nh\xi^\al})\circ\ph.
$$

Finally, since 
${\nh T^\be_\al}={\nh g}_{\al\tau}{\nh \Si^{\be\tau}}$ and 
${\nh \Si^{\be\tau}}={\nh \Si^{\tau\be}}$ (cf. Remark \ref{Tij}), and since 
${\nh e}_{\be\tau}[\nh\xi]=\frac12({\nh g}_{\al\tau}{\nh\na}_\be{\nh\xi^\al}+{\nh g}_{\al\be}{\nh\na}_\tau{\nh\xi^\al})$ (cf. relations \eqref{eij's} and \eqref{e-hat}), we also have 
$$
{\nh T^\be_\al}\ {\nh\na}_\be{\nh\xi^\al}
=\frac12{\nh \Si^{\be\tau}}({\nh g}_{\al\tau}{\nh\na}_\be{\nh\xi^\al}+{\nh g}_{\al\be}{\nh\na}_\tau{\nh\xi^\al})
={\nh \Si^{\be\tau}}\ {\nh e}_{\be\tau}[\ph,\nh\xi].
$$
\end{proof}

%==================================================================

\section{Applied forces}
\label{sect5}

We assume in this paper that the external body and surface forces acting on the elastic body under consideration are conservative, in the sense that they are defined by means of a potential $P:\Ccal^1(M,N)\to\Rbb$ of the form 
\begin{equation}
\label{pot}
{P}[\ph]:=\int_M \Fbs[\ph]+\int_{{\Ga_2}} \Hbs[\ph] 
=\int_{\ph(M)} {\nh\Fbs}[\ph]+\int_{\ph({\Ga_2})} {\nh\Hbs}[\ph],
\end{equation}
where the volume forms 
$$
\Fbs[\ph]=\ph^*({\nh\Fbs}[\ph])\in \Ccal^0(\La^nM) 
\text{ \ and \ }
\Hbs[\ph]=(\ph|_{{\Ga_2}})^*({\nh\Hbs}[\ph]) \in \Ccal^0(\La^{n-1}{\Ga_2})
$$
are given for each admissible deformation $\ph\in \Ccal^1(M,N)$ of the elastic body. 
 
 Let $\ph\in\Ccal^1(M,N)$ be a deformation of the elastic body. 
The work of the applied body and surface forces corresponding to a displacement field ${\nt\xi}=(\ph_*{\xi})\circ\ph$, ${\xi}\in\Ccal^0(TM)$, of the configuration $\ph(M)$ of the body is denoted ${V}[\ph]{\xi}$ and is defined as the derivative of the functional $P:\Ccal^1(M,N)\to\Rbb$ at $\ph$ in the direction ${\nt\xi}$. Assuming that ${\nh\Fbs}$ and ${\nh\Hbs}$ are sufficiently regular, there exist sections ${\nh\fbs}[\ph] \in \Ccal^0(T^*\! N\otimes \La^{n}\! N|_{\ph(M)})$ and ${\nh\hbs}[\ph] \in \Ccal^0(T^*\! N\otimes \La^{n-1}\! N|_{\ph({\Ga_2})})$ such that 
\begin{equation}
\label{work}
\aligned
{V}[\ph]{\xi}
:= {P}'[\ph]{\nt\xi}
&= \int_{\ph(M)} {\nh\fbs}[\ph]\cdot {\nh\xi} +\int_{\ph({\Ga_2})} {\nh\hbs}[\ph]\cdot {\nh\xi}\\
&=\int_{M} {\fbs}[\ph]\cdot {\xi} +\int_{{\Ga_2}} {\hbs}[\ph]\cdot {\xi} 
   =\int_{M} {\nt\fbs}[\ph]\cdot {\nt\xi} + \int_{{\Ga_2}} {\nt\hbs}[\ph]\cdot {\nt\xi},
\endaligned
\end{equation}
for all ${\xi}\in\Ccal^0(TM)$, where ${\nh\xi}:=\ph_*{\xi}$, ${\nt\xi}:={\nh\xi}\circ\ph$, and 
\begin{equation}
\label{f's}
\aligned
& {\fbs}[\ph]=\ph^*({\nh\fbs}[\ph]), 
 && {\hbs}[\ph]=(\ph^*({\nh\hbs}[\ph]))|_{\Ga_2},  
\\ & {\nt\fbs}[\ph]={\bpb\ph}({\nh\fbs}[\ph]), 
 &&{\nt\hbs}[\ph]=({\bpb\ph}({\nh\hbs}[\ph])|_{\Ga_2}
. \endaligned
\end{equation}
The tensor fields ${\fbs}[\ph]$, ${\nt\fbs}[\ph]$ and ${\nh\fbs}[\ph]$, resp. ${\hbs}[\ph]$, ${\nt\hbs}[\ph]$ and ${\nh\hbs}[\ph]$, are called the \emph{densities of the applied body}, resp. \emph{surface},  \emph{forces}. 

The notation $\cdot$ appearing in \eqref{work} denotes as before the contraction of one index: if ${\fbs}[\ph]={f}[\ph]\otimes{\om}[\ph]$, ${\nt\fbs}[\ph]={\nt f}[\ph]\otimes{\om}[\ph]$, and ${\nh\fbs}[\ph]={\nh f}[\ph]\otimes{\nh\om}$, then 
$$
{\fbs}[\ph]\cdot {\xi}:=({f}[\ph]\cdot {\xi})\,{\om}[\ph],
\quad
{\nt\fbs}[\ph]\cdot {\nt\xi}:=({\nt f}[\ph]\cdot{\nt\xi})\,{\om}[\ph],
\text{ \ and } 
{\nh\fbs}[\ph]\cdot {\nh\xi}:=({\nh f}[\ph]\cdot{\nh\xi})\,{\nh\om}.
$$
The pullback operator $\ph^*:T^*\! N\otimes \La^kN \to T^*\! M\otimes\La^kM$ and the ``bundle pullback'' operator $\bpb\ph:T^*\! N \otimes \La^kN\to \ph^*T^*\! N\otimes \La^kM$ appearing in \eqref{f's} are defined explicitly in Remark \ref{forces}(b) below. 

We assume in this paper that the applied body and surface forces $\fbs[\ph]$ and  $\hbs[\ph]$ are local, so that their constitutive equations are of the form: 
$$
\aligned
({\fbs}[\ph])(x) & =\bs{\Flaw f}(x,\ph(x),D\ph(x)), \ x\in M, \\
({\hbs}[\ph])(x) & =\bs{\Flaw h}(x,\ph(x),D\ph(x)),\ x\in \Ga_2,
\endaligned
$$
for some (given) mappings 
$\bs{\Flaw f}(x,y,\cdot)={\z{\Flaw f}}(x,y,\cdot)\otimes{\z\om}(x):T^*_xM\otimes T_yN\to T^*\!M\otimes \La^nM$, $(x,y)\in M\times N$,  and 
$\bs{\Flaw h}(x,y,\cdot)={\z{\Flaw h}}(x,y,\cdot)\otimes(\bs{i}_{{\z\nu}}{\z\om})(x):T^*_xM\otimes T_yN\to T^*\!M\otimes \La^{n-1}M|_{\Ga_2}$, $(x,y)\in \Ga_2\times N$.

\begin{remark}
\label{forces}
Let ${g}[\ph]:=\ph^*{\nh g}$ and ${\z g}:=\z\ph^*{\nh g}$ be the metric tensor fields on $M$ induced by the deformations $\ph \in \Ccal^1(M,N)$ and ${\z\ph} \in \Ccal^1(M,N)$, let  ${\nu}[\ph]$ and ${\z\nu}$ denote the unit outer normal vector fields to the boundary of $M$ with respect to ${g}[\ph]$ and ${\z g}$, respectively, and let ${\nh \nu}[\ph]$ denote the unit outer normal vector fields to the boundary of $\ph(M)$ with respect to the metric ${\nh g}$. Let 
$$
{\om}[\ph]:=\ph^*{\nh\om}, \ 
{\z\om}:=\z\ph^*{\nh\om} \in \Ccal^0(\La^nM),  \text{ \ and \ } 
{\nh\om}\in\Ccal^\infty(\La^n N),
$$ 
be the volume forms induced by these metrics on $M$ and on $N$, respectively, and let 
$$
\bs{i}_{{\nu}[\ph]}{\om}[\ph], \ 
\bs{i}_{{\z\nu}}{\z\om}\in \Ccal^1(\La^{n-1}{\Ga_2}), \text{ \ and \ } 
\bs{i}_{{\nh\nu}[\ph]}{\nh\om}\in \Ccal^0(\La^{n-1}\ph({\Ga_2})),
$$
denote the corresponding volume forms on the hypersurfaces $\Ga_2\subset M$ and on $\ph({\Ga_2})\subset N$.

(a)
The one-form fields
${f}[\ph], {\z f}[\ph] \in \Ccal^0(T^*\!M)$,
${\nt f}[\ph], {\z{\nt f}}[\ph] \in \Ccal^0(\ph^*T^*\! N)$, 
${\nh f}[\ph]  \in  \Ccal^0(T^*\!N|_{\ph(M)})$ 
and 
${h}[\ph], {\z h}[\ph]  \in \Ccal^0(T^*\!M|_{\Ga_2})$,
${\nt h}[\ph], {\z{\nt h}}[\ph]  \in \Ccal^0(\ph^*T^*\!N|_{\Ga_2})$, 
${\nh h}[\ph]  \in  \Ccal^0(T^*\!N|_{\ph(\Ga_2)})$,
defined in terms of the densities of the applied forces appearing in \eqref{work} and \eqref{f's} by 
$$
\aligned
& \fbs[\ph] = {f}[\ph] \otimes {\om}[\ph]={\z f}[\ph] \otimes {\z\om}, 
&& \quad \hbs[\ph] = {h}[\ph] \otimes \bs{i}_{{\nu}[\ph]}{\om}[\ph]={\z h}[\ph] \otimes \bs{i}_{{\z\nu}}{\z\om},
\\
& {\nt\fbs}[\ph] = {\nt f}[\ph] \otimes {\om}[\ph]={\z{\nt f}}[\ph] \otimes {\z\om},
&& \quad {\nt\hbs}[\ph] = {\nt h}[\ph] \otimes \bs{i}_{{\nu}[\ph]}{\om}[\ph]={\z{\nt h}}[\ph] \otimes \bs{i}_{{\z\nu}}{\z\om},
\\
& {\nh\fbs}[\ph]={\nh f}[\ph] \otimes {\nh\om},
&& \quad {\nh\hbs}[\ph] ={\nh h}[\ph] \otimes \bs{i}_{{\nh\nu}[\ph]}{\nh\om},
\endaligned
$$
are related by 
\begin{equation*}
\label{om-omz-fh}
\aligned
& {f}[\ph]=\rho[\ph]\, {\z f}[\ph]=\ph^*({\nh f}[\ph]), 
&& \quad
{\nt f}[\ph]={\nh f}[\ph]\circ\ph, \\
&{h}[\ph]=\rho[\ph|_\Ga]\, {\z h}[\ph]=\ph^*({\nh h}[\ph])|_{{\Ga_2}}, 
&& \quad 
{\nt h}[\ph]={\nh h}[\ph]\circ\ph|_{\Ga_2},
\endaligned
\end{equation*}
where the (scalar) functions $\rho[\ph]:M\to\Rbb$ and $\rho[\ph|_\Ga]:\Ga\to\Rbb$ are defined by 
$$
\rho[\ph]{\om}[\ph]={\z\om}
\text{ \quad and \quad } 
\rho[\ph|_\Ga] \ (\ph|_{{\Ga_2}})^*(\bs{i}_{{\nh\nu}[\ph]}{\nh\om})=\bs{i}_{\z\nu}{{\z\om}}.
$$

(b)
The components in a local chart of the external body and surface forces, which are defined at each $x\in M$ by the relations
$$
\aligned
& {f}[\ph](x) = {f_i}(x) \ dx^i(x),
&& \quad {h}[\ph](x) = {h_i}(x) \ dx^i(x),
\\
& {\nt f}[\ph](x) = {\nt f_\al}(x) \  dy^\al(\ph(x)),
&& \quad  {\nt h}[\ph](x) = {\nt h_\al}(x) \  dy^\al(\ph(x)),
\\
& {\nh f}[\ph](\ph(x)) = {\nh f_\al}(\ph(x))\ dy^\al(\ph(x)),
&&  \quad {\nh h}[\ph](\ph(x)) = {\nh h_\al}(\ph(x)) \ dy^\al(\ph(x)), 
\endaligned
$$
are related to one another by
$$
{f}_i:= \frac{\pa\ph^\al}{\pa x^i} {\nt f_\al}, \quad  {\nt f}_\al= {\nh f_\al}\circ\ph,\text{ \ and \ } 
{h}_i:= \frac{\pa\ph^\al}{\pa x^i} {\nt h_\al}, \  {\nt h}_\al= {\nh h_\al}\circ\ph.
$$

(c)
If the volume forms ${\nh\Fbs}[\ph]={\nh\Fbs}$ and ${\nh\Hbs}[\ph]={\nh\Hbs}$ are independent of the deformation $\ph$, then the densities of the body and surface forces appearing in \eqref{work} are given explicitly by 
$$ 
{\nt\fbs}[\ph]\cdot{\nt\xi}=\ph^*({\nh \Lcal}_{{\nh\xi}}{\nh\Fbs}) \text{ and } {\nt\hbs}[\ph]\cdot{\nt\xi}=\ph^*({\nh \Lcal}_{{\nh\xi}}{\nh\Hbs}),
$$ 
for all ${\nt\xi}={\nh\xi}\circ\ph\in \Ccal^1(\ph^*TN)$. Indeed, in this case we have
$$
\aligned
V[\ph]{\xi}={P}'[\ph]{\nt\xi}
& =\Big[\frac{d}{dt}P(\ga_{{\nh\xi}}(\cdot,t)\circ\ph)\Big]_{t=0} 
=\int_M \ph^*({\nh \Lcal}_{{\nh\xi}}{\nh\Fbs}) + \int_{{\Ga_2}} \ph^*({\nh \Lcal}_{{\nh\xi}}{\nh\Hbs}).
\endaligned
$$
\end{remark}

%==================================================================

\section{Nonlinear elasticity}
\label{sect6}

In this section we combine the results of the previous sections to derive the model of nonlinear elasticity in a Riemannian manifold, first as a minimization problem, then as variational equations, and finally as a boundary value problem.

The principle of least energy that constitutes the keystone of nonlinear elasticity theory developed in this paper states that the deformation $\ph:M\to N$ of the body under conservative forces independent of time should minimize the total energy of the body over the set of all admissible deformations.

The total energy is defined as the difference between the strain energy $I[\ph]$ and the potential of the applied forces ${P}[\ph]$, viz.,
$$
J[\ph]:=I[\ph]-{P}[\ph]=\int_M\Wbs[\ph]-\Big(\int_M\Fbs[\ph]+\int_{{\Ga_2}}\Hbs[\ph]\Big),
$$
the dependence on $\ph$ of the densities $\Wbs[\ph]$, $\Fbs[\ph]$ and $\Hbs[\ph]$ being that specified by the constitutive laws of the material and applied forces (cf. Sections~\ref{sect4} and~\ref{sect5}). 

We recall that a \emph{deformation} of the body is an immersion $\ph\in\Ccal^1(M,N)$ that preserves orientation at all points of $M$ and satisfies the axiom of impenetrability of matter at all points of the interior of $M$; cf. Section \ref{sect3}. An \emph{admissible deformation} is a deformation that satisfies in addition a Dirichlet boundary condition, also called boundary condition of place, on a given measurable subset ${\Ga_1}\subset {\pa M}$ of the boundary of the body. Thus the set of all admissible deformations is defined by 
$$
\Phi:=\{\ph\in\Ccal^1(M,N); \ \ph|_{\textup{int}M} \text{ injective}, \ \det D\ph>0 \text{ in } M, \ \ph=\ph_1 \text{ on } {\Ga_1}\}, 
$$
where $\ph_1\in\Ccal^1({\Ga_1},N)$ is an immersion that specifies the position in $N$ of the points of the body that are kept fixed. 
In this paper we assume that $\ph_1={\z\ph}|_{{\Ga_1}}$, where ${\z\ph}\in \Ccal^2(M,N)$ is the reference deformation of the body.

Therefore, the principle of least energy asserts that the following proposition is true without proof:

\begin{e-proposition}
\label{PAL}
Let the total energy associated with a deformation $\psi\in \Ccal^1(M,N)$ of an elastic body be defined by 
\begin{equation}
\label{TE}
J[\psi]:=I[\psi]-{P}[\psi]=\int_M\Wbs[\psi]-\Big(\int_M\Fbs[\psi]+\int_{{\Ga_2}}\Hbs[\psi]\Big),
\end{equation}
and let the set of all admissible deformations of the body be defined by 
\begin{equation}
\label{Phi}
\Phi:=\{\psi\in\Ccal^1(M,N); \ \psi|_{\textup{int}M} \text{ injective}, \ \det D\psi>0 \text{ in } M, \ \psi=\ph_0 \text{ on } {\Ga_1}\}.
\end{equation}
Then the deformation $\ph$ of the body satisfies the following minimization problem:
\begin{equation}
\label{MP}
\ph\in \Phi \text{ \ and \  } J[\ph]\leq J[\psi] \text{ \ for all } \psi\in \Phi.
\end{equation}
\end{e-proposition}

The set $\Phi$ defined in this fashion does not coincide with the set of deformations with finite energy. Therefore, minimizers of $J$ are usually sought in a larger set, defined by weakening the regularity of the deformations. However, the existence of such minimizers is still not guaranteed, since the functional $J[\psi]$ is not convex with respect to $\psi$ for realistic constitutive laws; cf. \cite{ball,ball2} and the references therein for the particular case where $(N,{\nh g})$ is an Euclidean space. One way to alleviate this difficulty is to adapt the strategy of J. Ball \cite{ball}, who assumed that ${\Flaw{\Wbs}}$ is polyconvex,  to a Riemannian manifold $(N,{\nh g})$. Another way is to study the existence of critical points instead of minimizers of $J$. It is the latter approach that we follow in this paper. 

To this end, we first derive the variational equations of nonlinear elasticity, or the principle of virtual work, in a Riemannian manifold from the principle of least energy stated in Proposition \eqref{PAL}. 
 
The principle of virtual work states that the deformation of a body should satisfy the Euler-Lagrange equations associated to the functional $J$ appearing in the principle of least energy. We will derive below several equivalent forms of the principle of virtual work, one for each stress tensor field ${\Sibs}[\ph]$,  ${\Tbs}[\ph]$, ${\nt\Tbs}[\ph]$, ${\nh\Tbs}[\ph]$, or ${\nh\Sibs}[\ph]$, defined in Section \ref{sect4}.

The set of admissible deformations $\Phi$ being that defined by \eqref{Phi}, the \emph{space of admissible displacement fields} associated with a deformation $\ph\in\Phi$ of the body is defined by 
\begin{equation}
\label{Xi-tilde}
\aligned
{\nt\Xi}[\ph]
& =\{{\nt\xi}\in \Ccal^1(\ph^*TN); \ {\nt\xi}=0 \text{ on } {\Ga_1}\}\\
& =\{{\nt\xi}=(\ph_*{\xi})\circ\ph;  \  {\xi}\in {\Xi}\}\\
& =\{{\nt\xi}={\nh\xi}\circ\ph; \ {\nh\xi}\in {\nh\Xi}[\ph]\},
\endaligned
\end{equation}
where 
\begin{equation}
\label{Xi}
\aligned
{\Xi} & :=\{{\xi}\in \Ccal^1(TM); \ {\xi}=0 \text{ on } {\Ga_1}\}, \\
{\nh\Xi}[\ph]& :=\{{\nh\xi}\in \Ccal^1(TN|_{\ph(M)});  \ {\nh\xi}=0 \text{ on } \ph({\Ga_1})\}.
\endaligned
\end{equation}
Note that ${\nt\Xi}[\ph]$ and ${\nh\Xi}[\ph]$ depend on the deformation $\ph$, whereas ${\Xi}$ does not. 
The space ${\nh\Xi}[\ph]$ is called the \emph{space of admissible displacement fields on the configuration} $\ph(M)$ of the body. 

In what follows we assume that the stored energy function $\bs{\Elaw W}={\z{\Elaw W}} {\z\om}$ of the elastic material constituting the  body is of class $\Ccal^1$,  i.e., that ${\z{\Elaw W}}\in \Ccal^1(S_2M)$. In this case, a solution $\ph\in\Ccal^1(M,N)$ to the minimization problem \eqref{MP} is a critical point of the total energy $J=I-P$, defined by $\eqref{TE}$, that is, it satisfies the variational equations
$$
J'[\ph] {\nt\xi} = 0 \text{ \ for all } {\nt\xi}\in {\nt\Xi}[\ph]. 
$$
This equation is called the \emph{principle of virtual work}. The next proposition states this principle in five equivalent forms, one for each stress tensor field appearing in Definition \ref{defT}. 
Note that the first two equations are defined over the abstract configuration $M$ and are expressed in terms of the stress tensors fields ${\Tbs}[\ph]$ and ${\Sibs}[\ph]$,  also defined on $M$. 
The third equation is defined over $M$, but is expressed in terms of the stress tensor field ${\nt\Tbs}[\ph]$, which is defined over both $M$ and the deformed configuration $\ph(M)$. 
The last two equations are defined over the deformed configuration $\ph(M)$ and are expressed in terms of the stress tensor fields ${\nh\Tbs}[\ph]$ and ${\nh\Sibs}[\ph]$, also defined over $\ph(M)$.

\begin{e-proposition}
\label{VE}
Assume that the stored energy function $\bs{\Elaw W}$ of the elastic material constituting the  body is of class $\Ccal^1$. 
Let the sets ${\Xi}$, ${\nt\Xi}[\ph]$ and ${\nh\Xi}[\ph]$ of admissible displacement fields associated with a deformation $\ph\in \Phi$ of an elastic body be defined by \eqref{Xi-tilde} and \eqref{Xi}. 

If a deformation $\ph\in\Ccal^1(M,N)$ satisfies the principle of least energy (Proposition \ref{PAL}), then it also satisfies each one of the following five equivalent variational equations: 
$$
\aligned
\int_M {\Sibs}[\ph]: {e}[\ph,{\xi}] 
&= \int_M {\fbs}[\ph]\cdot {\xi} + \int_{{\Ga_2}}{\hbs}[\ph]\cdot{\xi}
&&  \text{for all \ }{\xi}\in{\Xi},
\\
\int_M {\Tbs}[\ph]: {\na}{\xi} 
& = \int_M {\fbs}[\ph]\cdot {\xi} + \int_{{\Ga_2}}{\hbs}[\ph]\cdot{\xi}
&&  \text{for all \ }{\xi}\in{\Xi},
\\
\int_M {\nt\Tbs}[\ph]: {\nt\na}{\nt\xi} 
& = \int_M {\nt\fbs}[\ph]\cdot {\nt\xi} +\int_{{\Ga_2}}{\nt\hbs}[\ph]\cdot{\nt\xi}
&& \text{for all \ } {\nt\xi}\in{\nt\Xi}[\ph],
\\
\int_{\ph(M)} {\nh\Tbs}[\ph]: {\nh\na}{\nh\xi} 
& = \int_{\ph(M)} {\nh\fbs}[\ph]\cdot {\nh\xi}  +\int_{\ph({\Ga_2})}{\nh\hbs}[\ph]\cdot{\nh\xi}
&&  \text{for all \ } {\nh\xi}\in {\nh\Xi}[\ph],
\\
\int_{\ph(M)} {\nh\Sibs}[\ph]: {\nh e}[{\nh\xi}]
& = \int_{\ph(M)} {\nh\fbs}[\ph]\cdot {\nh\xi}  +\int_{\ph({\Ga_2})}{\nh\hbs}[\ph]\cdot{\nh\xi}
&&  \text{for all \ } {\nh\xi}\in {\nh\Xi}[\ph]. 
\endaligned
$$
In these equations, $\cdot$ and $:$ denote the contraction of one index and of two indices, respectively; in particular, 
$$
\aligned
{\Tbs}[\ph]: {\na}{\xi} 
& :=({T}[\ph]: {\na}{\xi})\otimes\om[\ph], 
&& {\fbs}[\ph]\cdot {\xi}:=({f}[\ph]\cdot {\xi})\,{\om}[\ph],
\\
{\nh\Tbs}[\ph]: {\nh\na}{\nh\xi}
& :=({\nh T}[\ph]: {\nh\na}{\nh\xi})\otimes{\nh\om},
&& {\nh\fbs}[\ph]\cdot {\nh\xi} :=({\nh f}[\ph]\cdot{\nh\xi})\,{\nh\om},
\endaligned
$$
whenever ${\Tbs}[\ph]={T}[\ph]\otimes{\om}[\ph]$, ${\fbs}[\ph]={f}[\ph]\otimes{\om}[\ph]$, ${\nh\Tbs}[\ph]={\nh T}[\ph]\otimes{\nh\om}$, and ${\nh\fbs}[\ph]={\nh f}[\ph]\otimes{\nh\om}$.
\end{e-proposition}

\begin{proof}
The right-hand sides appearing in the above variational equations are equal when the vector fields ${\xi}$, ${\nt\xi}$ and ${\nh\xi}$ are related by 
$$
{\nt\xi}=(\ph_*{\xi})\circ\ph 
\text{ \ and \ }
{\nh\xi}=\ph_*{\xi}, 
$$
since they all define the same scalar, ${P}'[\ph]{\nt\xi}\in\Rbb$, representing the work of the applied forces; cf. Section \ref{sect5}. Likewise, the left-hand sides appearing in the same equations are equal for the same vector fields, since 
$$
{\Sibs}[\ph]: e[\ph,{\xi}] 
={\Tbs}[\ph]: {\na}{\xi}
={\nt\Tbs}[\ph]: {\nt\na}{\nt\xi}
=\ph^*({\nh\Tbs}[\ph]: {\nh\na}{\nh\xi})
=\ph^*({\nh\Sibs}[\ph]: {\nh e}[\ph,\xi])\ ;
$$
cf. Proposition \ref{propT}. Therefore, it suffices to prove the first equation. 

Let $\ph\in\Ccal^1(M,N)$ be a solution to the minimization problem \eqref{MP}. Given any admissible displacement field ${\xi}\in{\Xi}$, let ${\nt\xi}$ and ${\nh\xi}$ be defined as above, and let ${\nh\xi}\in\Ccal^1(TN)$ also denote any extension to $N$ of the vector field ${\nh\xi}=\ph_*{\xi}\in\Ccal^1(TN|_{\ph(M)})$.  Let $\ga_{{\nh\xi}}$ denote the flow of ${\nh\xi}$ (see Section \ref{sect2}) and define the one-parameter family of deformations 
$$
\psi(\cdot,t):=\ga_{{\nh\xi}}(\cdot,t)\circ\ph, \quad t\in (-\ep,\ep).
$$

It is clear that there exists $\ep>0$ such that $\psi(\cdot,t)\in \Phi$ for all $t\in (-\ep,\ep)$. Hence 
$$
J[\ph]\leq J[\psi(\cdot,t)] \text{ \ for all } t\in(-\ep,\ep),
$$
which implies in particular that 
$$ 
\Big[\frac{d}{dt}J[\psi(\cdot,t)]\Big]_{t=0}= 0,
$$
or equivalently, that 
$$
\aligned
\Big[\frac{d}{dt}I[\psi(\cdot,t])\Big]_{t=0}
=\Big[\frac{d}{dt}{P}[\psi(\cdot,t)]\Big]_{t=0}
={P}'[\ph]{\nt\xi}.
\endaligned
$$
It remains to compute the first term of this relation.

Using the Lebesgue dominated convergence theorem, the chain rule, and the relations $\Wbs[\ph]=\bs{\Elaw W}(\cdot , E[\z\ph,\ph])$ and $\frac{\pa \bs{\Elaw W}}{\pa E}(x,E)={\Elaw{\Sibs}}(x,E)$ (cf. Section \ref{sect4}, relations \eqref{w} and \eqref{cl2}), we deduce that, on the one hand, 
$$
\aligned
\Big[\frac{d}{dt}I[\psi(\cdot,t])\Big]_{t=0}
&=
\int_M \Big[\frac{d}{dt}\bs{\Elaw W}(\cdot , E[\z\ph,\psi(\cdot,t)])\Big]_{t=0} \\
& =
\int_M {\Elaw{\Sibs}}(\cdot, E[\z\ph,\ph]):\Big[\frac{d}{dt}E[\z\ph,\psi(\cdot,t)]\Big]_{t=0}.
\endaligned
$$
On the other hand, we established in the proof of Theorem \ref{e} that 
$$
e[\ph,{\xi}] = \Big[\frac{d}{dt}E[\ph,\psi(\cdot,t)]\Big]_{t=0},
$$
which implies in turn that 
$$
e[\ph,{\xi}] = \Big[\frac{d}{dt}E[{\z\ph},\psi(\cdot,t)]\Big]_{t=0}.
$$
Besides, $\Sibs[\ph]={\Elaw{\Sibs}}(\cdot , E[\z\ph,\ph])$; cf. Definition \ref{defT}. Therefore, the above relations imply that
$$
\Big[\frac{d}{dt}I[\psi(\cdot,t])\Big]_{t=0}=\int_M \Sibs[\ph]: e[\ph,{\xi}], 
$$
and the first variational equations of Proposition \ref{VE} follow. 
\end{proof}

We are now in a position to formulate the \emph{equations of nonlinear elasticity in a Riemannian manifold}. These equations are defined as the boundary value problem satisfied by a a sufficiently regular solution $\ph$ of the variational equations that constitute the principle of virtual work (Proposition \ref{VE}). We derive below several equivalent forms of this boundary value problem, one for each stress tensor field, as does the principle of virtual work. 

The divergence operators induced by the connections ${\na}={\na}[\ph]$, ${\nt\na}={\nt\na}[\ph]$, and ${\nh\na}$, are denoted ${\divop}={\divop}[\ph]$, ${\wt\divop}={\wt\divop}[\ph]$, and ${\wh\divop}$, respectively.  We emphasize that the differential operators ${\na}$, ${\nt\na}$,  ${\divop}$, and ${\wt\divop}$, depend on the unknown deformation $\ph$, while the differential operators ${\nh\na}$ and ${\wh\divop}$ do not; see Section \ref{sect2}.

\begin{e-proposition}
\label{BVP}
A deformation $\ph\in\Ccal^2(M,N)$ satisfies the principle of virtual work (Proposition \ref{VE}) if and only if it satisfies one of the  following six equivalent boundary value problems: 
$$
\aligned
&
\left\{
\aligned
 - \divop\, {\Tbs}[\ph] &={\fbs}[\ph] && \text{in }  \textup{int}M,\\
{\Tbs}[\ph]_{\nu}  &={\hbs}[\ph] && \text{on } {\Ga_2},\\
        \ph &=\z\ph && \text{on } {\Ga_1},
\endaligned
\right.
&&
\Leftrightarrow \
\left\{
\aligned
 -\divop {T}[\ph] &={f}[\ph] && \text{in }  \textup{int}M,\\
{T}[\ph]\cdot ({\nu}[\ph]\cdot {g}[\ph])   &={h}[\ph] && \text{on } {\Ga_2},\\
        \ph &=\z\ph && \text{on } {\Ga_1},
\endaligned
\right.
\\[.3cm]
\Leftrightarrow \
&
\left\{
\aligned
 -{\wt\divop}\, {\nt\Tbs}[\ph] &={\nt\fbs}[\ph] && \text{in }  \textup{int}M,\\
{\nt\Tbs}[\ph]_{\nu}   &={\nt\hbs}[\ph] && \text{on } {\Ga_2},\\
        \ph &=\z\ph && \text{on } {\Ga_1},
\endaligned
\right.
&& 
\Leftrightarrow \
\left\{
\aligned
 -{\wt\divop} \,{\nt T}[\ph] &={\nt f}[\ph] && \text{in }  \textup{int}M,\\
{\nt T}[\ph]\cdot({\nu}[\ph]\cdot  {g}[\ph])  &={\nt h}[\ph] && \text{on } {\Ga_2},\\
        \ph &=\z\ph && \text{on } {\Ga_1},
\endaligned
\right.
\\[.3cm]
\Leftrightarrow \
&
\left\{
\aligned
 -{\wh\divop}\, {\nh\Tbs}[\ph]  &={\nh\fbs} && \text{in }   \textup{int}(\ph(M)),\\
{\nh\Tbs}[\ph]_{{\nh\nu}} &={\nh\hbs}[\ph] && \text{on } \ph({\Ga_2}),\\
\ph &=\z\ph && \text{on } {\Ga_1},
\endaligned
\right.
&&
\Leftrightarrow \
\left\{
\aligned
 -{\wh\divop}\, {\nh T}[\ph]  &={\nh f}[\ph] && \text{in }   \textup{int}(\ph(M)),\\
{\nh T}[\ph] \cdot ({\nh\nu}[\ph]\cdot{\nh g}) &={\nh h}[\ph] && \text{on } \ph({\Ga_2}),\\
\ph &=\z\ph && \text{on } {\Ga_1},
\endaligned
\right.
\endaligned
$$
where $\nu:={\nu}[\ph]$, resp. ${\nh\nu}:={\nh\nu}[\ph]$, denotes the unit outer normal vector field to the boundary of $M$, resp. of $\ph(M)$, defined by the metric tensor field $g[\ph]:=\ph^*{\nh g}$, resp. ${\nh g}$.
\end{e-proposition}

\begin{proof} 
A deformation $\ph\in\Ccal^2(M,N)$ satisfies the principle of virtual work if and only if the associated  stress tensor field ${\Tbs}[\ph]$ satisfies the variational equations
$$
\int_M \Tbs[\ph]:{\na}{\xi}
= \int_M {\fbs}[\ph] \cdot {\xi}
+ \int_{{\Ga_2}} {\hbs}[\ph]\cdot {\xi}, 
$$
for all vector fields ${\xi}\in {\Xi}$. The standard integration by parts formula on the Riemann manifold $(M, {g})$ (or Lemma \ref{IP} with  $N=M$ and $\ph=\id_M$) applied to the left-hand side integral yields the first boundary value problem. 

Likewise, since the principle of virtual work satisfied by the stress tensor field ${\nt\Tbs}[\ph]$, respectively ${\nh\Tbs}[\ph]$, is equivalent to the variational equations 
$$
\int_M {\nt\Tbs}[\ph]:{\nt\na}{\nt\xi}
= \int_M {\nt\fbs}[\ph] \cdot {\nt\xi}
+ \int_{{\Ga_2}} {\nt\hbs}[\ph] \cdot{\nt\xi}
\text{ \ for all }
{\nt\xi}\in {\nt\Xi}[\ph],
$$
respectively to the variational equations 
$$
\int_{\ph(M)} {\nh\Tbs}[\ph] : {\nh\na}{{\nh\xi}}
=\int_{\ph(M)} {\nh\fbs}[\ph] \cdot{{\nh\xi}}  \\
+ \int_{\ph({\Ga_2})} {\nh\hbs}[\ph] \cdot {{\nh\xi}}
\text{ \ for all }
{\nh\xi}\in {\nh\Xi}[\ph],
$$
Lemma \ref{IP}, respectively Lemma \ref{IP} with  $M=N$ and $\ph=\id_N$,  yields the second boundary value problem, respectively the third boundary value problem.
\end{proof}

We conclude this section by defining the \emph{elasticity tensor field} associated with an elastic material (relation \eqref{ET} below), followed by an example of constitutive law that can be used in nonlinear elasticity to explicitly define (by means of the relations \eqref{ssSED}-\eqref{ssCE-bis} below) the strain energy density appearing in Proposition \ref{PAL}, and the stress tensor fields appearing in Propositions \ref{VE}-\ref{BVP} above. The minimization problem, the variational equations, and the boundary value problem, furnished by this example are known in the literature as the equations of ``small strain nonlinear elasticity". They constitute a useful approximation of the equations of (fully) nonlinear elasticity, as well as a generalization of the frequently used  Saint Venant - Kirchhoff elastic materials (see \eqref{svk-A}-\eqref{svk-I}). 

Consider an elastic body that occupies in a reference configuration a subset $\z\ph(M)\subset N$ of the physical space and whose stored energy function $\bs{\Elaw W}$ is of class $\Ccal^2$. There is no loss of generality in replacing the stored energy density $\bs{\Elaw W}(x,E)$ by $(\bs{\Elaw W}(x,E)-\bs{\Elaw W}(x,0))$; so we henceforth assume that $\bs{\Elaw W}(x,0)=0$ for all $x\in M$. If the reference configuration is unconstrained, then ${\Elaw{\Sibs}}(x,0)=\frac{\pa \bs{\Elaw W}}{\pa E}(x,0)=0$ too, so the Taylor expansion of $\bs{\Elaw W}(x,E)$ as a function of $E\in S_{2,x}M$ in a neighborhood of $0\in S_{2,x}M$ starts with the second derivative. This justifies the following definition of the elasticity tensor field, a notion which plays a fundamental role in both nonlinear elasticity and linearized elasticity; see  Sections \ref{sect7}-\ref{sect9}.

The \emph{elasticity tensor field} of an elastic material with stored energy function $\bs{\Elaw W}={\z{\Elaw W}}\,{\z\om}$, ${\z{\Elaw W}}\in \Ccal^2(S_2M)$, is the section $\Abs={\z A}\otimes{\z\om}$, ${\z A}\in \Ccal^1(S^2M\otimes_{\sym} S^2M)$, defined at each $x\in M$ by  
\begin{equation}
\label{ET}
\Abs(x):=\frac{\pa^2 \bs{\Elaw W}}{\pa E^2}(x,0)
\ \Leftrightarrow \ 
{\z A}(x):=\frac{\pa^2 {\zz{\Elaw{W}}}}{\pa E^2}(x,0).
\end{equation}

\noindent
Note that the components of ${\z A}$ in any local chart possess the symmetries
\begin{equation}
\label{ET-sym}
{\z A}^{ij k\ell}={\z A}^{k\ell ij}={\z A}^{ji k\ell}= {\z A}^{ij \ell k}, 
\end{equation}
and that 
\begin{equation}
\label{W-Taylor}
\bs{\Elaw W}(x,E)=\frac12 (\Abs(x):E):E+o(|E|^2), \ x\in M, \ E\in S_{2,x}M,
\end{equation}
where 
\begin{equation}
\label{AEE}
(\Abs(x):E):E:=[{\z A}(x)]^{ijk\ell}E_{k\ell}E_{ij}\,{\z\om}(x).
\end{equation}

The relation \eqref{W-Taylor} justifies the following definition of \emph{small strain nonlinear elasticity}. Small strain nonlinear elasticity is an approximation of nonlinear elasticity whereby the stored energy function ${\Elaw\Wbs}(x,\cdot):S_{2,x}M\to \La^nM$ of the elastic material is replaced by its quadratic approximation, which is denoted and defined by
\begin{equation}
\label{ssW}
\bs{\Elaw W}^{\sse}(x,E):=\frac12 (\Abs(x):E):E=\Big\{\frac12 ({\z A}(x):E):E\Big\}\,{\z\om}(x)
\end{equation}
for all $x\in M$ and $E\in S_{2,x}M$. The corresponding constitutive law ${\Elaw\Sibs}^{\sse}$ is then defined at each $x\in M$ and each $E\in S_{2,x}M$ by (see \eqref{cl2})
\begin{equation}
\label{ssSi}
{\Elaw\Sibs}^{\sse}(x,E) :=\Abs(x):E= ({\z A}(x):E)\otimes \,{\z\om}(x),  
\end{equation}
where $({\z A}(x):E)^{ij}:=[{\z A}(x)]^{ijk\ell}E_{k\ell}$. Note that  ${\Elaw\Sibs}^{\sse}$ is linear with respect to $E$. 

Therefore, the deformation of an elastic body satisfies in small strain nonlinear elasticity the minimization problem of Proposition \ref{PAL} with 
\begin{equation}
\label{ssSED}
({\Wbs}[\ph])(x):=\bs{\Elaw W}^{\sse}(x,E[{\z\ph},\ph]), \ x\in M,
\end{equation}
the variational equations of Proposition \ref{VE} with
\begin{equation}
\label{ssCE}
(\Sibs[\ph])(x) :={\Elaw\Sibs}^{\sse}(x,(E[\z\ph,\ph])(x))= \Abs(x):(E[\z\ph,\ph])(x), \ x\in M,
\end{equation}
and the boundary value problem of Propositions \ref{BVP} with 
\begin{equation}
\label{ssCE-bis}
\aligned
\Tbs[\ph]={g}[\ph]\cdot \Sibs[\ph]:={g}[\ph]\cdot(\Abs:E[\z\ph,\ph]),
\endaligned
\end{equation}
where $E[\z\ph,\ph]:=\frac12({g}[\ph] -{\z g})$, ${g}[\ph]:=\ph^*{\nh g}$, and ${\z g}:=\z\ph^*{\nh g}$. 
Note that the tensor field defined by \eqref{ssCE} is quadratic in $D\ph$, while the tensor field defined by \eqref{ssCE-bis} is quartic in $D\ph$.

Examples of elastic materials obeying classical small strain nonlinear elasticity are those characterized by a \emph{Saint Venant - Kirchhoff's constitutive law}, which characterizes the simplest elastic materials that obey the axiom of frame-indifference, are homogeneous and isotropic, and whose reference configuration is a natural state; cf. Theorem 3.8-1 in \cite{ciarlet}. Its interest in practical applications is due to the fact that it depends on only two scalar parameters, the Lam\'e constants $\lambda\geq 0$ and $\mu>0$ of the elastic material constituting the body (which are determined by experiment for each elastic material), by means of the relations
\begin{equation} 
\label{svk-A}
{\z A}^{ij k\ell}:=\lambda {\z g}^{ij}{\z g}^{k\ell}+\mu({\z g}^{ik}{\z g}^{j\ell}+{\z g}^{i\ell}{\z g}^{jk}), 
\end{equation}
defining the elasticity tensor field $\Abs={\z A}\otimes{\z\om}$. 

The stored energy function of a Saint Venant - Kirchhoff material is then defined by
\begin{equation} 
\label{svk-W}
\bs{\Elaw W}^{\svk}(x,E):=\Big(\frac{\lambda}{2} ({\rm tr\,}E)^2+\mu\, |E|^2\Big)\ {\z\om}(x) 
\end{equation}
for all $x\in M$ and all $E\in S_{2,x}M$, where ${\rm tr\,}E:={\z g}^{ij}E_{ij}$, $|E|^2:={\z g}^{ik}{\z g}^{j\ell} E_{k\ell} E_{ij}$, and ${\z g}:=\z\ph^*{\nh g}$. 

Hence the strain energy of a body made of a St Venant - Kirchhoff material is given by 
\begin{equation} 
\label{svk-I}
I^{\svk}[\ph]:=\int_M \bs{\Elaw W}^{\svk}(x,(E[\z\ph, \ph])(x))=\int_M \Big(\frac{\lambda}{2} ({\rm tr\,}E[\z\ph, \ph])^2+\mu\, |E[\z\ph, \ph]|^2\Big)\ {\z\om}. 
\end{equation}

%==================================================================

\section{Linearized elasticity} 
\label{sect7}

The objective of this section is to define the equations of linearized elastostatics in a Riemannian manifold.  These equations, which take the form of a minimization problem, of variational equations, or of a boundary value problem (see Proposition \ref{EE-lin} below), are deduced from those of nonlinear elasticity (Section \ref{sect6}) by linearizing the stress tensor field $\Sibs[\ph]$ with respect to the displacement field ${\xi}:=\exp_{\z\ph}^{-1}\ph$, and by retaining only the affine part with respect to ${\xi}$ of the densities ${\fbs}[\ph]$ and ${\hbs}[\ph]$ of the applied forces. Thus the vector field ${\xi}\in\Ccal^1(TM)$ becomes the new unknown in linearized elasticity, instead of the deformation $\ph$ in nonlinear elasticity. 

Consider a body made of an elastic material whose stored energy function is $\bs{\Elaw W}={\z{\Elaw W}}\,{\z\om}$, ${\z{\Elaw W}}\in\Ccal^2(S_2M)$, and occupying a reference configuration ${\z\ph}(M)\subset N$, ${\z\ph}\in\Ccal^2(M,N)$. Assume that this reference configuration is a natural state, so that ${\Elaw\Sibs}(x,0):=\frac{\pa {\Elaw\Wbs}}{\pa E}(x,0)=0$ for all $x\in M$. Without loss in generality, assume in addition that $\bs{\Elaw W}(x,0)=0$ for all $x\in M$. 

Let ${\z\om}$, $\bs{i}_{{\z\nu}}{\z\om}$, and ${\z\nu}$, respectively denote the volume form on $M$,  the volume form on $\Ga=\pa M$, and the unit outer normal vector field to the boundary of $M$, all induced by the metric ${\z g}=g[\z\ph]:={\z\ph}^*{\nh g}$; see Section \ref{sect2}.
Let $\Abs={\z A}\otimes{\z\om}$, ${\z A}\in \Ccal^1(S^2M\otimes_{\sym} S^2M)$, denote the elasticity tensor field of an elastic material constituting the elastic body under consideration; see \eqref{ET} in Section \ref{sect6}. 
 
In linearized elasticity, the \emph{stored energy function} and the \emph{constitutive law} of the elastic material are defined, at each $x\in M$ and each $E\in S_{2,x}M$, by 
\begin{equation}
\label{linWSi}
\aligned
\bs{\Elaw W}^{\lin}(x,E)& :=\frac12 (\Abs(x):E):E=\Big\{\frac12 ({\z A}(x):E):E\Big\}\otimes {\z\om}(x),\\
{\Elaw\Sibs}^{\lin}(x,E) & :=\Abs(x):E=({\z A}(x):E)\otimes {\z\om}(x),
\endaligned
\end{equation}
respectively, where $({\z A}(x) : H):K=[{\z A}(x)]^{ijk\ell}H_{k\ell}K_{ij}$. 
Hence the \emph{constitutive equation} of a linearly elastic material is given by either one of the following relations:
\begin{equation}
\label{Tbs-lin}
\aligned
{\bs{W}^{\lin}}[{\xi}]
& :=\bs{\Elaw W}^{\lin}(x,e[\z\ph,{\xi}])= \frac12 (\Abs:e[\z\ph,{\xi}]):e[\z\ph,{\xi}], \\
\Sibs^{\lin}[{\xi}]
& :={\Elaw\Sibs}^{\lin}(x,e[\z\ph,{\xi}])= \Abs:e[\z\ph,{\xi}], \\
\Tbs^{\lin}[{\xi}]
& :={\z g} \cdot \Sibs^{\lin}[{\xi}] = {\z g} \cdot (\Abs:e[\z\ph,{\xi}]), 
\endaligned
\end{equation}
for all displacement fields ${\xi}\in \Ccal^1(TM)$. 
Note that in linearized elasticity, the stress tensor fields $\Sibs[\ph]:=\Sibs^{\lin}[{\xi}]$ and $\Tbs[\ph]:=\Tbs^{\lin}[{\xi}]$, where $\ph:=\exp_{\z\ph}{\xi}$, are linear with respect to the displacement field ${\xi}$. 

The applied body forces ${\fbs}[\ph]$ and ${\hbs}[\ph]$, which are given in nonlinear elasticity by the constitutive equations (see Section~\ref{sect5})
$$
({\fbs}[\ph])(x)=\bs{\Flaw f}(x,\ph(x),D\ph(x))
\text{ \ and \ } 
({\hbs}[\ph])(x)=\bs{\Flaw h}(x,\ph(x),D\ph(x)),
$$
are replaced in linearized elasticity by their affine part with respect to the displacement field $\xi:=\exp_{\z\ph}^{-1}\ph$, that is, by ${\fbs}^{\aff}[{\xi}]$ and ${\hbs}^{\aff}[{\xi}]$, respectively, where 
\begin{equation}
\label{fbs-lin}
\aligned
& {\fbs}^{\aff}[{\xi}] :={\fbs}[\z\ph]+{\fbs}'[\z\ph]{\xi}
\text{ \ and \ } 
{\hbs}^{\aff}[{\xi}]:={\hbs}[\z\ph]+{\hbs}'[\z\ph]{\xi},\\
& {\fbs}'[\z\ph]{\xi} 
:=\lim_{t\to 0} \frac1{t}
\left(
{\fbs}[\exp_{\z\ph}(t{\xi})]-{\fbs}[{\z\ph}]
\right)
=\bs{f_1}\cdot {\xi}+ \bs{f_2} : {\zz\na}{\xi},
\\
&{\hbs}'[\z\ph]{\xi} 
:=\lim_{t\to 0} \frac1{t}
\left(
{\hbs}[\exp_{\z\ph}(t{\xi})]-{\hbs}[{\z\ph}]
\right)
=\bs{h_1}\cdot {\xi}+\bs{h_2} : {\zz\na}{\xi}, 
\endaligned
\end{equation}
for some sections 
$\bs{f_1}\in \Ccal^0(T^0_2M\otimes\La^nM)$, 
$\bs{f_2}\in \Ccal^0(T^1_2M\otimes\La^nM)$, 
$\bs{h_1}\in \Ccal^0(T^0_2M\otimes\La^{n-1}M|_{\Ga_2})$, and 
$\bs{h_2}\in \Ccal^0(T^1_2M\otimes\La^{n-1}M|_{\Ga_2})$.

Likewise, the densities ${\Fbs}[\ph]$ and ${\Hbs}[\ph]$ appearing in the definition of the potential of the applied forces (see \eqref{pot}) are replaced in linearized elasticity by their quadratic part with respect to the displacement field $\xi:=\exp_{\z\ph}^{-1}\ph$, that is, by ${\Fbs}^{\qua}[{\xi}]$ and ${\Hbs}^{\qua}[{\xi}]$, respectively, where 
\begin{equation}
\label{Fbs-qua}
\aligned
{\Fbs}^{\qua}[{\xi}] 
& :={\Fbs}[\z\ph]+{\Fbs}'[\z\ph]{\xi}
+\frac12{\Fbs}''[\z\ph][{\xi},{\xi}],
\\
{\Hbs}^{\qua}[{\xi}]
& :={\Hbs}[\z\ph]+{\Hbs}'[\z\ph]{\xi} 
+\frac12{\Hbs}''[\z\ph][{\xi},{\xi}],
\endaligned
\end{equation}
where ${\Fbs}'[\z\ph]{\xi}$ is defined as above and 
$$
{\Fbs}''[\z\ph][{\xi},{\xi}] 
:=\lim_{t\to 0}\frac2{t^2}\left(
{\Fbs}[\exp_{\z\ph}(t{\xi})]
-{\Fbs}[{\z\ph}]
-t{\bs{F}'}[{\z\ph}]{\xi}
\right)
$$
(a similar definition holds for ${\Hbs}'[\z\ph]{\xi}$ and ${\Hbs}''[\z\ph][{\xi},{\xi}]$). 

Finally, define the tensor fields ${\z{T}^{\lin}}[{\xi}]\in \Ccal^1(T^1_1M)$, ${\z{f^{\aff}}}[{\xi}]\in \Ccal^0(T^*\!M)$ and ${\z{h^{\aff}}}[{\xi}]\in \Ccal^0(T^*\!M|_{\Ga_2})$, by letting
\begin{equation}
\label{tf-lin} 
{\Tbs^{\lin}}[{\xi}] 
={\z{T}^{\lin}}[{\xi}]\otimes{\z\om},  
\quad 
{\fbs^{\aff}}[{\xi}]
={\z{f^{\aff}}}[{\xi}]\otimes{\z\om},  
\quad 
{\hbs^{\aff}}[{\xi}]
={\z{h^{\aff}}}[{\xi}]\otimes \bs{i}_{{\z\nu}}{\z\om}.
\end{equation}

We are now in a position to derive the boundary value problem, the variational equations, and the minimization problem, of linearized elasticity from the corresponding problems of nonlinear elasticity:

\begin{e-proposition}
\label{EE-lin}
Let
$$
{\Xi} :=\{{\eta}\in \Ccal^1(TM); \ {\eta}=0 \text{ on } {\Ga_1}\}
$$ 
denote the space of all admissible displacement fields (the mappings $\ph:=\exp_{{\z\ph}}{\eta}$, $\eta\in {\Xi}$, are then the admissible deformations of the body; see Section~\ref{sect6}). 

(a)
The displacement field ${\xi}\in \Ccal^2(TM)$ satisfies in linearized elasticity the following two equivalent boundary value problems: 
\begin{equation}
\label{bvp-lin}
\left\{
\aligned
 - {\z\divop}\, {\Tbs^{\lin}}[{\xi}]&={\fbs}^{\aff}[{\xi}] && \text{in }  \textup{int}M,\\
{\Tbs^{\lin}}[{\xi}]_{\z\nu}  &={\hbs}^{\aff}[{\xi}] && \text{on } {\Ga_2},\\
        {\xi} &=0 && \text{on } {\Ga_1}.
\endaligned
\right.
\quad  \Leftrightarrow \quad 
\left\{
\aligned
 - {\z\divop}\, {\z{T}^{\lin}}[{\xi}]&={\z{f^{\aff}}}[{\xi}] && \text{in }  \textup{int}M,\\
{\z{T}^{\lin}}[{\xi}]\cdot ({\z\nu}\cdot {\z g})  &={\z{h^{\aff}}}[{\xi}]  && \text{on } {\Ga_2},\\
        {\xi} &=0 && \text{on } {\Ga_1}.
\endaligned
\right.
\end{equation}

(b)
The displacement field ${\xi}\in \Ccal^1(TM)$ satisfies in linearized elasticity the following variational equations: 
\begin{equation}
\label{VE-lin}
{\xi}\in {\Xi}
\text{ \ and \ } 
\int_M (\Abs:e[\z\ph,{\xi}]):e[\z\ph,{\eta}] 
=\int_M({\fbs}^{\aff}[{\xi}] )\cdot{\eta}
+\int_{{\Ga_2}}({\hbs}^{\aff}[{\xi}] )\cdot{\eta}
\text{ \ for all } {\eta}\in{\Xi}. 
\end{equation}

(c)
If the external forces ${\fbs}[\ph]$ and ${\hbs}[\ph]$ are conservative (cf. Section \ref{sect5}), then the displacement field ${\xi}\in \Ccal^1(TM)$ satisfies in linearized elasticity the following minimization problem: 
\begin{equation}
\label{MP-lin}
{\xi}\in {\Xi}, 
\text { and \ } 
J^{\qua}[{\xi}]\leq J^{\qua}[{\eta}] 
\text{ \ for all } {\eta}\in{\Xi}, 
\end{equation}
where 
\begin{equation}
\label{J-lin}
J^{\qua}[{\eta}] 
:= \frac12 \int_M (\Abs:e[\z\ph,{\eta}]):e[\z\ph,{\eta}] 
- \Big(
\int_M {\Fbs}^{\qua}[{\eta}]
+ \int_{\Ga_1} {\Hbs}^{\qua}[{\eta}]
\Big)
\end{equation}
denotes the total energy of the body in linearized elasticity.
\end{e-proposition}

\begin{proof}
(a) 
The boundary value problem of linearized elasticity is the affine (with respect to ${\xi}$) approximation of the following boundary value problem of nonlinear elasticity (see Proposition \ref{BVP})
\begin{equation}
\label{a-1}
\aligned
 -\divop\, {\Tbs}[\ph] &={\fbs}[\ph] && \text{in }  \textup{int}M,\\
{\Tbs}[\ph]_{\nu}  &={\hbs}[\ph] && \text{on } {\Ga_2},\\
        \ph &=\z\ph && \text{on } {\Ga_1},
\endaligned
\end{equation}
satisfied by the deformation $\ph:=\exp_{\z\ph}{\xi}$. It remains to compute this affine approximation explicitly. 

The dependence of the stress tensor field $\Tbs[\ph]$ on the vector field ${\xi}=\exp_{\z\ph}^{-1}\ph$ has been specified in Section~\ref{sect4} by means of the constitutive law of the elastic material, namely, 
$$
\aligned
({\Tbs}[\ph])(x)
& =({g}[\ph])(x)\cdot({\Sibs}[\ph])(x)\\
&= (\ph^*{\nh g})(x)\cdot {\Elaw{\Sibs}}(x, (E[\z\ph,\ph])(x)), \ x\in M. 
\endaligned
$$

Since the reference configuration $\z\ph(M)$ has been assumed to be a natural state, we have ${\Elaw{\Sibs}}(x, 0)=0$ for all $x\in M$. The definition of the elasticity tensor field $\Abs={\z A}\otimes{\z\om}$ next implies that
$$
\frac{\pa{\Elaw{\Sibs}}}{\pa E}(x,0)H=\Abs(x):H, \quad  x\in M, \ H\in S_{2,x}M. 
$$
Besides (see Section \ref{sect4}),
$$
\left[\frac{d}{dt}E[\z\ph,\exp_{\z\ph}(t{\xi})]\right]_{t=0}=e[\z\ph,{\xi}]
\text{ \ and \ }
\left[\frac{d}{dt}g[\exp_{\z\ph}(t{\xi)}]\right]_{t=0}={\z g}. 
$$

Combining the last three relations yields
$$
{\Tbs}[\ph]={g}[\ph]\cdot {\Sibs}[\ph]={\z g}\cdot (\Abs:e[\z\ph,{\xi}]) + o(\|{\xi}\|_{\Ccal^{1}(TM)}), 
$$
which implies in turn that
\begin{equation}
\label{a-2}
{\divop}\, {\Tbs}[\ph]={\z\divop}\, {\Tbs^{\lin}}[{\xi}]+ o(\|{\xi}\|_{\Ccal^{1}(TM)}),
\end{equation}
since ${\Tbs^{\lin}}[{\xi}]:={\z g}\cdot (\Abs:e[\z\ph,{\xi}])$ is linear with respect to ${\xi}$. As above, 
${\z\divop}$ denotes the divergence operator induced by the connection ${\zz\na}:={\na}[{\z\ph}]$.

The dependence of the applied force densities $\fbs[\ph]$ and  $\hbs[\ph]$ on the vector field ${\xi}=\exp_{\z\ph}^{-1}\ph$ has been specified in Section~\ref{sect5} by means of the relations 
$$
({\fbs}[\ph])(x)=\bs{\Flaw f}(x,\ph(x),D\ph(x))
\text{ \ and \ } 
({\hbs}[\ph])(x)=\bs{\Flaw h}(x,\ph(x),D\ph(x)). 
$$
Thus, using the notation \eqref{fbs-lin} above, we have 
\begin{equation}
\label{a-3}
{\fbs}[\ph]={\fbs}^{\aff}[{\xi}]+ o(\|{\xi}\|_{\Ccal^{1}(TM)})
\text{ \ and \ } 
{\hbs}[\ph]={\hbs}^{\aff}[{\xi}]+ o(\|{\xi}\|_{\Ccal^{1}(TM)}).
\end{equation}

The boundary value problems \eqref{bvp-lin} of linearized elasticity follow from the boundary value problem \eqref{a-1} of nonlinear elasticity by using the estimates \eqref{a-2} and \eqref{a-3}.

(b) 
The variational equations of linearized elasticity are the affine part with respect to ${\xi}$ of the variational equations of nonlinear elasticity (see Proposition \ref{VE})
$$
\Scal[\exp_{\z\ph}{\xi}]{\eta}=0 \text{ \ for all } {\eta}\in{\Xi},
$$
where 
$$
\Scal[\ph]{\eta}
:=\int_M {\Sibs}[\ph] : e[\ph,{\eta}] 
- \Big(\int_M{\fbs}[\ph]\cdot{\eta}
+\int_{{\Ga_2}}{\hbs}[\ph]\cdot{\eta}\Big)
$$
and ${\Sibs}[\ph]:={\Elaw{\Sibs}}(\cdot, E[\z\ph,\ph])$. 
Thus the variational equations of linearized elasticity satisfied by ${\xi}\in{\Xi}$ read:
$$
\Scal^{\aff}[{\xi}]{\eta}
:=\Scal[\z\ph]{\eta}
+\left[\frac{d}{dt}\Scal[\exp_{\z\ph}(t{\xi})]{\eta}\right]_{t=0}
=0 \text{ \ for all } {\eta}\in{\Xi}.
$$
It remains to compute $\Scal^{\aff}[{\xi}]{\eta}$ explicitly. Using that 
${\Sibs}[\exp_{{\z\ph}}{\xi}]={\bs{\Si}^{\lin}}[{\xi}]  + o(\|{\xi}\|_{\Ccal^{1}(TM)})$ 
(see part (a) of the proof), that ${\bs{\Si}^{\lin}}[{\xi}]$ is linear with respect to ${\xi}$, that 
$e[\exp_{{\z\ph}}{\xi},{\eta}]=e[{\z\ph},{\eta}]+ o(\|{\xi}\|_{\Ccal^{1}(TM)})$, and the relations \eqref{a-3}, in the above definition of $\Scal[\ph]{\eta}$,  we deduce that 
$$
\Scal^{\aff}[{\xi}]{\eta}
=\int_M (\Abs:e[\z\ph,{\xi}]):e[\z\ph,{\eta}] 
-\Big(
\int_M({\fbs}^{\aff}[{\xi}] )\cdot{\eta}
+\int_{{\Ga_2}}({\hbs}^{\aff}[{\xi}] )\cdot{\eta}
\Big).
$$
Using this expression of $\Scal^{\aff}[{\xi}]{\eta}$ in the equation $\Scal^{\aff}[{\xi}]{\eta}=0$ yields \eqref{VE-lin}.

(c) 
The minimization problem of linearized elasticity consists in minimizing the functional $J^{\qua}:{\Xi}\to\Rbb$ over the set ${\Xi}$, where $J^{\qua}$ is defined at each ${\xi}\in{\Xi}$ as the quadratic approximation with respect to the parameter $t$ of the total energy $J[\exp_{\z\ph}(t{\xi})]$, where 
$$
J[\ph]=\int_M\Wbs[\ph]-\Big(\int_M\Fbs[\ph]+\int_{{\Ga_2}}\Hbs[\ph]\Big);
$$
cf. Proposition \ref{PAL}. Thus
$$
J^{\qua}[{\xi}]:=J[\z\ph]+J'[\z\ph]\xi+\frac12 J''[\z\ph][\xi,\xi]
\text{ \ for all } {\xi}\in \Ccal^1(TM),
$$ 
where 
$$
\aligned
J'[\z\ph]\xi
& :=\lim_{t\to 0} \frac1t
\Big(
J[\exp_{\z\ph}(t{\xi}]-J[{\z\ph}]
\Big),
\\
J''[\z\ph][\xi,\xi]
& :=\lim_{t\to 0} \frac{2}{t^2}
\Big(
J[\exp_{\z\ph}(t{\xi}]-J[{\z\ph}]-tJ'[\z\ph]\xi
\Big).
\endaligned
$$

Using that $({\Wbs}[\z\ph])(x)={\Elaw\Wbs}(x, 0)=0$ for all $x\in M$, that the reference configuration is a natural state (i.e., ${\Elaw\Sibs}(x,0)=0$ for all $x\in M$), and the definition of the elasticity tensor field $\Abs$ (see \eqref{ET} in Section \ref{sect6}), we deduce that  
$$
J^{\qua}[{\xi}]
 :=\frac12\int_M (\Abs:e[\z\ph,{\xi}]):e[\z\ph,{\xi}] 
- \Big( {P}[\z\ph]{\xi}+{P}'[\z\ph]{\xi}+\frac12 {P}''[\z\ph][{\xi},{\xi}]\Big)
$$
for all ${\xi}\in\Ccal^1(TM)$, where ${P}[\ph]:=\int_M\Fbs[\ph]+\int_{{\Ga_2}}\Hbs[\ph]$ denotes the potential of the applied forces. Then the explicit expression \eqref{MP-lin} of the functional $J^{\qua}$ follows from the definition \eqref{Fbs-qua} of ${\Fbs}^{\qua}[\xi]$ and ${\Hbs}^{\qua}[\xi]$. 
\end{proof}

\begin{remark}
(a)
The variational equations of linearized elasticity of Proposition \ref{EE-lin} are extended by density to displacement fields ${\xi}\in H^1(TM)$ that vanish on $\Ga_1$ in order to prove that they possess solutions; cf. Theorem \ref{exist-lin}.

(b)
If the forces are conservative, then the three formulations of linearized elasticity are equivalent. 

(c)
Elastic materials modeled by Hooke's constitutive law correspond to linearized elasticity. Their elasticity tensor field and stored energy function are respectively defined by 
$$
\aligned
{\z A}^{ij k\ell}
& :=\lambda {\z g}^{ij}{\z g}^{k\ell}+\mu({\z g}^{ik}{\z g}^{j\ell}+{\z g}^{i\ell}{\z g}^{jk}),\\
\bs{\Elaw W}^{\hke}(x,E)
& :=\Big(\frac{\lambda}{2} ({\rm tr\,}E)^2+\mu\, |E|^2\Big)\ {\z\om}(x), 
\quad x\in M, \ E\in S_{2,x}M, 
\endaligned
$$
where $\lambda\geq 0$ and $\mu>0$ denote the Lam\'e constants of the elastic material constituting the body under consideration. The corresponding strain energy is defined at each ${\xi}\in\Ccal^1(TM)$ by 
$$
I^{\hke}[{\xi}]:=\int_M \bs{\Elaw W}^{\hke}(x,(e[\z\ph,{\xi}])(x))
= \int_M\Big(\frac{\lambda}{2} ({\rm tr\,}e[\z\ph,{\xi}])^2+\mu\, |e[\z\ph,{\xi}]|^2\Big)\ {\z\om},
$$
where 
$$
e[{\z\ph},{\xi}]
=\frac12\Lcal_{{\xi}}{\z g}
=\frac12({\zz\na}{\xi}^\flat+({\zz\na}{\xi}^\flat)^T)
$$
denotes the linearized strain tensor field; cf. Section \ref{sect3}. Note that $\bs{\Elaw W}^{\hke}=\bs{\Elaw W}^{\svk}$ (see \eqref{svk-W}), and that $I^{\hke}[{\xi}]$ is the quadratic part with respect to ${\xi}$ of $I^{\svk}[\exp_{\z\ph}{\xi}]$, where $I^{\svk}$ denotes the Saint Venant - Kirchhoff  stored energy function (see \eqref{svk-I}).
 \hfill $\square$
\end{remark}

%==================================================================

\section{Existence and regularity theorem in linearized elasticity}
\label{sect8}

In this and the next sections, $M$ denotes the closure of an open subset $\Omega$ of a smooth oriented differentiable manifold of dimension $n$, $\Omega$ being in addition bounded, connected, with a Lipschitz-continuous  boundary $\Ga:=\pa M$; see the beginning of Section \ref{sect2}. The reference deformation ${\z\ph}:M\to N$ being given such that ${\z\ph}(M)$ is a natural state of the elastic body under consideration, the Riemannian metric ${\z g}={g}[\z\ph]:=\z\ph^*{\nh g}$ makes $M$ a Riemannian manifold and ${\z\ph}:M\to N$ becomes an isometry. 

As in the previous sections, ${\zz\na}$, ${\z\divop}$, and ${\z\om}$ denote the connection, the divergence operator, and the volume form on $M$ induced by ${\z g}$. The solutions to the boundary value problem of linearized elasticity will be sought in Sobolev spaces whose elements are sections of the tangent bundle $TM$; these spaces have been defined in Sect. \ref{sect2}. 

The existence of solutions in linearized elasticity relies on the following Riemannian version of Korn's inequality, due to Chen \& Jost \cite{chen}: Assume that ${\Ga_1}\subset{\pa M}$ is a non-empty relatively open subset of the boundary of $M$. Then there exists a constant $C_K$ such that 
\begin{equation}
\label{korn}
\|{\xi}\|_{H^1(TM)}\leq C_K\|e[{\z\ph},{\xi}]\|_{L^2(S_2M)}, \ e[{\z\ph},{\xi}]:=\Lcal_{{\xi}}{{\z g}}
\end{equation}
for all ${\xi}\in H^1(TM)$ satisfying ${\xi}=0$ on ${\Ga_1}$. 

The smallest possible constant $C_K$ in the above inequality, called the Korn constant of $M$ and ${\Ga_1}\subset\pa M$,  plays an important role in both linearized elasticity and nonlinear elasticity (see assumptions \eqref{small-lin-fh} and \eqref{small-lin-f} of Theorems \ref{exist-lin} and \ref{exist}, respectively) since the smaller the Korn constant is, the larger the applied forces are in both existence theorems. To our knowledge, the dependence of the Korn constant on the metric ${\z g}$ of $M$ and on ${\Ga_1}$ is currently unknown, save a few particular cases; see, for instance, \cite{horgan,kondratev} and the references therein for estimates of the Korn constant when $(N,{\nh g})$ is an Euclidean space, or \cite{grubic} when $(N,{\nh g})$ is a Riemannian manifold.

One such particular case, relevant to our study, is when ${\Ga_1}={\pa M}$ and the metric ${\z g}$ is close to a flat metric, in the sense that its Ricci tensor field satisfies the inequality $\|{\z{\rm Ric}}\|_{L^\infty(S_2M)}\leq \frac1{C_P}$, where $C_P$ is the Poincar\'e constant of $M$, i.e., the smallest constant $C_P$ that satisfies
$$
\|{\xi}\|_{L^2(TM)}^2\leq C_P\|{\zz\na}{\xi}\|_{L^2(T^1_1M)}^2
\text{ \ for all } 
{\xi}\in H^1_0(TM).
$$
To see this, it suffices to combine the inequality
$$
\aligned
\|{\zz\na}{\xi}\|_{L^2(T^1_1M)}^2+
\|{\z\divop}\,{\xi}\|_{L^2(M)}^2
=
\frac12\|\Lcal_{{\xi}}{{\z g}}\|_{L^2(S_2M)}^2+
\int_M{\z{\rm Ric}}({\xi},{\xi})\,{\z\om}\\
\leq 
\frac12\|\Lcal_{{\xi}}{{\z g}}\|_{L^2(S_2M)}^2
+\|{\z{\rm Ric}}\|_{L^\infty(S_2M)}\|{\xi}\|_{L^2(TM)}^2,
\endaligned
$$
which holds for all ${\xi}\in H^1_0(TM)$, with the above assumption on the Ricci tensor field of ${\z g}$, to deduce that 
$$
\|{\zz\na}{\xi}\|_{L^2(T^1_1M)}^2
\leq C_K^* \|\Lcal_{{\xi}}{{\z g}}\|_{L^2(S_2M)}^2, 
$$
where $C_K^*:=\Big\{2(1- C_P \|{\z{\rm Ric}}\|_{L^\infty(S_2M)})\Big\}^{-1}$. Hence the constant $C_K=2\Big\{(1+C_P)C_K^*)\Big\}^{1/2}$ can be used in Theorems \ref{exist-lin} and \ref{exist} when  ${\Ga_1}={\pa M}$ and $\|{\z{\rm Ric}}\|_{L^\infty(S_2M)}\leq \frac1{C_P}$. Interestingly enough, particularizing these theorems to a flat metric ${\z g}$ yields existence theorems in classical elasticity with $C_K^*=1/2$, which constant is optimal.

The next theorem establishes the existence and regularity of the solution to the equations of linearized elasticity under specific assumptions on the data. Recall that in linearized elasticity the applied body and surface forces are of the form (see relations \eqref{fbs-lin})
$$
\aligned
{\fbs^{\aff}}[{\xi}]
& ={\fbs}[{\z\ph}]+{\fbs'}[{\z\ph}]{\xi}
={\fbs}[{\z\ph}]+({{\bs{f_1}}} \cdot {\xi}+ {{\bs{f_2}}} : {\zz\na}{\xi}),
\\
{\hbs^{\aff}}[{\xi}]
& ={\hbs}[{\z\ph}]+{\hbs'}[{\z\ph}]{\xi}
 =  {\hbs}[{\z\ph}] + ({{\bs{h_1}}} \cdot {\xi}+ {{\bs{h_2}}} : {\zz\na}{\xi}),
\endaligned
$$
and that (see \eqref{AEE})
$$
({\z A}(x):H):H:=[{\z A}(x)]^{ijk\ell}H_{k\ell}H_{ij}, \ x\in M.
$$

\begin{theorem}
\label{exist-lin}
Assume that ${\Ga_1}\subset{\pa M}$ is a non-empty relatively open subset of the boundary of $M$, that the elasticity tensor field $\Abs={\z A}\otimes{\z\om}$, ${\z A}\in L^\infty(T^4_0M)$, is uniformly positive-definite, that is,  there exists a constant $C_{\!{\z A}}>0$ such that
\begin{equation}
\label{A-ell}
({\z A}(x):H(x)):H(x) \geq C_{\!{\z A}}\,  |H(x)|^2, \text{ \ where } |H(x)|^2:={\z g}(x)(H(x),H(x)), 
\end{equation}
for almost all $x\in M$ and all $H(x)\in S_{2,x}M$, and that the applied body and surface forces satisfy the smallness assumption 
\begin{equation}
\label{small-lin-fh}
\|{\fbs'}[{\z\ph}]]\|_{\Lcal(H^1(TM),L^2(T^*\!M\otimes \La^nM))} +
\|{\hbs'}[{\z\ph}]\|_{\Lcal(H^1(TM),L^2(T^*\!M\otimes \La^{n-1}M)|_{\Ga_2})}
\leq {C_{\!{\z A}}}/{C_K},
\end{equation}
where $C_K$ denotes the constant appearing in Korn's inequality \eqref{korn}.

(a) If ${\fbs}[\z\ph]\in L^2(T^*\!M\otimes \La^nM)$ and ${\hbs}[\z\ph]\in L^2(T^*\!M\otimes \La^{n-1}M|_{\Ga_2})$, there exists a unique vector field ${\xi}\in H^1(TM)$, ${\xi}=0$ on $\Ga_1$, such that 
\begin{equation}
\label{var}
\int_M (\Abs:e[\z\ph,{\xi}]):e[\z\ph,{\eta}] 
=\int_M{\fbs}^{\aff}[{\xi}]\cdot{\eta}
+\int_{{\Ga_2}}{\hbs}^{\aff}[{\xi}]\cdot{\eta}
\end{equation}
for all ${\eta}\in H^1(TM)$, ${\eta}=0$ on $\Ga_1$.

(b) 
Assume in addition that $\Ga_1=\pa M$ and, for some integer $m\geq 0$ and  $1<p<\infty$, the boundary of $M$ is of class $\Ccal^{m+2}$, ${\z\ph}\in \Ccal^{m+2}(M,N)$, ${\Abs}\in \Ccal^{m+1}(T^4_0M\otimes \La^nM)$, ${{\bs{f_1}}}\in \Ccal^m(T^0_2M\otimes \La^nM)$, ${{\bs{f_2}}}\in \Ccal^m(T^1_2M\otimes \La^nM)$, and ${\fbs}[{\z\ph}]\in W^{m,p}(T^*\!M\otimes \La^nM)$. Then ${\xi}\in W^{m+2,p}(TM)$ and satisfies the boundary value problem 
\begin{equation}
\label{sedp}
\aligned
-{\z\divop}\, ({\Tbs^{\lin}}[{\xi}]) & = {\fbs^{\aff}}[{\xi}] && \text{in}\ M,\\
{\xi} & =0 &&\text{on}\ \pa M.
\endaligned
\end{equation}
Furthermore, the mapping ${\bs{\Acal}}^{\lin}: W^{m+2,p}(TM) \to W^{m,p}(T^*\!M\otimes\La^nM)$ defined by  
\begin{equation}
\label{A-lin}
{\bs{\Acal}}^{\lin}[{\eta}]:={\z\divop}\, {\Tbs}^{\lin}[{\eta}]+{\fbs}'[\z\ph]{\eta} \text{\quad  for all } {\eta}\in W^{m+2,p}(TM), 
\end{equation}
is linear, bijective, continuous, and its inverse $({\bs{\Acal}}^{\lin})^{-1}$ is also linear and continuous. 
\end{theorem}

\begin{proof}
(a) Korn's inequality, the uniform positive-definiteness of the elasticity tensor field ${\Abs}$, and the smallness of the linear part of the applied forces (see \eqref{korn}, \eqref{A-ell}, and \eqref{small-lin-fh}), together imply by means of Lax-Milgram theorem that the variational equations of linearized elasticity \eqref{var} possess a unique solution ${\xi}$ in the space $\{{\xi}\in H^1(TM);\ {\xi}=0 \text{ on } \Ga_1\}$. 

(b) It is clear that the solution of \eqref{var} is a weak solution to the boundary value problem \eqref{sedp}. 
Since the latter is locally (in any local chart) an elliptic system of linear partial differential equations,  the regularity assumptions on $\Abs$ and $\fbs^{\aff}$ and the theory of elliptic systems of partial differential equations imply that this solution is locally of class $W^{m+2,p}$; see, for instance, \cite{ADN,geymonat,necas} and the proof of Theorem 6.3-6 in \cite{ciarlet}. Furthermore, the regularity of the boundary of $M$ together with the assumption that $\Ga_1=\pa M$ imply that ${\xi}\in W^{m+2,p}(TM)$.

The mapping ${\bs{\Acal}}^{\lin}$ defined in the theorem is clearly linear and continuous. It is injective, since ${\bs{\Acal}}^{\lin}[{\xi}]=0$ with ${\xi}\in W^{m+2,p}(TM)$ implies that ${\xi}$ satisfies the variational equations \eqref{var}, hence ${\xi}=0$ by the uniqueness part of (a). It is also surjective since, given any ${\bs{f_0}}\in W^{m,p}(T^*\!M\otimes\La^nM)$, there exists (by part (a) of the theorem) a vector field ${\xi}\in H^1_0(TM)$ such that 
$$
\int_M (\Abs:e[\z\ph,{\xi}]):e[\z\ph,{\eta}] =\int_M{\bs{f_0}}\cdot{\eta} \text{ \  \ for all } {\eta}\in H^1_0(TM),
$$ 
and ${\xi}\in W^{m+2,p}(TM)$ by the regularity result established above. That the inverse of ${\bs{\Acal}}^{\lin}$ is also linear and continuous follows from the open mapping theorem. 
\end{proof}

\begin{remark}
\label{regularity}
The regularity assumption ${\Abs}\in \Ccal^{m+1}(T^4_0M\otimes \La^nM)$ can be replaced in Theorem \ref{exist-lin}(b) by the weaker regularity ${\Abs}\in W^{m+1,p}(T^4_0M\otimes \La^nM)$, $(m+1)p>n:=\dim M$, by using improved regularity theorems for elliptic systems of partial differential equations; cf. \cite{simpson}. 
\end{remark}

%==================================================================

\section{Existence theorem in nonlinear elasticity}
\label{sect9}

We show in this section that the boundary value problem of nonlinear elasticity (see Proposition \ref{BVP}) possesses at least a solution in an appropriate Sobolev space if $\Ga_2=\emptyset$ and the applied body forces are sufficiently small in a sense specified below. The assumption that $\Ga_2=\emptyset$ means that the boundary value problem is of pure Dirichlet type, that is, the boundary condition $\ph=\z\ph$ is imposed on the whole boundary $\Ga_1=\Ga$ of the manifold $M$.  Thus our objective is to prove the existence of a deformation $\ph:M\to N$ that satisfies the system (see Proposition \eqref{BVP}):
\begin{equation}
\label{bvp}
\aligned
 - \divop {\Tbs}[\ph] &={\fbs}[\ph] && \text{in }  \textup{int}M,\\
        \ph &=\z\ph && \text{on } {\Ga},
\endaligned
\end{equation}
where 
\begin{equation}
\label{cl}
\aligned
({\Tbs}[\ph])(x)
&:={\Flaw{\Tbs}}(x,\ph(x), D\ph(x)) 
,\ x\in M,\\
({\fbs}[\ph])(x)
&:=\bs{\Flaw f}(x,\ph(x), D\ph(x))  
, \ x\in M,
\endaligned
\end{equation}
the functions ${\Flaw{\Tbs}}$ and $\bs{\Flaw f}$ being the constitutive laws of the elastic material and of the applied forces, respectively (see Sections \ref{sect4} and \ref{sect5}). Recall that the divergence operator $\divop=\divop[\ph]$ depends itself on the unknown $\ph$ (since it is induced by the metric ${g}={g}[\ph]:=\ph^*{\nh g}$; cf. Section \ref{sect2}) and that $\om[\ph]:=\ph^*{\nh\om}$ denotes the volume form induced by the metric ${g}[\ph]$. 
 
The idea is to seek a solution of the form $\ph:=\exp_{\z\ph}{\xi}$, where the reference deformation ${\z\ph}:M\to N$ corresponds to a natural state ${\z\ph}(M)$ of the body, and ${\xi}:M\to TM$ is a sufficiently regular vector field that belongs to the set
$$
\Ccal^0_{{\z\ph}}(TM)
:=\{{\xi}\in \Ccal^0(TM); \ \|{\z\ph}_*{\xi}\|_{\Ccal^0(TN|_{{\z\ph}(M)})}<{\nh\de}({\z\ph}(M))\},
$$
where ${\nh\de}({\z\ph}(M))$ denotes the injectivity radius of the compact subset ${\z\ph}(M)$ of $N$; see \eqref{C0-phi} in Section \ref{sect3}. It is then clear that the deformation $\ph:=\exp_{\z\ph}{\xi}$, where ${\xi}\in \Ccal^1(TM)\cap \Ccal^0_{{\z\ph}}(TM)$, satisfies the boundary value problem \eqref{bvp} if and only if the displacement field ${\xi}$ satisfies the boundary value problem  
\begin{equation}
\label{bvp-xi}
\aligned
 - \divop {\Tbs}[\exp_{{\z\ph}}{\xi}]
 &={\fbs}[\exp_{{\z\ph}}{\xi}] 
 && \text{in }  \textup{int}M,
 \\
{\xi} 
&=0 
&& \text{on } {\Ga}.
\endaligned
\end{equation}
 
\begin{remark}
The divergence operator appearing in \eqref{bvp-xi} depends itself on the unknown ${\xi}$, since it is defined in terms of the connection ${\na}=\na[\ph]$ induced by the metric ${g}={g}[\ph]:=\ph^*{\nh g}$, $\ph:=\exp_{\z\ph}{\xi}$. 
\end{remark}

Given any vector field ${\xi}\in\Ccal^1(TM)\cap \Ccal^0_{{\z\ph}}(TM)$, let 
\begin{equation}
\label{Acal}
{\bs{\Acal}}[{\xi}]
:=\divop ({\Tbs}[\exp_{{\z\ph}}{\xi}]) +{\fbs}[\exp_{{\z\ph}}{\xi}].
\end{equation}
Proving an existence theorem to the boundary value problem \eqref{bvp-xi} amounts to proving the existence of a solution to the equation ${\bs{\Acal}}[{\xi}]=0$ in an appropriate space of vector fields ${\xi}:M\to TM$ satisfying the boundary condition ${\xi}=0$ on $\Ga$. This could be done by using Newton's method, which finds a zero of ${\bs{\Acal}}$ as the limit of the sequence defined by
$$
{\xi}_1:=0 \text{ and } {\xi}_{k+1}
:={\xi}_k-{\bs{\Acal}}'[{\xi}_k]^{-1}{\bs{\Acal}}[{\xi}_k], \ k\geq 1.
$$
This sequence converges to a zero of ${\Acal}$ under the assumptions of Newton-Kantorovich theorem (see, for instance, \cite{NK}) on the mapping ${\bs{\Acal}}$, which turn out to be stronger than those of Theorem \ref{exist} below, which uses a variant of Newton's method, where a zero of ${\bs{\Acal}}$ is found as the limit of the sequence defined by
$$
{\xi}_1:=0 \text{ and } {\xi}_{k+1}
:={\xi}_k-{\bs{\Acal}}'[0]^{-1}{\bs{\Acal}}[{\xi}_k], \ k\geq 1. 
$$

The key to applying Newton's method is to find function spaces $X$ and $\bs{Y}$ such that the mapping ${\bs{\Acal}}:U\subset X\to \bs{Y}$ be differentiable in a neighborhood $U$ of ${\xi}=0\in X$. The definition \eqref{Acal} of ${\bs{\Acal}}$ can be recast in the equivalent form 
\begin{equation}
\label{Acal2}
{\bs{\Acal}}[{\xi}]
:=\divop (({\Tbs}\circ\exp_{{\z\ph}})[{\xi}]) 
+({\fbs}\circ\exp_{{\z\ph}})[{\xi}], 
\end{equation}
where the mappings $({\Tbs}\circ\exp_{{\z\ph}})$ and $({\fbs}\circ\exp_{{\z\ph}})$ are defined by the constitutive equations 
\begin{equation}
\label{cl+}
\aligned
(({\Tbs}\circ\exp_{{\z\ph}})[{\xi}])(x)
={\Dlaw{\Tbs}}(x,{\xi}(x), {\zz\na}{\xi}(x))
:={\Flaw{\Tbs}}(x,\ph(x), D\ph(x)),\ x\in M,\\
(({\fbs}\circ\exp_{{\z\ph}})[{\xi}])(x)
={\Dlaw{\fbs}}(x,{\xi}(x), {\zz\na}{\xi}(x))
:={\Flaw{\fbs}}(x,\ph(x), D\ph(x)),\ x\in M,
\endaligned
\end{equation} 
for all vector fields ${\xi}\in\Ccal^1(TM)\cap \Ccal^0_{{\z\ph}}(TM)$, where $\ph=\exp_{\z\ph}{\xi}$ and ${\Flaw{\Tbs}}$ and ${\Flaw{\fbs}}$ are the mappings appearing in \eqref{cl}.

Relations \eqref{cl+} show that $({\Tbs}\circ\exp_{{\z\ph}})$ and $({\fbs}\circ\exp_{{\z\ph}})$ are Nemytskii (or substitution) operators. It is well known that such operators are not differentiable between Lebesgue spaces unless they are linear, essentially because these spaces are not Banach algebras. Therefore $\xi$ must belong to a space $X$ with sufficient regularity, so that the nonlinearity of  ${\Dlaw{\Tbs}}$ and ${\Dlaw{\fbs}}$ with respect to $({\xi}(x),{\zz\na}{\xi}(x))$ be compatible with the desired differentiability of ${\bs{\Acal}}$. Since we also want $\xi$ to belong to an appropriate Sobolev space (so that we could use the theory of elliptic systems of partial differential equations), we set (see Section \ref{sect2} for the definition of the Sobolev spaces and norms used below)
\begin{equation}
\label{X}
X:=W^{m+2,p}(TM)\cap W^{1,p}_0(TM),
\end{equation}
for some $m\in\Nbb$ and $1<p<\infty$ satisfying $(m + 1)p>n$. Note that the space $X$ endowed with the norm $\|\cdot\|_X:=\|\cdot\|_{m+2,p}$ is a Banach space, and that the condition $(m + 1)p>n$ is needed to ensure that the Sobolev space $W^{m+1,p}(T^1_1M)$, to which ${\zz\na}{\xi}$ belongs, is a Banach algebra.  It also implies that $X\subset \Ccal^1(TM)$, so the deformation $\ph=\exp_{\z\ph}\xi$ induced by a vector field $\xi\in X\cap \Ccal^0_{{\z\ph}}(TM)$ is at least of class $\Ccal^1$; hence the results of Section \ref{sect7} about modeling nonlinear elasticity hold for ${\xi}\in X\cap \Ccal^0_{{\z\ph}}(TM)$. 

Define 
\begin{equation}
\label{U}
U={B_X(\de)}:=\{{\xi}\in X; \|\xi\|_X<\de\}\subset X
\end{equation}
as an open ball in $X$ centered at the origin over which the exponential map $\ph=\exp_{\z\ph}{\xi}$ is well-defined (which is the case if the radius $\de$ is sufficiently small). For instance, it suffices to set 
\begin{equation}
\label{delta}
\de=\de({\z\ph,m,p})
:=\frac{{\nh\de}({\z\ph}(M))}{{C_{\!S}(m+2,p)}\|D{\z\ph}\|_{\Ccal^0(T^*\!M\otimes{\z\ph^*}TN)}},  
\end{equation}
where ${C_{\!S}}(m+2,p)$ denotes the norm of the Sobolev embedding $W^{m + 2,p}(TM)\subset \Ccal^0(TM)$. To see this, note that  
$
\|{\z\ph}_*{\xi}\|_{\Ccal^0(TN|_{{\z\ph}(M)})}
=\sup_{x\in M} |D{\z\ph}(x)\cdot {\xi}(x)|
\leq \|D{\z\ph}\|_{\Ccal^0(T^*\!M\otimes{\z\ph^*}TN)} {C_{\!S}}(m+2,p)\|{\xi}\|_X
<{\nh\de}({\z\ph}(M))
$ 
for all ${\xi}\in {B_X(\de)}$. 

We assume that the reference configuration ${\z\ph}(M)\subset N$ of the elastic body under consideration is a natural state, and that the reference deformation, the constitutive laws of the elastic material constituting the body, and the applied body forces defined by \eqref{cl+}, satisfy the following regularity assumptions:
\begin{equation}
\label{C-reg}
\aligned
{\z\ph} 
& \in \Ccal^{m + 2}(M, N), 
\\
{\Dlaw{\Tbs}}  
& \in \Ccal^{m+1}(M\times TM\times T^1_1M, \, T^1_1M\otimes \La^nM),  
\\
({\Dlaw{\fbs}}-{\fbs}[{\z\ph}] ) 
& \in \Ccal^{m}(M\times TM\times T^1_1M, \, T^*\!M\otimes \La^nM),
\endaligned
\end{equation}
and 
\begin{equation}
\label{W-reg}
{\fbs}[{\z\ph}]    
\in W^{m,p}(T^*\!M\otimes \La^nM), 
\end{equation}
for some $m\in\Nbb$ and $p\in(1,\infty)$ satisfying $(m + 1)p>n$. 
Under these assumptions, standard arguments about composite mappings and the fact that $W^{m+1,p}(M)$ is an algebra together imply that the mappings
$$
\aligned
({\Tbs}\circ\exp_{{\z\ph}})
: {\xi}\in {B_X(\de)} 
& \to {\Tbs}[\exp_{{\z\ph}}{\xi}]
\in W^{m+1,p}(T^1_1M\otimes \La^nM),\\
({\fbs}\circ\exp_{{\z\ph}})
: {\xi}\in {B_X(\de)} 
& \to {\fbs}[\exp_{{\z\ph}}{\xi}]
\in W^{m,p}(T^*\!M\otimes \La^nM),
\endaligned
$$
are of class $\Ccal^1$ over the open subset ${B_X(\de)}$ of the Banach space $X$. Since 
$
{\bs{\Acal}}[{\xi}]
= \divop {\Tbs}[\exp_{{\z\ph}}{\xi}] 
+ {\fbs}[\exp_{{\z\ph}}{\xi}]
$ 
for all ${\xi}\in {B_X(\de)}$, this in turn implies that ${\bs{\Acal}}\in \Ccal^1({B_X(\de)},{\bs{Y}})$, where the space 
\begin{equation}
\label{Y}
{\bs{Y}}:=W^{m,p}(T^*\!M\otimes \La^nM))
\end{equation}
is endowed with its usual norm $\|\cdot\|_{\bs{Y}}:=\|\cdot\|_{m,p}$. 

Finally, we assume that the elasticity tensor field ${\Abs}={\z A}\times{\z\om}$, where ${\z\om}:={\z\ph^*}{\nh\om}$, of the elastic material constituting the body under consideration is uniformly positive-definite, that is, there exists a constant $C_{\!{\z A}}>0$ such that
\begin{equation}
\label{A-ell-bis}
({\z A}(x):H(x)):H(x) 
\geq C_{\!{\z A}}\,  |H(x)|^2, 
\text{ \ where } 
|H(x)|^2:={\z g}(x)(H(x),H(x)), 
\end{equation}
for almost all $x\in M$ and all $H(x)\in S_{2,x}M$ (the same condition as in linearized elasticity; cf. \eqref{A-ell}). Note that the elasticity tensor field 
$\displaystyle 
{\Abs}:=\frac{\pa{\Elaw{\Sibs}}}{\pa E}
$  
is defined in terms of the constitutive law ${\Flaw{\Tbs}}$ appearing in \eqref{cl} by means of the relation 
$$
{\Flaw{\Tbs}}(x,\ph(x), D\ph(x))
={g}(x)\cdot {\Elaw{\Sibs}}(x,(E[{\z\ph},\ph])(x)), 
\ x\in M, 
$$
where 
$
E[{\z\ph},\ph]
:=({g}[\ph]-{g}[{\z\ph}])/2
=(\ph^*{\nh g}-{\z\ph^*}{\nh g})/2
$; 
cf. Section \ref{sect7}. 

We are now in a position to establish the existence of a solution to the Dirichlet boundary value problem of nonlinear elasticity if the density ${\fbs}[{\z\ph}] $, resp. the first variation ${{\fbs}'}[{\z\ph}]$, of the applied body forces acting on, resp. in a neighborhood of, the reference configuration ${\z\ph}(M)$ are both small enough in appropriate norms.

\begin{theorem}
\label{exist}
Suppose that the reference deformation ${\z\ph}$ and the constitutive laws ${\Dlaw{\Tbs}}$ and ${\Dlaw{\fbs}}$ satisfy the regularity assumptions \eqref{C-reg} and \eqref{W-reg}, that the elasticity tensor field ${\Abs}={\z A}\otimes {\z\om}$ satisfy the inequality \eqref{A-ell-bis}, and that the manifold $M$ possesses a non-empty boundary of class $\Ccal^{m+2}$. 
Let ${\bs{\Acal}}:{B_X(\de)}\subset X\to {\bs{Y}}$ be the (possibly nonlinear) mapping defined by \eqref{Acal}-\eqref{delta} and \eqref{Y}. 

(a)  Assume that the first variation at ${\z\ph}$ of the density  
of the applied body forces satisfies the smallness assumption:
\begin{equation}
\label{small-lin-f}
\|{{\fbs}'}[{\z\ph}]\|_{\Lcal(H^1(TM),\, L^2(T^*\!M\otimes \La^nM))} 
\leq {C_{\!{\z A}}}/{C_K},
\end{equation}
where $C_K$ denotes the constant appearing in Korn's inequality \eqref{korn} with $\Ga_1=\Ga$.

Then the mapping ${\bs{\Acal}}$ is differentiable over the open ball ${B_X(\de)}$ of $X$, ${\bs{\Acal}}'[0] \in \Lcal(X,{\bs{Y}})$ is bijective, and ${\bs{\Acal}}'[0]^{-1}\in \Lcal({\bs{Y}},X)$. Moreover, ${\bs{\Acal}}'[0]={{\bs{\Acal}}^{\lin}}$ is precisely the differential operator of linearized elasticity defined by \eqref{A-lin}. 

(b) Assume in addition that the density of the applied body forces acting on the reference configuration ${\z\ph}(M)$ of the body satisfies the smallness assumption:
\begin{equation}
\label{small-f[0]}
\|{\fbs}[{\z\ph}] \|_{\bs{Y}}<\ep_1:=
\sup_{0< r <\de} r\,
\Big(\|{\bs{\Acal}}'[0]^{-1}\|_{\Lcal({\bs{Y}},X)}^{-1}
-\sup_{\|{\xi}\|_X <r} \|{\bs{\Acal}}'[{\xi}]-{\bs{\Acal}}'[0]\|_{\Lcal(X,{\bs{Y}})}\Big),
\end{equation}
where $\de$ is defined by \eqref{delta}. Let $\de_1$ be any number in $(0,\de)$ for which 
\begin{equation}
\label{delta1}
\|{\fbs}[{\z\ph}] \|_{\bs{Y}}<
\de_1\Big(\|{\bs{\Acal}}'[0]^{-1}\|_{\Lcal({\bs{Y}},X)}^{-1}
-\sup_{\|{\xi}\|_X <\de_1} \|{\bs{\Acal}}'[{\xi}]-{\bs{\Acal}}'[0]\|_{\Lcal(X,{\bs{Y}})}\Big).
\end{equation}
Then the equation ${\bs{\Acal}}[{\xi}]=0$ has a unique solution ${\xi}$ in the open ball $B_X(\de_1)\subset B_X(\de)$. Moreover, the mapping $\ph:=\exp_{\z\ph}{\xi}$ satisfies the boundary value problem \eqref{bvp}-\eqref{cl}.

(c) 
Assume further that the mapping ${\z\ph}:M\to N$ is injective and orientation-preserving. Then there exists $\ep_2\in (0,\ep_1)$ such that, if $\|{\fbs}[{\z\ph}] \|_{\bs{Y}}<\ep_2$, the deformation $\ph := \exp_{{\z\ph}}{\xi}$ found in (b) is injective and orientation-preserving.
\end{theorem}

\begin{proof} 
(a) 
It is clear from the discussion preceding the theorem that ${\bs{\Acal}}\in\Ccal^1({B_X(\de)},{\bs{Y}})$. 
Let ${\xi}\in {B_X(\de)}$ and let $\ph:=\exp_{{\z\ph}}{\xi}$. We have seen in Section \ref{sect8} (relations \eqref{a-2} and \eqref{a-3}) that
$$
{\divop}\, {\Tbs}[\ph] 
+ {\fbs}[\ph] 
={\z\divop}\, {\Tbs^{\lin}}[{\xi}]
+ {\fbs}^{\aff}[{\xi}]
+ o(\|{\xi}\|_{\Ccal^{1}(TM)}), 
$$
where ${\divop}:={\divop}[\ph]$ and ${\z\divop}:={\divop}[{\z\ph}]$ denote the divergence operators induced by the connections ${\na}:={\na}[\ph]$ and ${\zz\na}:={\na}[{\z\ph}]$, respectively. 

Using the definitions of the mappings ${\fbs}^{\aff}$, ${\bs{\Acal}}^{\lin}$, and ${\bs{\Acal}}$ (see \eqref{fbs-lin}, \eqref{A-lin}, and \eqref{Acal}, respectively) in this relation, we deduce that
$$
{\bs{\Acal}}[{\xi}]={\fbs}[{\z\ph}] + {\bs{\Acal}}^{\lin}[{\xi}] + o(\|{\xi}\|_{\Ccal^{1}(TM)}).
$$
This relation shows that ${\bs{\Acal}}'[0]={{\bs{\Acal}}^{\lin}}$. Since ${\bs{\Acal}}^{\lin}$ is precisely the differential operator appearing in Theorem \ref{exist-lin}(b), and since assumption \eqref{small-lin-f} of Theorem \ref{exist} is the same as assumption \eqref{small-lin-fh} of Theorem \ref{exist-lin} when $\Ga_2=\emptyset$, Theorem \ref{exist-lin}(b) implies that ${\bs{\Acal}}'[0]\in \Lcal(X,{\bs{Y}})$ is bijective and ${\bs{\Acal}}'[0]^{-1}\in \Lcal({\bs{Y}},X)$. Note in passing that this property can be used to prove an existence theorem to the equations of nonlinear elasticity by means of the local inversion theorem, instead of the Newton's method used below: see Remark \ref{rk-exist}(a) at the end of this proof. 

(b)
The idea is to prove that the relations 
\begin{equation}
\label{xi-k}
{\xi}_1:=0 \text{ and } {\xi}_{k+1}:={\xi}_k-{\bs{\Acal}}'[0]^{-1}{\bs{\Acal}}[{\xi}_k], \ k\geq 1,  
\end{equation}
define a convergent sequence in $X$, since then its limit will be a zero of ${\bs{\Acal}}$. This will be done by applying the contraction mapping theorem to the mapping $\Bcal: V\subset {B_X(\de)} \to {\bs{Y}}$, defined by
$$
\Bcal[{\xi}]:={\xi}-{\bs{\Acal}}'[0]^{-1}{\bs{\Acal}}[{\xi}].
$$
The set $V$ has to be endowed with a distance that makes $V$ a complete metric space and must be defined in such a way that $\Bcal$ be a contraction and $\Bcal[V]\subset V$ (the set $\Bcal[V]$ denotes the image of $V$ by $\Bcal$).

Since the mapping ${\bs{\Acal}}': {B_X(\de)}\to \Lcal(X,{\bs{Y}})$ is continuous, it is clear that $\ep_1>0$. Hence there exists ${\de_1}\in (0,\de)$ such that 
\begin{equation}
\label{delta1bis}
\|{\fbs}[{\z\ph}] \|_{\bs{Y}} < {\de_1}\Big(\|{\bs{\Acal}}'[0]^{-1}\|_{\Lcal({\bs{Y}},X)}^{-1}
-\sup_{\|{\xi}\|_X < {\de_1}} \|{\bs{\Acal}}'[{\xi}]-{\bs{\Acal}}'[0]\|_{\Lcal(X,{\bs{Y}})}\Big). 
\end{equation}
Note that this definition is the same as that appearing in the statement of the theorem; cf \eqref{delta1}. 
So pick such a ${\de_1}$ and define
$$
V={B_X({\de_1}]}:=\{{\xi}\in X; \ \|{\xi}\|_X\leq {\de_1}\}
$$
as the closed ball in $X$ of radius $\de_1$ centered at the origin of $X$. As a closed subset of the Banach space $(X,\|\cdot\|_X)$, the set ${B_X({\de_1}]}$ endowed with the distance induced by the norm $\|\cdot\|_X$ is a complete metric space. Besides, the mapping $\Bcal:{B_X({\de_1}]}\to X$ is well defined since ${B_X({\de_1}]}\subset {B_X(\de)}$. It remains to prove that $\Bcal$ is a contraction and that $\Bcal[{B_X({\de_1}]}]\subset {B_X({\de_1}]}$. 

Let ${\xi}$ and ${\eta}$ be two elements of ${B_X({\de_1}]}$. Then
$$
\|\Bcal[{\xi}]-\Bcal[{\eta}]\|_X
\leq \| {\bs{\Acal}}'[0]^{-1}\|_{\Lcal({\bs{Y}},X)}\|{\bs{\Acal}}[{\eta}]-{\bs{\Acal}}[{\xi}]-{\bs{\Acal}}'[0]({\xi}-{\eta}) \|_{\bs{Y}}.
$$
Applying the mean value theorem to the mapping ${\bs{\Acal}}\in \Ccal^1({B_X(\de)},{\bs{Y}})$ next implies that
$$
\|\Bcal[{\xi}]-\Bcal[{\eta}]\|_X
\leq C_{\Bcal}\|{\xi}-{\eta}\|_X,
$$
where
$$
C_{\Bcal}:=\| {\bs{\Acal}}'[0]^{-1}\|_{\Lcal({\bs{Y}},X)}
\sup_{\|{\zeta}\|< {\de_1}}\|{\bs{\Acal}}'[{\zeta}]-{\bs{\Acal}}'[0]\|_{\Lcal(X,{\bs{Y}})}.
$$
But the inequality \eqref{delta1bis} implies that 
$$
\aligned
C_{\Bcal}
& =1- \| {\bs{\Acal}}'[0]^{-1}\|_{\Lcal({\bs{Y}},X)} 
\Big(\| {\bs{\Acal}}'[0]^{-1}\|_{\Lcal({\bs{Y}},X)}^{-1}-\sup_{\|{\zeta}\|< {\de_1}}\|{\bs{\Acal}}'[{\zeta}]-{\bs{\Acal}}'[0]\|_{\Lcal(X,{\bs{Y}})}\Big)\\
& < 1- \| {\bs{\Acal}}'[0]^{-1}\|_{\Lcal({\bs{Y}},X)} \frac{\|{\fbs}[{\z\ph}] \|_{\bs{Y}}}{\de_1}
\leq 1, 
\endaligned
$$
which shows that $\Bcal$ is indeed a contraction on ${B_X({\de_1}]}$. 

Let ${\xi}$ be any element of ${B_X({\de_1}]}$. Since
$$
\|\Bcal[{\xi}]\|_X
\leq \|\Bcal[0]\|_X+\|\Bcal[{\xi}]-\Bcal[0]\|_X
\leq \|{\bs{\Acal}}'[0]^{-1}{\fbs}[{\z\ph}] \|_X+C_{\Bcal} {\de_1},
$$
using the above expression of $C_{\Bcal}$ and the inequality \eqref{delta1bis} yields
 $$
\|\Bcal[{\xi}]\|_X
\leq \| {\bs{\Acal}}'[0]^{-1}\|_{\Lcal({\bs{Y}},X)}
\Big( 
\|{\fbs}[{\z\ph}] \|_{\bs{Y}} 
+{\de_1}\sup_{\|{\zeta}\|< {\de_1}}\|{\bs{\Acal}}'[{\zeta}]
-{\bs{\Acal}}'[0]\|_{\Lcal(X,{\bs{Y}})}
\Big)
< {\de_1}, 
$$
which shows that $\Bcal[{B_X({\de_1}]}]\subset {B_X({\de_1}]}$. 

The assumptions of the contraction mapping theorem being satisfied by the mapping $\Bcal$, there exists a unique ${\xi}\in {B_X({\de_1}]}$ such that $\Bcal[{\xi}]={\xi}$, which means that ${\xi}$ satisfies the equation ${\bs{\Acal}}[{\xi}]=0$. This equation being equivalent to the boundary value problem \eqref{bvp-xi}, the deformation $\ph:=\exp_{\z\ph}{\xi}$ satisfies the boundary value problem \eqref{bvp}-\eqref{cl}.

(c) 
The contraction mapping theorem shows that the rate at which the sequence ${\xi}_k=\Bcal^k[0]$,  $k=1,2,...$, converges to the solution $\xi$ of the equation ${\bs{\Acal}}[{\xi}]=0$ is 
\begin{equation}
\label{xik-xi}
\|{\xi}_k-{\xi}\|_X\leq \frac{(C_{\Bcal})^k}{1-C_{\Bcal}}\|\Bcal[0]\|_X.
\end{equation}
In particular, for $k=0$, 
\begin{equation}
\label{f-zeta}
\|{\xi}\|_X 
\leq \frac1{1-C_{\Bcal}}\|\Bcal[0]\|
\leq \frac{\|{\bs{\Acal}}'[0]^{-1}\|_{\Lcal({\bs{Y}},X)}}{1-C_{\Bcal}}\|{\fbs}[{\z\ph}] \|_{\bs{Y}}
\leq C_{\bs{\Acal}}\|{\fbs}[{\z\ph}] \|_{\bs{Y}},
\end{equation}
where
$$
C_{\bs{\Acal}}:=\Big\{\|{\bs{\Acal}}'[0]^{-1}\|_{\Lcal({\bs{Y}},X)}^{-1}-\sup_{\|{\zeta}\|< {\de_1}}\|{\bs{\Acal}}'[{\zeta}]-{\bs{\Acal}}'[0]\|_{\Lcal(X,{\bs{Y}})}\Big\}^{-1}.
$$

The Sobolev embedding $W^{m+2,p}(TM)\subset C^1(TM)$ being continuous, the mapping 
$$
{\eta}\in {B_X({\de_1}]} \to \psi:=\exp_{\z\ph}{\eta}\in \Ccal^1(M,N) \to \det (D\psi) \in \Ccal^0(M)
$$
is also continuous. Besides $\min_{z\in M}\det (D\z\ph(z))>0$ since $\z\ph$ is orientation-preserving and $M$ is compact. It follows that there exists $0<{\de_2}\leq {\de_1}$ such that
$$
\|{\eta}\|_X<{\de_2} \Rightarrow \|\det (D\psi)-\det (D\z\ph)\|_{\Ccal^0(M)}< \min_{z\in M}\det (D\z\ph(z)),
$$
which next implies that 
\begin{equation}
\label{o-p}
\|{\eta}\|_X<{\de_2} \Rightarrow \det (D\psi(x))>0 \text{ for all } x\in M.
\end{equation}

Assume now that the applied forces satisfy $\|{\fbs}[{\z\ph}] \|_{\bs{Y}}<\ep_2:={{\de_2}}/{C_{\bs{\Acal}}}$. Then the relations \eqref{f-zeta} and \eqref{o-p} together show that the deformation $\ph:=\exp_{\z\ph}{\xi}$,  where ${\xi}\in {B_X({\de_1}]}$ denotes the solution of the equation ${\bs{\Acal}}[{\xi}]=0$, satisfies
$$
\det (D\ph(x))>0 \text{ for all } x\in M,
$$
which means that $\ph$ is orientation-preserving. 

Moreover, since $\ph=\z\ph$ on ${\pa M}$ and $\z\ph:M\to N$ is injective, the inequality $\det D\ph(x)>0$ for all $x\in M$ implies that $\ph:M\to N$ is injective; cf. \cite[Theorem 5.5-2]{ciarlet}.
\end{proof}

\begin{remark}
\label{rk-exist}
(a) 
The mapping ${\bs{\Fcal}}:{B_X(\de)}\subset X \to {\bs{Y}}$ defined by 
$$
{\bs{\Fcal}}[{\xi}]:={\bs{\Acal}}[{\xi}]-{\fbs}[{\z\ph}] 
$$
satisfies the assumptions of the local inversion theorem at the origin of $X$ if the assumption \eqref{small-lin-f} is satisfied; see part (a) of the proof of Theorem \ref{exist}. Hence there exist constants $\de_3>0$ and $\ep_3>0$ such that the equation ${\bs{\Fcal}}[{\xi}]=-{\fbs}[{\z\ph}]  $, or equivalently 
$$
{\bs{\Acal}}[{\xi}]=0,
$$
has a unique solution ${\xi}\in X$, $\|{\xi}\|_X<\de_3$, if $\|{\fbs}[{\z\ph}] \|_{\bs{Y}}<\ep_3$. Using Newton's method instead of the local inversion theorem in the proof of Theorem \ref{exist} provides (as expected) explicit estimations of the constants $\de_3$ and $\ep_3$, namely $\de_3=\de_1$ and $\ep_3=\ep_1$ (see \eqref{small-f[0]} and \eqref{delta1} for the definitions of $\ep_1$ and $\de_1$). 

(b) The proof of Theorem \ref{exist} provides an iterative procedure for numerically computing approximate solutions to the equations of nonlinear elasticity in a Riemannian manifold, as well as an error estimate: see relations \eqref{xi-k} and \eqref{xik-xi} above. Another iterative procedure, this time in classical nonlinear elasticity, can be found in \cite[Chapter 6.10]{ciarlet}; the corresponding error estimate is given in Theorem 6.13-1 of \cite{ciarlet}.

(c) Previous existence theorems for the equations of nonlinear elasticity in Euclidean spaces (see, for instance, \cite{ciarlet,IE,valent}) can be obtained from Theorem \ref{exist} by making additional assumptions on the applied forces: either ${\Dlaw{\fbs}}-{\fbs}[{\z\ph}] =0$ in the case of ``dead" forces, or ${\Dlaw{\fbs}} \in \Ccal^{m}(M\times TM\times T^1_1M, T^*\!M\otimes \La^nM)$ in the case of ``live" forces. 

(d) Theorem \ref{exist-lin} \textit{(a)} and \textit{(b)} can be generalized to mixed Dirichlet-Neumann boundary conditions provided that $\ov\Ga_1\cap\ov\Ga_2=\emptyset$, since in this case the regularity theorem for elliptic systems of partial differential equations still holds. 
\end{remark}

%==================================================================

\section{Concluding remarks}
\label{sect10}

The equations of nonlinear elastostatics  (Section \ref{sect6}), as well as those of linearized elastostatics (Section \ref{sect7}), satisfied by an elastic body subjected to applied body and surface forces, have been generalized from their classical formulation, which is restricted to the particular case where the body is immersed in the three-dimensional Euclidean space, to a new formulation, which is valid in the general case where the body is immersed in an arbitrary Riemannian manifold. This new formulation is intrinsic, i.e., it does not depend on the choice of the local charts of the Riemannian manifold, in contrast to the classical formulation, which depends on the choice of Cartesian coordinates in the three-dimensional Euclidean space. 

One application of this new formulation of the equations of nonlinear elastostatics is to model three-dimensional bodies whose geometry in a reference configuration is described by several local charts (as for instance spherical coordinates in one part of the body, and cylindrical coordinates in another part). Therefore the assumption that the body be described by a global chart, which is made in most, if not all, classical textbooks (see, for instance, \cite{ciarletG}) is not needed in our approach. Another application is to to model two-dimensional bodies whose deformations are constrained to a given surface (as for instance a cylinder deforming only longitudinally), then to numerically compute approximate solutions for the deformation of such bodies subjected to specific body and surface forces.

The main novelty of this paper is the definition of the stress tensor field of an elastic body, and of the equations satisfied by the corresponding deformation, in any Riemannian manifold. The classical Cauchy stress tensor field, first Piola-Kirchhoff stress tensor field, and the second Piola-Kirchhoff stress tensor field, are obtained from the stress tensor field defined in this paper by letting the Riemannian manifold be the three-dimensional Euclidean space, by choosing a particular volume form in the reference configuration of the body, and by using a Cartesian frame in the three-dimensional Euclidean space.

Another novelty is the definition of a new kind of stress tensor field, denoted $\Tbs[\ph]$ in Section \ref{sect4}, which is essential in proving the existence of a solution to the equations of nonlinear elasticity in a Riemannian manifold; cf.~Theorem \ref{exist}. This tensor field has not been defined in classical elasticity since the first Piola-Kirchhoff stress tensor field associated with a deformation $\ph$, which in the general case where the body is immersed into a Riemannian manifold $(N,{\nh g})$ is defined by 
$$
{\z{\nt T}}(x):=({\z{\nt T}})^i_\al(x)\frac{\pa}{\pa x^i}(x)\otimes dy^\al(\ph(x)), \ x\in M, 
$$
can be identified in the particular case where $(N,{\nh g})$ is the three-dimensional Euclidean space with the tensor field (with the notations of this paper)
$$
{T}_{\textup{P-K}}(x):=({\z{\nt T}})^i_\al(x)\frac{\pa}{\pa x^i}(x)\otimes {\nh e}^\al, \ x\in M, 
$$
by choosing $(y^\al)$ as the Cartesian coordinates of a generic point $y$ in the three-dimensional Euclidean space with respect to a given orthonormal frame $\{O,({\nh e}^\al)\}$. The advantage of this identification is that ${\nh e}^\al$ does not depend on the unknown deformation $\ph$, while $dy^\al(\ph(x))$ does. Such an  identification is obviously not possible if $(N,{\nh g})$ is an arbitrary Riemannian manifold. 

The main result of this paper is Theorem \ref{exist}, which establishes the existence of a solution to the equations of nonlinear elastostatics in a Riemannian manifold. This existence theorem at the same time generalizes several theorems of the same kind in classical elasticity, and weakens their assumptions. In particular, the applied forces considered here are more general than those in classical elasticity, the smallness assumption on these forces is explicit, and a new algorithm is provided for approaching the (exact) solution of the equations of nonlinear elastostatics. Another use of Theorem \ref{exist} in classical nonlinear elasticity is to prove an existence theorem for two-dimensional elastic bodies whose deformed configurations are contained in a given surface of the three-dimesional Euclidean space. 

Finally, the proof of Theorem \ref{exist} is to our knowledge new even in classical elasticity, since is based on the Newton's method for finding the zeroes of a nonlinear mapping, rather than on the implicit, or the inverse, function theorem.

%==================================================================

\section*{Acknowledgments} 
The authors are grateful to P.G. Ciarlet for his comments on a preliminary version of this paper. 
The authors were supported by the Agence Nationale de la Recherche through the grants ANR 2006-2--134423 and ANR SIMI-1-003-01. This paper was completed when the second author (PLF) was in residence at the Mathematical Sciences Research Institute in Berkeley, California, during the Fall Semester 2013 and was supported by the National Science Foundation under Grant No. 0932078 000.

%==================================================================

\end{document}